\documentclass[a4paper,11pt]{amsart} 
\usepackage{amsmath, amsfonts, amsxtra, amssymb, latexsym, amscd, amsthm}
\usepackage[mathscr]{eucal}
\usepackage{mathrsfs}
\usepackage{bbm}
\usepackage{enumerate}
\usepackage{pict2e}
\usepackage{tikz}
\usepackage{graphicx}
\usepackage[a4paper]{geometry}
\usepackage{color}
\usepackage[colorlinks=true,%
            citecolor=blue,%
            urlcolor=darkgray]{hyperref}  
\numberwithin{equation}{section}

\theoremstyle{plain}
\newtheorem{thm}{Theorem}[section]

\newtheorem{prp}[thm]{Proposition}
\newtheorem{cor}[thm]{Corollary}
\newtheorem{lem}[thm]{Lemma}
\theoremstyle{definition}

\newtheorem{rem}[thm]{Remark}

\newtheorem*{rem*}{Remark}

\newcommand{\dd}{\mathrm{d}}
\newcommand{\ee}{\mathrm{e}}
\newcommand{\ii}{\mathrm{i}}
\renewcommand{\Im}{\operatorname{Im}}
\renewcommand{\Re}{\operatorname{Re}}

\newcommand{\N}{\mathbb{N}}
\newcommand{\Z}{\mathbb{Z}}

\newcommand{\R}{\mathbb{R}}
\newcommand{\C}{\mathbb{C}}
\newcommand{\h}{\mathbb{H}}

\newcommand{\T}{\mathbb{T}}

\newcommand{\cN}{\mathcal{N}}
\newcommand{\cA}{\mathcal{A}}
\newcommand{\cC}{\mathcal{C}}
\newcommand{\cH}{\mathcal{H}}
\newcommand{\cW}{\mathcal{W}}

\newcommand{\cL}{\mathcal{L}}

\newcommand{\eps}{\epsilon}
\newcommand{\gm}{\gamma}

\newcommand{\ka}{\kappa}

\newcommand{\de}{\delta}
\newcommand{\be}{\beta}

\newcommand{\al}{\alpha}

\DeclareMathOperator{\Vis}{Vis}
\newcommand{\ap}[1]{\left( #1\right)}
\newcommand{\Sp}{{\mathbb S}}

\newcommand{\ED}{{\mathcal{E}_{\mathrm{D}}}}
\newcommand{\curlyF}{\mathscr{F}}

\DeclareMathOperator{\expect}{\mathbb{E}}

\DeclareMathOperator{\prob}{\mathbb{P}}

\DeclareMathOperator{\supp}{supp}

\DeclareMathOperator{\tr}{tr}
\DeclareMathOperator{\ch}{cosh}
\DeclareMathOperator{\sh}{sinh}

\DeclareSymbolFont{extraup}{U}{zavm}{m}{n}
\DeclareMathSymbol{\varheart}{\mathalpha}{extraup}{86}
\DeclareMathSymbol{\vardiamond}{\mathalpha}{extraup}{87}

\title{Absolutely continuous spectrum for quantum trees}
\author{Nalini Anantharaman, Maxime Ingremeau, Mostafa Sabri, Brian Winn}
\address{Universit\'e de Strasbourg, CNRS, IRMA UMR 7501, F-67000 Strasbourg, France.}
\email{anantharaman@math.unistra.fr}
\address{Laboratoire J.A.Dieudonn\'e, UMR CNRS-UNS 7351, Universit\'e C\^ote d'Azur, 06108 Nice, France}
\email{maxime.ingremeau@univ-cotedazur.fr}
\address{Department of Mathematics, Faculty of Science, Cairo University, Cairo 12613, Egypt.}
\email{mmsabri@sci.cu.edu.eg}
\address{Department of Mathematical Sciences, Loughborough University, Leicestershire, LE11 3TU, United Kingdom.}
\email{b.winn@lboro.ac.uk}
\usepackage{calc}
\usepackage{graphicx}

\makeatletter
\newlength{\temp@wc@width}
\newlength{\temp@wc@height}
\newcommand{\widecheck}[1]{%
  \setlength{\temp@wc@width}{\widthof{$#1$}}%
  \setlength{\temp@wc@height}{\heightof{$#1$}}%
  #1\hspace{-\temp@wc@width}%
  \raisebox{\temp@wc@height+2pt}[\heightof{$\widehat{#1}$}]%
     {\rotatebox[origin=c]{180}{\vbox to 0pt{\hbox{$\widehat{\hphantom{#1}}$}}}}%
}
\makeatother

\begin{document}

\begin{abstract}
We study the spectra of quantum trees of finite cone type. These are quantum graphs whose geometry has a certain homogeneity, and which carry a finite set of edge lengths, coupling constants and potentials on the edges. We show the spectrum consists of bands of purely absolutely continuous spectrum, along with a discrete set of eigenvalues. Afterwards, we study random perturbations of such trees, at the level of edge length and coupling, and prove the stability of pure AC spectrum, along with resolvent estimates.
\end{abstract}

\maketitle

\section{Introduction}\label{sec:mainres}

Our aim in this paper is to establish the existence of bands of purely absolutely continuous (AC) spectrum for a large family of quantum trees. One of our motivations is to provide a collection of examples relevant for the Quantum Ergodicity result proven in \cite{AISW}.

For discrete trees, the problem is quite well understood when the tree is somehow homogeneous. The adjacency matrix of the $(q+1)$-regular tree $\T_q$ has pure AC spectrum $[-2\sqrt{q},2\sqrt{q}]$ as is well-known \cite{Kesten}. If we fix a root $o\in \T_q$ and regard the tree as descending from $o$, then the subtree descending from any offspring is the same (each is a $q$-ary tree), except for the subtree at the origin (which has $(q+1)$ children). We say that $\T_q$ has two ``cone types''. It was shown in \cite{KLW2} that if $\T$ is a general tree with finitely many cone types, such that each vertex has a child of its own type, and all types arise in each progeny subtree, then the spectrum consists of bands of pure AC spectrum. This problem was revisited in \cite{AS4}, where these assumptions were relaxed to allow $\T$ to be any universal cover of a finite graph of minimal degree at least $2$. In this case however, besides the bands of AC spectrum, a finite number of eigenvalues may appear.

A natural question is whether AC spectrum survives if we add a
potential.  This is motivated by the famous Anderson model
\cite{Anderson} where random independent, identically-distributed
potentials are attached at lattice sites.  It remains a major open problem to prove such stability
for the Anderson model on the euclidean lattice $\Z^d$, $d\ge 3$ \cite{SimonProblems}. The first mathematical
proof showing the stability of pure AC spectrum was obtained in \cite{Klein} in the case of regular trees (Bethe
lattice) under weak random
perturbations, thus providing the first example of spectral delocalization
for an Anderson model. More general trees were subsequently treated in
\cite{KLW,FHS},  always in the setting of discrete Schr\"odinger operators. 
The stability of AC spectrum under perturbation by a non-random radial potential was proved in \cite{KLW2} in case of non-regular trees of finite cone type.

In this article we consider quantum trees, i.e.\ each edge is endowed some length $L_e$ and we study differential operators acting on the edges with appropriate boundary conditions at the vertices specified by certain coupling constants. The presence of AC spectrum for quantum trees appears to have been studied less systematically than in the case of discrete Schr\"odinger operators. In case of regular trees $\T_q$, it was shown in \cite{Car97} that the quantum tree obtained by endowing each edge with the same length $L$, the same symmetric potential $W$ on the edges and the same coupling constant $\alpha$ at the vertices, has a spectrum consisting of bands of pure AC spectrum, along with eigenvalues between the bands. The setting was a bit generalized quite recently in \cite{Car17}, where each vertex in a $2q$-regular tree is surrounded by the same set of lengths $(L_1,\dots,L_q)$, each length repeated twice, similarly the same set of symmetric potentials $(W_1,\dots,W_q)$, and the boundary conditions are taken to be Kirchhoff. The nature of the spectrum is partly addressed, but the possibility that it consists of a discrete set of points is not excluded. Finally, it was shown in \cite{ASW06} that the AC spectrum of the equilateral quantum tree \cite{Car97} remains stable under weak random perturbation of the edge lengths. The theorem however does not yield purity of the AC spectrum in some interval; one can only infer that the Lebesgue measure of the perturbed AC spectrum is close to the unperturbed one. We also mention the papers \cite{HP08,RS17} which consider radial quantum trees, for which a reduction to a half-line model can be performed.

Our aim here is twofold. First, go beyond regular graphs. We are
mainly interested in the case where the tree is the universal cover of
some compact quantum graph. This implies the set of different lengths,
potentials and coupling constants is finite, but the situation can be
much more general than the special Cayley graph setting considered in
\cite{Car17}. We show in this framework that the spectrum will consist
of (nontrivial) bands of pure AC spectrum, plus some discrete set of
eigenvalues. Next, we consider random perturbations of these trees. We can
perturb both the edge lengths and coupling constants. This setting is
more general than \cite{ASW06}, where the tree was regular and the
coupling constants were zero. But our main motivation here is
especially to derive the \emph{purity} of the perturbed AC spectrum, along
with a strong control on the resolvent, which is an important
ingredient to prove \emph{quantum ergodicity} for large quantum
graphs. We do this in a companion paper \cite{AISW}.

\subsection{Some definitions}\label{sec:defs}
\subsubsection{Quantum graphs}
Let $G=(V,E)$ be a graph with vertex set $V$ and edge set $E$. We will assume that there are no self-loops and that there is at most one edge between any two vertices, so that we can see $E$ as a subset of $V\times V$.
For each vertex $v\in V$, we denote by $d(v)$ the degree of $v$. We let $B= B(G)$ be the set of oriented edges (or bonds), so that $|B|=2|E|$. If $b\in B$, we shall denote by $\hat{b}$ the reverse bond. We write $o_b$ for the origin of $b$ and $t_b$ for the terminus of $b$. We define the map $e: B\longrightarrow E$ by $e( (v,v')) = \{v,v'\}$. An \emph{orientation} of $G$ is a map $\mathfrak{or}: E \longrightarrow B$ such that $e \circ \mathfrak{or} = Id_E$.

A \emph{length graph} $(V,E,L)$ is a connected combinatorial graph $(V,E)$ endowed with a map $L: E\rightarrow (0,\infty)$. If $b\in B$, we denote $L_b:= L(e(b))$.

A \emph{quantum graph} $\mathbf{Q}=(V,E,L,W,\alpha)$ is the data of:
\begin{itemize}
\item A length graph $(V,E,L)$,
\item A potential $W=(W_b)_{b\in B} \in \bigoplus_{b\in B} C^0([0,L_b]; \R)$ satisfying for $x\in [0,L_b]$,
\begin{equation}\label{eq:ReversePot}
W_{\widehat{b}}(L_b-x) = W_b(x)\,.
\end{equation}
\item Coupling constants $\alpha = (\alpha_v)_{v\in V} \in \R^V$.
\end{itemize}

The underlying \emph{metric graph} is the quotient
\[
\mathcal{G}:= \{ x= (b, x_b); b\in B, x_b\in [0,L_b]\} \slash \sim \,,
\]
where $(b,x_b)\sim (b',x'_{b'})$ if $b'= \hat{b}$ and $x'_{b'}= L_b-x_b$.

A function on the graph will be a map $f: \mathcal{G}\longrightarrow \R$. It can also be identified with a collection of maps $(f_b)_{b\in B}$ such that $f_b (L_b - \cdot) = f_{\hat{b}}(\cdot)$. We say that $f$ is \emph{supported on} $e$ for some $e\in E$ if $f_b \equiv 0$ unless $e(b)= e$.

If each $f_b$ is positive and measurable, we define $\int_\mathcal{G} f(x) \mathrm{d}x := \frac{1}{2}\sum_{b\in B} \int_0^{L_b} f_{b}(x_b) \mathrm{d}x_b$. We may then define the spaces $L^p(\mathcal{G})$ for $p\in [1, \infty]$ in the usual way.

Condition \eqref{eq:ReversePot} simply requires $W$ to be well-defined on $\mathcal{G}$ (no symmetry assumption).

When $G=(V,E)$ is a tree, i.e., contains no cycles (which will be the case in most of the paper), we say that $\mathbf{Q}$ is a \emph{quantum tree}, and we denote it by the letter $\mathbf{T}$ rather than $\mathbf{Q}$, while the set $\mathcal{G}$ is called a \emph{metric tree} and is denoted by $\mathcal{T}$.

\subsubsection{Orienting quantum trees}\label{subsec:Or}

Let $\T$ be a combinatorial tree, that is, a graph containing no cycles. We denote its vertex set by $V(\T)$ or
just $V$, its edge set by $E(\T)$, and its set of oriented edges by $B(\T)$. In all the paper, we will often write $v\in \T$ instead of $v\in V(\T)$ to lighten the notations.

In this paragraph, we explain how we can present the tree $\T$ in a \emph{coherent view}, that is to say, fix an oriented edge $b_o \in B(\T)$, and give an orientation to all the other edges of $\T$, by asking that they ``point in the same direction as $b_o$''.

More precisely, let us fix once and for all an oriented edge $b_o \in B(\T)$, corresponding to an edge $e_o\in E(\T)$.
If we remove the edge $e_o$ from $\T$, we obtain two connected components which are still combinatorial trees. We will write $\T_{b_o}^+$ for the connected component containing $t_{b_o}$, and $\T_{b_o}^-$ for the component containing $o_{b_o}$.

Let $v\in \T_{b_o}^+$ be at a distance $n$ from $t_{b_o}$. Amongst the neighbours of $v$, one of them is at distance $n-1$ from $t_{b_o}$: we denote it by $v_-$, and say that $v_-$ is the parent of $v$. The other neighbours of $v$ are at a distance $n+1$ from $t_{b_o}$, and are called the children of $v$. The set of children of $v$ is denoted by $\cN_v^+$. On the contrary, if $v\in \T_{b_o}^-$ is at distance $n$ from $o_{b_o}$, its unique neighbour at a distance $n-1$ from $o_{b_o}$ is called the child of $v$, and denoted by $v_+$, and its other neighbours are its parents, whose set we denote by $\cN_v^-$. These definitions are natural if we see the tree at the left of Figure \ref{fig:coherent-twisted} as a genealogical tree.

Let $V^{\ast} = \big{(}V(\T) \setminus \{o_{b_o},t_{b_o}\}\big{)}\cup \{o\}$. We define a map $\mathfrak{b}: V^* \longrightarrow B(\T)$ as follows: we set $\mathfrak{b}(o) = b_o$, and, if $v\in \T_{b_o}^+$, then $\mathfrak{b}(v) = (v_-, v)$, while if $v\in \T_{b_o}^-$, then $\mathfrak{b}(v) = (v, v_+)$. One easily sees that $e\circ \mathfrak{b}: V^* \longrightarrow E(\T)$ is a bijection, so that $\mathfrak{b} \circ (e\circ \mathfrak{b})^{-1}$ is an orientation of $\T$.  The map $\mathfrak{b}$ serves to index all oriented edges: those in $\T_{b_o}^+$ by their terminus, those in $\T_{b_o}^-$ by their origin, and $b_o$ by its ``midpoint'' $o$. The latter makes sense once we turn $\T$ into a quantum tree $\mathbf{T}$. We denote $L_v:=L_{\mathfrak{b}(v)}$ and $W_v:=W_{\mathfrak{b}(v)}$. The metric tree $\mathcal{T}$ can be identified with the set
\[
\mathcal{T} \equiv \bigsqcup_{v\in V^*} [0, L_v]= \left\{x= (v, x_v)| v\in V^*, x_v\in [0, L_v] \right\} .
\]

A function on $\mathcal{T}$ will then be the data of $\psi= (\psi_v)_{v\in V^*}$, where each $\psi_v$ is a function of the variable $x_v\in [0, L_v]$.

On a quantum tree, we consider the Schr\"odinger operator
\begin{equation}\label{e:schrodi}
(\cH_{\mathbf{T}} \psi_v)(x_v) = -\psi''_v(x_v) + W_v(x_v) \psi_v(x_v)
\end{equation}
with domain $D(\cH_{\mathbf{T}})$, the set of functions $(\psi_v)\in \underset{v\in V^\ast}{\bigoplus} W^{2,2}(0,L_v)$ satisfying the so-called $\delta$-conditions. Namely, for all $v\in \T_{b_o}^+$,
\begin{equation}\label{e:kir+}
\psi_v(L_v) = \psi_{v_+}(0) ~ \forall v_+\in \mathcal{N}_v^+ \qquad \text{and} \qquad \sum_{v_+\in \cN_v^+} \psi_{v_+}'(0) = \psi_v'(L_v) + \alpha_v \psi_v(L_v) \,,
\end{equation}
while for all $v\in \T_{b_o}^-$,
\begin{equation}\label{e:kir-}
\psi_{v_-}(L_{v_-}) = \psi_v(0)~ \forall v_-\in \mathcal{N}_v^- \qquad \text{and} \qquad \sum_{v_-\in\cN_v^-} \psi_{v-}'(L_{v_-}) + \alpha_v\psi_v(0) = \psi_v'(0) \,.
\end{equation}
Finally, $\psi_o(L_o) = \psi_{o_+}(0)$ $\forall o_+\in \mathcal{N}_o^+$, $\sum_{o_+} \psi_{o_+}'(0) = \psi_o'(L_o) + \alpha_{t_{b_o}} \psi_o(L_o)$, and $\psi_{o_-}(L_{o_-}) = \psi_o(0)$, $\sum_{o_-}\psi_{o_-}'(L_{o_-}) + \alpha_{o_{b_o}} \psi_o(0) = \psi_o'(0)$.
In a common convention we will refer to the $\al_v=0$ case as the \emph{Kirchhoff-Neumann condition}.

\begin{rem}\label{rem:coherent-twisted}
The above conventions mean that we see $\T$ as a doubly infinite genealogical tree. This is what we called the coherent view; it can also be pictured by saying that we imagine an electric flow moving from $\T_{b_o}^-$ to $\T_{b_o}^+$.

There is another way of orienting the graph which we call the \emph{twisted view}. This is done by turning $b_o$ into a V-shape and viewing $V(\T)$ as offspring of $o$. See Figure \ref{fig:coherent-twisted}  for an illustration; here one should think that $o$ is a source from which the electric flows moves outwards. When necessary to highlight this genealogical structure, we will write ${\T}_o^+$ for the set of offsprings of $o$. Each vertex $v$ has a single parent $v_-$ and several children; all the edges take the form $\{v, v_-\}$ for a unique $v$.

The link between the two views is immediate: functions on $\T_{b_o}^+$ in both views coincide, while on  $\T_{b_o}^-$, one replaces $b$ by $\hat{b}$ and derivatives take a  sign. Here $\hat{b}=(t_b,o_b)$ is the edge reversal of $b$. Hence, in the twisted view\footnote{These views of the tree have the advantage of avoiding the assumption that $\T$ has a special ``root'' vertex of degree one \cite{ASW06,RS17}. Such assumption simplifies the orientation a bit, but is not satisfied in many natural situations. As will be clear later, we will only need to study functions supported in $\mathbf{T}_{b_o}^{\pm}$, which is why we did not specify what happens on $b_o$. But one could specify that $\psi_o$ from the coherent view becomes $(\psi_o^{(1)},\psi_o^{(2)})$ in the twisted view, with $\psi_o^{(1)}(x)=\psi_o(x+\frac{L_o}{2})$ and $\psi_o^{(2)}(x) = \psi_o(\frac{L_o}{2}-x)$, for $x\in [0,\frac{L_o}{2}]$.}, all  functions in the domain of $\cH$ satisfy \eqref{e:kir+}.
\end{rem}

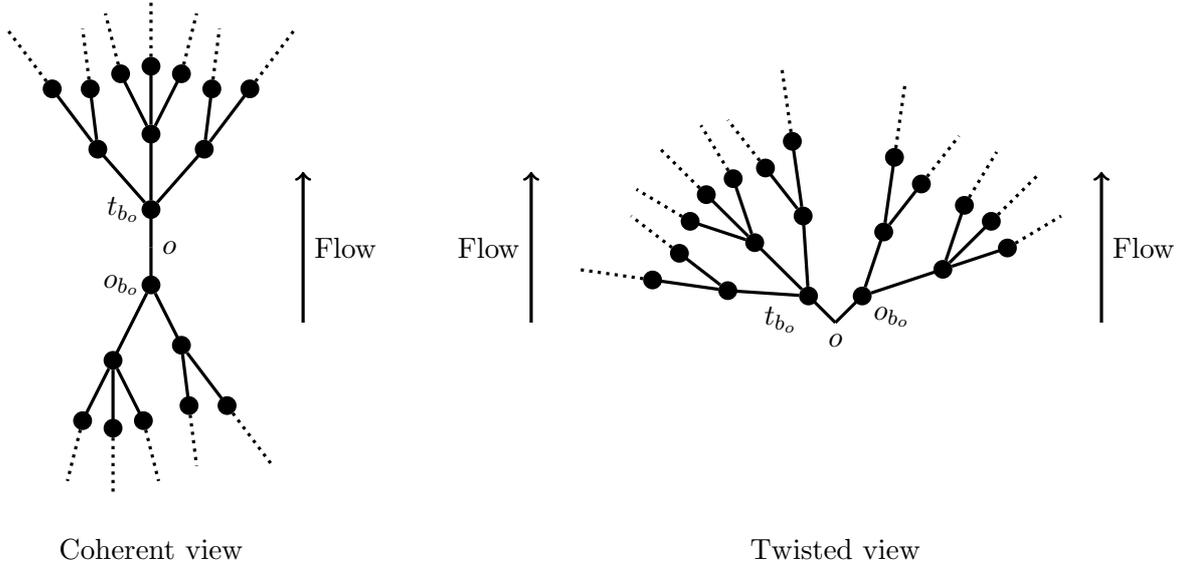
\begin{figure}
  \centering
  \begin{tikzpicture}
    [vertex/.style={fill, circle, inner sep=2.5pt, outer sep=0pt}]
\newcommand{\toptree}{%
      \draw (0,0) -- (0,0.5) node[vertex] (tb0) {};
      \draw (tb0) -- ++(-0.7,0.8) node[vertex] (tb00) {};
      \draw (tb0) -- ++(0,1) node[vertex] (tb01) {};
      \draw (tb0) -- ++(0.7,0.8) node[vertex] (tb02) {};
      \draw (tb00) -- ++(-0.6,0.8) node[vertex] (tb000) {};
      \draw[dotted] (tb000) -- ++(-0.6,0.8);
      \draw (tb00) -- ++(-0.1,0.8) node[vertex] (tb001) {};
      \draw[dotted] (tb001) -- ++(-0.1,0.8);
      \draw (tb01) -- ++(-0.4,0.8) node[vertex] (tb010) {};
      \draw[dotted] (tb010) -- ++(-0.2,0.8);
      \draw (tb01) -- ++(0,0.9) node[vertex] (tb011) {};
      \draw[dotted] (tb011) -- ++(0,0.9);
      \draw (tb01) -- ++(0.4,0.8) node[vertex] (tb012) {};
      \draw[dotted] (tb012) -- ++(0.2,0.8);
      \draw (tb02) -- ++(0.6,0.8) node[vertex] (tb020) {};
      \draw[dotted] (tb020) -- ++(0.6,0.8);
      \draw (tb02) -- ++(0.1,0.8) node[vertex] (tb021) {};
      \draw[dotted] (tb021) -- ++(0.1,0.8);
    }
\newcommand{\bottomtree}{%
      \draw (0,0) -- (0,0.5) node[vertex] (ob0) {};
      \draw (ob0) -- ++(-0.4,0.8) node[vertex] (ob00) {};
      \draw (ob0) -- ++(0.5,1) node[vertex] (ob01) {};
      \draw (ob00) -- ++(-0.6,0.8) node[vertex] (ob000) {};
      \draw[dotted] (ob000) -- ++(-0.6,0.8);
      \draw (ob00) -- ++(-0.1,0.8) node[vertex] (ob001) {};
      \draw[dotted] (ob001) -- ++(-0.1,0.8);
      \draw (ob01) -- ++(-0.4,0.8) node[vertex] (ob010) {};
      \draw[dotted] (ob010) -- ++(-0.2,0.8);
      \draw (ob01) -- ++(0,0.9) node[vertex] (ob011) {};
      \draw[dotted] (ob011) -- ++(0,0.9);
      \draw (ob01) -- ++(0.4,0.8) node[vertex] (ob012) {};
      \draw[dotted] (ob012) -- ++(0.2,0.8);
      }
    \begin{scope}[very thick]
      \toptree
      \begin{scope}[rotate=180] \bottomtree \end{scope}
      \node[anchor=west] at (0,0) {$o$};
      \node[anchor=east] at (0,0.5) {$t_{b_o}$};
      \node[anchor=east] at (0,-0.5) {$o_{b_o}$};
      \draw[->] (2,-1) --node[anchor=west] {Flow} (2,1);
      \node at (0,-4) {Coherent view};
    \end{scope}
    
    \begin{scope}[very thick, xshift=9cm, yshift=-1cm] 
      \begin{scope}[rotate=45] 
        \toptree
      \end{scope}
      \begin{scope}[rotate=-45]
        \bottomtree
      \end{scope}
      \node[anchor=north] at (0,0) {$o$};
      \node[anchor=north east] at (135:0.5) {$t_{b_o}$};
      \node[anchor=north west] at (45:0.5) {$o_{b_o}$};
      \node at (0,-3) {Twisted view};
      \draw[->] (-4,0) --node[anchor=east] {Flow} (-4,2);
      \draw[->] (3.5,0) --node[anchor=west] {Flow} (3.5,2);
    \end{scope}
  \end{tikzpicture}
  \caption{The two views of a tree.}
  \label{fig:coherent-twisted}
\end{figure}

Given $v\in V^{\ast}$, $z\in \C$, let $C_{z}(x)$ and $S_{z}(x)$ be a basis of solutions of the problem $-\psi_v''+W_v\psi_v = z \psi_v$ satisfying
\begin{equation}\label{e:csfun}
\begin{pmatrix} C_{z}(0) & S_{z}(0) \\ C_{z}'(0) & S_{z}'(0) \end{pmatrix} = \begin{pmatrix} 1 & 0 \\ 0 & 1 \end{pmatrix} \,.
\end{equation}
Then any solution $\psi_v$ of the problem satisfies
\begin{equation}\label{e:endvalue}
\begin{pmatrix} \psi_v(L_v,z) \\ \psi_v'(L_v,z) \end{pmatrix} = M_{z}(v) \begin{pmatrix} \psi_v(0,z)\\ \psi_v'(0,z) \end{pmatrix} \qquad \text{where } M_{z}(v) = \begin{pmatrix} C_{z}(L_v) & S_{z}(L_v) \\ C_{z}'(L_v) & S_{z}'(L_v) \end{pmatrix}.
\end{equation}
If $W_v\equiv0$ then the basis of solutions is
\begin{displaymath}
  S_z(x) = \frac{\sin\sqrt{z}x}{\sqrt{z}},\qquad\qquad
  C_z(x) = \cos\sqrt{z}x;
\end{displaymath}
if $W_v(x)=c_1+c_2\cos(2\pi x/L_v)$ then $S_z, C_z$ would be Mathieu
functions.  It is a standard fact that $S_z(x), C_z(x)$ are analytic
functions of $z\in\C$ (see for instance \cite[Chapter 1]{PT87}).

\subsection{Trees of finite cone type}\label{sec:condefs}

We define a \emph{cone} in $\T$ to be a subtree of the form $\T_b^+$ or $\T_b^-$, for some $b\in B(\T)$. Each cone $\T_b^+$ has an origin $t_b$, and each cone $\T_b^-$ has an end $o_b$. 

We say that two \emph{quantum cones} $\mathbf{T}_b^+$ and $\mathbf{T}_{b'}^+$ are \emph{isomorphic} if there is an isomorphism of combinatorial graphs $\varphi:\T_b^+ \to \T_{b'}^+$ such that $L_{\varphi(v)} = L_v$, $W_{\varphi(v)} = W_v$ and $\alpha_{\varphi(v)} = \alpha_v$ for all $v\in\T_b^+$. Isomorphic $\mathbf{T}_b^-$ and $\mathbf{T}_{b'}^-$ are defined the same way.

We say that $\mathbf{T}$ is a \emph{tree of finite cone type} if there exists $b_o\in B(\T)$ such that:
\begin{enumerate}[\rm (i)]
\item There are finitely many non-isomorphic quantum cones $\mathbf{T}_{(v_-,v)}^+$ as $v\in \T_{b_o}^+$.
\item There are finitely many non-isomorphic quantum cones $\mathbf{T}_{(w,w_+)}^-$ as $w\in \T_{b_o}^-$.
\end{enumerate}
Here $(t_{b_o})_-=o_{b_o}$ and $(o_{b_o})_+=t_{b_o}$. Note that in a regular tree, all cones $\T_b^{\pm}$ are isomorphic, but a necessary condition for it to be a quantum tree of finite cone type, is that its edges and vertices be endowed with finitely many lengths, potentials and coupling constants.\footnote{Also note that it is not required that there are finitely many non-isomorphic quantum trees $\mathbf{T}_{(v_-,v)}^-$ as $v\in \T_{b_o}^+$. To illustrate this point, consider the binary tree (so each vertex has $3$ neighbors except for the special root $\star$ with $2$ neighbors), let $b_o=(\star,v)$, with $v$ either neighbor. Then all cones $\T_b^+\subset \T_{b_o}^+$ look the same; they are binary trees. However, the backward cones $\T_b^-$ are distinct in each generation, because they ``see'' the special root at distinct distances. Despite this, $\T$ has finite cone type.}

If $\mathbf{T}$ is a tree of finite cone type, with $b_o\in B(\T)$ fixed, we may introduce a \emph{type} function $\ell:\T_{b_o}^+ \to \N_0=\N\cup\{0\}$, taking values in a finite set, such that $\ell(v) = \ell(w)$ iff $\mathbf{T}_{(v_-,v)}^+ \equiv \mathbf{T}_{(w_-,w)}^+$ as quantum trees. Similarly, $\ell:\T_{b_o}^- \to \N_0$ satisfies $\ell(v)=\ell(w)$ iff $\mathbf{T}_{(v,v_+)}^- \equiv \mathbf{T}_{(w,w_+)}^-$. Note that if $\ell(v)=\ell(w)$, then $W_v = W_w$, $L_v = L_w$ and $\alpha_v = \alpha_w$, since the corresponding isomorphism respects this information. Hence, any \emph{coherent} quantum tree $\mathbf{T}$ of finite cone type comes with the following structure:
\begin{enumerate}[\rm (a)]
\item A fixed $b_o\in B(\T)$.
\item Two finite sets of labels $\mathfrak{A}^+ = \{i_1,\dots,i_m\}$, $\mathfrak{A}^- = \{j_1,\dots,j_n\}$ and two matrices $M=(M_{i,j})_{i,j\in \mathfrak{A}^+}$, $N=(N_{i,j})_{i,j\in \mathfrak{A}^-}$. If $v\in \T_{b_o}^+$ has type $j$, it has $M_{j,k}$ children of type $k$. If $w\in\T_{b_o}^-$ has type $j$, it has $N_{j,k}$ parents of type $k$.
\item Finite sets $\{L_i\}_{i\in\mathfrak{A}^\pm}$, $\{W_i\}_{i\in\mathfrak{A}^\pm}$ and $\{\alpha_i\}_{i\in\mathfrak{A}^\pm}$ encoding the lengths, potentials and coupling constants, respectively. More precisely, $b_o$ is endowed a special length $L_o$ and potential $W_o$. If $(v_-,v)\in \T_{b_o}^+$ with $\ell(v) = i$, then $L_v = L_i$, $W_v = W_i$ and $\alpha_v = \alpha_i$. The same attribution is made if $(v,v_+)\in \T_{b_o}^-$ with $\ell(v) = i$.
\end{enumerate}

If we take the \emph{twisted} view instead, we only need one alphabet $\mathfrak{A} = \mathfrak{A}^+\cup\mathfrak{A}^-$ and one corresponding matrix $M=(M_{i,j})_{i,j\in\mathfrak{A}}$.

A trivial example is the equilateral, $(q+1)$-regular quantum tree, with identical potentials $W$ on each edge and identical coupling constant $\alpha$ on each vertex \cite{Car97}. In this case, all vertices in $\T_{b_o}^{\pm}$ have the same type, and we get two $1\times 1$ matrices $M = N = \begin{pmatrix} q \end{pmatrix}$.

An important class of examples comes from universal covers of finite undirected graphs. More precisely, if $G$ is a finite undirected graph and $\T$ is its universal cover, then $\T$ is a combinatorial tree satisfying condition (i). If we endow $G$ with a quantum structure $\mathbf{G}$ and lift it to $\T$ in the natural way, then the corresponding $\mathbf{T}$ will be a quantum tree of finite cone type.

Quantum trees of finite cone type satisfying (a)--(c) will be our basic, ``unperturbed'' trees. We denote the Schr\"odinger operator \eqref{e:schrodi} acting in this setting as $\cH_0$. Later on, we shall study random perturbations of these trees, and denote the corresponding operator by $\cH^{\omega}_\eps$, where $\eps$ is the strength of the disorder.

We make the following assumption on $\T$:
\begin{itemize}
\item[\textbf{(C1*)}] For any $k,l\in \mathfrak{A}^+$, there is $n=n(k,l)$ such that $(M^n)_{k,l}\ge 1$. Similarly, for $i,j\in \mathfrak{A}^-$, there is $n=n(i,j)$ with $(N^n)_{i,j}\ge 1$.
\end{itemize}

See Remark~\ref{rem:c1toilec1} below for a discussion of this condition. We may now state a first theorem, which describes the structure of the spectrum of $\cH_0= \mathcal{H}_\mathbf{T}$ on a tree $\mathbf{T}$ of finite cone type. We denote by $G_0^z(x,y)=(\cH_0-z)^{-1}(x,y)$ the Green's function of $\cH_0$.

\begin{thm}\label{thm:det0}
  Let $M,N$ satisfy \emph{\textbf{(C1*)}}. Then the spectrum of
  $\cH_0$ consists of a disjoint union of closed intervals and of isolated points:
  $\sigma(\cH_0)= \left(\bigsqcup_r I_r \right) \cup \mathfrak{P}$, where the $I_r$ are closed intervals, and $\mathfrak{P}$ is a discrete set. The spectrum is purely absolutely
  continuous in the interior of each band $\mathring{I}_r$. For
  $\lambda\in\mathring{I}_r$, and for any $v\in \T$, the limit $G^{\lambda+\ii0}_0(v,v)$
  exists and satisfies $\Im G^{\lambda+\ii0}_0(v,v) >0$,
  where $G_0^z$ is the Green's function of $\cH_0$.
  
\end{thm}

Let $R_{z,0}^{\pm}$ be the Weyl-Titchmarsh functions of $\cH_0$ as defined in \cite{ASW06}, see \eqref{e:WT}. Let $R_{\lambda,0}^{\pm}=R_{\lambda+\ii0,0}^{\pm}$ when the limit exists. Theorem~\ref{thm:det0} implies that $\Im R_{\lambda,0}^+(v) + \Im R_{\lambda,0}^-(v)>0$ in $\mathring{I}_r$. We will need the stronger property that $\Im R_{\lambda,0}^+(v)>0$ for all $v$. For this, we introduce the following strengthening of \textbf{(C1*)}.

\smallskip

\begin{itemize}
\item[\textbf{(C1)}] The quantum tree $\mathbf{T}$ is the universal cover of a finite quantum graph $\mathbf{G}$ of minimal degree $\ge 2$ which is not a cycle.
\end{itemize}

\begin{rem}\label{rem:c1toilec1}
Condition \textbf{(C1*)} means that on $\T_{b_o}^+$, any cone type $l\in\mathfrak{A}^+$ appears as offspring of any $k\in \mathfrak{A}^+$ after a finite number of generations, and similarly for $\T_{b_o}^-$. It is not required that cone types in $\mathfrak{A}^-$ appear in $\T_{b_o}^+$ - we only need the matrices $M$ and $N$ to be separately irreducible. We can also allow for ``rooted'' trees where the root $o$ has degree one. In this case the situation is a bit simpler actually; we only have to deal with one matrix $M$. Condition \textbf{(C1*)} applies in particular to trees with a ``radial periodic'' data, i.e.\ data that are periodic functions of the distance to the origin (such as some examples appearing in \cite{RS17}).

Assumption \textbf{(C1)} implies \textbf{(C1*)} (see Remark~\ref{rem:c1c1toile}), and is in fact more restrictive. In particular, $\mathbf{T}$ is ``unimodular'', that is, all data is somehow homogeneous as we move along the tree. This excludes for example the binary tree and more generally radial periodic trees, where the root plays a special role. However, such unimodular trees are still very general, they are actually the most interesting for us, and many techniques (such as a reduction to a half-line model) fail to tackle them. Even in the very simple case where the base graph $\mathbf{G}$ is regular but the edge lengths are not equal, the lifted structure in general will be neither radial periodic, nor identical around each vertex (in contrast to \cite{Car17}).

Note that the case where $\mathbf{G}$ is a cycle is already known when the couplings are zero. In this case $\cH_{\mathbf{T}}$ is just a periodic Schr\"odinger operator on $\R$ (of period $\le |\mathbf{G}|$), it is well-known that the spectrum is purely AC in this case \cite[Section XIII.16]{RS4}.
\end{rem}

\begin{thm}\label{thm:det}
If $\mathbf{T}$ satisfies \emph{\textbf{(C1)}}, then the spectrum of
  $\cH_0$ consists of a disjoint union of closed intervals and of isolated points:
  $\sigma(\cH_0)= \left(\bigsqcup_r I_r \right) \cup \mathfrak{P}$, where the $I_r$ are closed intervals, and $\mathfrak{P}$ is a discrete set. The spectrum is purely absolutely continuous in the interior of each band $\mathring{I}_r$. For $\lambda\in\mathring{I}_r$, the limit $R_{\lambda+\ii0,0}^+(v)$ exists for any $v\in V$ and satisfies $\Im R_{\lambda,0}^+(v) >0$.
\end{thm}

\begin{rem}\label{rem:trivialspec}
In Theorems~\ref{thm:det0} and \ref{thm:det}, it is not excluded in principle that $\bigcup_r I_r = \emptyset$, i.e.\ the spectrum consists of isolated points. We think this never happens for infinite quantum trees of finite cone type with the $\delta$-conditions we consider, i.e.\ we believe these should always have some continuous spectrum. We did not find such a result in the literature however. This is why we dedicate Section~\ref{sec:nontrivialspec} to prove the following: if $\mathbf{T}$ satisfies either:
\begin{enumerate}[\rm (1)]
\item assumption \textbf{(C1)} and has a single data $(L,\alpha,W)$ (all edges carry the same length, coupling and symmetric potential),
\item or has a general data $(L_e,\alpha_{o_e},W_e)_{e\in E(\mathbf{G})}$, but the finite graph $G$ is moreover \emph{Hamiltonian},
\end{enumerate}
then $\cH_{\mathbf{T}}$ always has some continuous spectrum, i.e.\ $\bigcup_r \mathring{I}_r \neq \emptyset$. Recall that a finite graph is Hamiltonian if it has a cycle that visits each vertex exactly once. Note that as a discrete tree, $\T$ may cover many different graphs. We only need \emph{one} of these finite graphs to be Hamiltonian. For example, we can consider any regular tree, despite the fact that some regular graphs (like the Petersen graph) are not Hamiltonian.

In particular, the Cayley tree considered in \cite{Car17} can be realized as the universal cover of the complete bipartite graph $K_{2q,2q}$, which is Hamiltonian. For this, use the fact that $K_{2q,2q}$ has a proper $2q$-edge-colouring and put the same length/potential on edges of the same colour. The lift of this is then a tree which has the same data around each vertex, and we may take $L_{q+j}=L_j$, $W_{q+j}=W_j$ to be in the setting \cite{Car17}. Then our theorems imply this tree has \emph{nontrivial} bands of pure AC spectrum, thus enriching the results of \cite{Car17}. Again, this is just one very special application of our framework.
\end{rem}

\subsection{Random perturbations of trees of finite cone type}
Fix a quantum tree $\mathbf{T}$ satisfying \textbf{(C1)}. As explained in Remark~\ref{rem:c1c1toile}, any such tree is a tree of finite cone type. We fix an edge $e\in E(\T)$, and see our quantum tree in the twisted view (in which all vertices are descendent of a vertex $o$), so as to deal with a single alphabet $\mathfrak{A}$ and a corresponding matrix $M$. We denote the lengths and coupling constants of the unperturbed tree $\mathbf{T}$ by $(L_v^0)_{v\in V^{\ast}}$ and $(\alpha_v^0)_{v\in V}$. These can also be denoted $(L_i^0)_{i\in\mathfrak{A}\cup \{o\}}$ and $(\alpha_i^0)_{i\in \mathfrak{A}}$. 
 We assume there are no potentials on the edges and the couplings are nonnegative:
\[
W_v^0\equiv 0\,  \quad \text{and} \quad \alpha_v^{0} \ge 0.
\]

We now want to analyze random perturbations of $\mathbf{T}$. For this purpose, we introduce a probability space $(\Omega, \curlyF, \prob)$, a family of random variables $\omega\in \Omega\mapsto (L_v^{\omega})_{v\in V^{\ast}}$ representing random lengths, and a family of random variables $\omega\in \Omega\mapsto (\alpha_v^{\omega})_{v\in V}$ representing random coupling constants. In principle we could also consider random potentials  $\omega\in \Omega\mapsto (W_v^{\omega})_{v\in  V^{\ast}}$, however here we assume there are no potentials on the edges even after perturbation. We also assume the perturbed couplings are nonnegative:
\[
 W_v^{\omega}\equiv 0\, \quad   \text{and} \quad \alpha_v^{\omega} \ge 0.
\]
We make the following assumptions on the random perturbation (see Remark \ref{rem:coherent-twisted} for the notation  $\T_{o}^+$):

\begin{itemize}
\item[\textbf{(P0)}] The operator $\cH_{\eps}^{\omega}$ is the Laplacian on the edges acting on $\bigoplus W^{2,2}(0,L_v^{\omega})$, satisfying $\delta$-conditions with coupling constants $(\alpha_v^{\omega})_{v\in \T}$, which are assumed to satisfy
\[
L_v^{\omega} \in \left[L_v^0 - \eps, L_v^0 + \eps\right] \quad \text{and} \quad \alpha_v^{\omega} \in \left[\alpha_v^0 - \eps, \alpha_v^0 + \eps\right] .
\]
\item[\textbf{(P1)}] For all $v,w\in \T_{o}^+$, the random variables $(\alpha_v^{\omega},L_v^{\omega})$ and $(\alpha_w^{\omega},L_w^{\omega})$ are independent if the forward trees of $v$ and $w$ do not intersect, i.e.\ if $\T_{(v_-,v)}^+ \cap \T_{(w_-,w)}^+ = \emptyset$.
\item[\textbf{(P2)}] For all $v,w\in \T_{o}^+$ that share the same label, the restrictions of the random variables $(\alpha^{\omega},L^{\omega})$ to the isomorphic forward trees of $v$ and $w$ are identically distributed.
\end{itemize}
 
\medskip

\begin{rem} Assumptions \textbf{(P1)} and \textbf{(P2)} hold, in particular, for independent identically distributed random variables (which is the main case we have in mind).
\end{rem}

We shall consider intervals $I$ lying in the interior of the unperturbed AC spectrum:
\begin{equation}\label{e:Sigma}
\Sigma = \textstyle\bigcup_r \mathring{I}_r\,,
\end{equation}
where $I_r$ are given in Theorem~\ref{thm:det}.

We will also need to ensure that the various $\sin \big{(}\sqrt{\lambda}L_v\big{)}$ do not vanish. More precisely, by \textbf{(P0)}, the perturbed lengths all lie in $\bigcup_{j\in \mathfrak{A}\cup \{0\}} [L_{j,\min}(\eps), L_{j,\max}(\eps)]$, where $L_{j,\min}(\eps) = L_j^0-\eps$ and $L_{j,\max}(\eps) = L_{j}^0+\eps$. We then assume
\begin{equation}\label{e:nodir}
I\cap \mathscr{D} = \emptyset \,,
\end{equation}
where the set $\mathscr{D}=\mathscr{D}_{\eps}$ 
is a ``thickening'' of the Dirichlet spectrum, given by
\[
\mathscr{D} = \bigcup_{j\in \mathfrak{A}\cup\{o\}}\bigcup_{n\ge 0} \left[\frac{\pi^2 n^2}{L_{j,\max}^2(\eps)},\frac{\pi^2n^2}{L_{j,\min}^2(\eps)}\right] .
\]
This ensures that $\sin \big{(}\sqrt{\lambda}L_v^{\omega}\big{)},\sin \big{(}\sqrt{\lambda}L_v^0\big{)}\neq 0$ for any $\lambda\in I$, $v\in \T$ and $\omega$.

Recall that the Weyl-Titchmarsh functions $R^+_z(v)$ will be introduced in (\ref{e:WT}). Introduce the following condition:

\medskip

\textbf{(Green\itshape{-s})} There is a non-empty open set $I_1$ and some $s>0$ such that for all $b\in \T$,
\[
\sup_{\lambda\in I_1,\eta\in (0,1)} \expect\left(\left|\Im R_{\lambda+\ii\eta}^+(o_{b})\right|^{-s} \right)<\infty \,.
\]
Condition \textbf{(Green\itshape{-s})} implies in particular that the spectrum in $I_1$ is purely AC, as long as it stays away from the Dirichlet spectrum, see Appendix~\ref{app:A2}. Here \textbf{(Green\itshape{-s})} refers to ``Green's function'' and the moment value $s$. In fact, such inverse bounds on the WT function imply moments bounds on the Green's function; see Corollary~\ref{cor:grencontrol}.

Introduce the following assumptions:
\begin{itemize}
\item[\textbf{(C0)}] The minimal degree of $\T$ is at least $3$.
\item[\textbf{(C2)}] For each $k\in\mathfrak{A}$, there is $k'$ with $M_{k,k'}\ge 1$ such that for any $l\in\mathfrak{A}$: $M_{k,l}\ge 1$ implies $M_{k',l}\ge 1$.
\end{itemize}
The second assumption ensures that each vertex $v\in \T$ has at least one child $v'$ such that
each label found in $\cN_v^+$ can also be found in $\cN_{v'}^+$. See \cite{AS4} for examples of such trees.

\begin{thm}\label{thm:random}
  Let $\mathbf{T}$ satisfy \emph{\textbf{(C0), (C1)}},
  \emph{\textbf{(C2)}} and $(\alpha,L)$ satisfy \emph{\textbf{(P0)},
    \textbf{(P1)}} and \emph{\textbf{(P2)}}, and be without edge
  potentials.  Let $I\subset \Sigma$ be compact with
  $I\cap\mathscr{D} = \emptyset$. Then for any $s>1$, we may find
  $\eps_0(I,s)$ such that \emph{\textbf{(Green\itshape{-s})}} holds on
  $I$ for any $\eps\le \eps_0$. In particular,
  $\sigma(\cH^{\omega}_{\eps})$ has purely absolutely continuous
  spectrum almost-surely in $I$.
\end{thm}

The ``in particular'' part is due to Theorem~\ref{thm:accrit}.

In the above theorem, the disorder window $\eps_0(I,s)$ depends on the value of the moment $s$. We can actually obtain a disorder window valid uniformly for all $s$, but at the price of assuming some regularity on the $\delta$-potential~:

\textbf{(P3)} For any $v\in\T$, $\ell(v)=j$, $j\in\mathfrak{A}$, the distribution $\nu_j$ of $\alpha_v^\omega$ is H\"older continuous~: there exist $C_{\nu}>0$ and $\beta\in (0,1]$ such that for any bounded $I\subset \R$,
\[
\max_{j\in\mathfrak{A}}\nu_j(I) \le C_{\nu} \cdot |I|^\beta \, .
\]

This holds e.g. if the $\nu_j$ are absolutely continuous with a bounded density (then $\beta=1$).

\begin{thm}\label{thm:moments}
Suppose in addition to the assumptions of Theorem~\ref{thm:random} that \emph{\textbf{(P3)}} is satisfied. Then there exists $\eps_0(I)$ such that for any $\eps\le \eps_0$ and any $s\ge 1$, \emph{\textbf{(Green\itshape{-s})}} holds on $I$.
\end{thm}

\section{Green's function on quantum trees}\label{sec:greenquan}
The aim of this section is to derive quantum analogs for the well-known recursive formulas of the Green's function on combinatorial trees. These identities will play a key role in the spectral analysis of the quantum tree, and may be of independent interest. In fact, we shall also need them when studying quantum ergodicity in \cite{AISW}. Some of these identities appeared before in \cite{ASW06}.

In all this section, we fix a quantum tree $\mathbf{T}$, and denote by $W_{\max}^{2,2}(\mathbf{T})$ the set of $\psi=(\psi_v)$ such that $\psi_v\in W^{2,2}(0,L_v)$, $\sum_{v}\|\psi_v\|_{W^{2,2}}^2<\infty$.

If $b\in B(\T)$, recall the notation $\T_b^{\pm}$ of \S~\ref{sec:defs}. If $x = (b, x_b)\in \mathcal{T}$, we define a quantum tree $\mathbf{T}_x^+$ by $\mathbf{T}_x^+ = [x_b,t_b]\cup \mathbf{T}_b^+$. More precisely, add a vertex $v_x$ at $x$, let $V(\mathbf{T}_x^+)= V(\mathbf{T}^+_b) \cup \{v_x\}$, $E(\mathbf{T}_x^+)= E(\mathbf{T}^+_b) \cup \{v_x,t_b\}$, $L_{\{v_x,t_b\}} = L_b-x_b$, $W_{(v_x,t_b)} = (W_b)|_{[L_b-x_b, L_b]}$, $\alpha_{v_x}=0$, and the lengths, potentials and coupling constants be the same as those of $\mathbf{T}^+_b$ on the rest of the edges.
In a similar fashion, we define $\mathbf{T}_x^- = \mathbf{T}_b^-\cup[o_b,x_b]$.

Let $u=(b, u_b)\in \mathcal{T}$. By \cite[Theorem 2.1]{ASW06}, which remains true in our context, if we define $\cH_{\mathbf{T}_u^{\pm}}^{\max}$ on $\mathcal{T}_u^{\pm}$ to be the Schr\"odinger operator $-\Delta+W$ with domain $D(\cH_{\mathbf{T}_u^{\pm}}^{\max})$, the set of $\psi\in W_{\max}^{2,2}(\mathbf{T}_u^{\pm})$ satisfying $\delta$-conditions on inner vertices of $\mathbf{T}_u^{\pm}$, then for any $z\in\C^+:=\h:= \{z\in\C : \Im z >0\}$,
there are unique $z$-eigenfunctions $V_{z;u}^+\in D(\cH_{\mathbf{T}_u^+}^{\max})$, $U_{z;u}^-\in D(\cH_{\mathbf{T}_u^-}^{\max})$ satisfying $U_{z;u}^-(u)=V_{z;u}^+(u)=1$. Complex eigenvalues exist because $\cH_{\mathbf{T}_u^\pm}^{\max}$ is not self-adjoint, as there are no domain conditions at $u$. 

\begin{lem}\label{lem:A.2}
Let $z\in \C^+$. The resolvent $G^{z}$ of $\cH_{\mathbf{T}}$ is an integral operator with kernel $G^{z}(x,y)$ defined as follows. Given $x,y\in \mathcal{T}$, fix $o,v$ such that $x,y\in \mathcal{T}_o^+ \cap \mathcal{T}_v^-$. Then
\begin{equation}\label{e:greenasw}
G^{z}(x,y) = \begin{cases} \frac{U_{z;v}^-(x) V_{z;o}^+(y)}{\cW^{z}_{v,o}(y)} &\text{if } y\in \mathcal{T}_x^+,\\ \frac{U_{z;v}^-(y) V_{z;o}^+(x)}{\cW^{z}_{v,o}(y)} &\text{if } y\in \mathcal{T}_x^-,\end{cases}
\end{equation}
where $\cW^{z}_{v,o}(x)$ is the Wronskian
\[
  \cW^{z}_{v,o}(x) = V_{z;o}^+(x)(U_{z;v}^-)'(x) - (V_{z;o}^+)'(x) U_{z;v}^-(x)
  \,.
\]
\end{lem}
Versions of this lemma previously appeared in \cite[Lemma A.2]{ASW06} and \cite[Lemma~D.15]{HP08}. We give the proof in Appendix~\ref{sec:app1} for completeness.

Since for each $z\in \C^+$, $G^z$ satisfies the $\delta$-boundary conditions in each of its arguments, we deduce that, whenever $o_b= o_{b'}=v$, we have $G^z( (b, 0), \cdot) = G^z ((b', 0), \cdot)$ and $G^z(\cdot, (b, 0)) = G^z (\cdot, (b', 0))$. These quantities will therefore be denoted by $G^z (v, \cdot)$ and $G^z(\cdot, v)$ respectively.

As in \cite{ASW06}, we define the \emph{Weyl-Titchmarch} (WT) functions for $x\in \mathcal{T}$ by
\begin{equation}\label{e:WT}
R^+_{z}(x) = \frac{(V_{z;o}^+)'(x)}{V_{z;o}(x)} \quad \text{and} \quad R^-_{z}(x) =  \frac{-(U_{z;v}^-)'(x)}{U_{z;v}^-(x)} \,.
\end{equation}
Note that we take here the \emph{coherent} point of view, which is why there is a negative sign in the definition of $R^-_{z}(x)$.

Given an oriented edge $b=(o_b,t_b)$, we define
\begin{equation}\label{e:zetadef}
\zeta^{z}(b) = \frac{G^{z}(o_b,t_b)}{G^{z}(o_b,o_b)} \,.
\end{equation}

\begin{rem}\label{rem:zeta}
If $a< \inf \sigma(\cH_{\mathbf{T}})$, then $\zeta^{z}(b)$ is well-defined on $\C \setminus (a,\infty)$, i.e.\ the denominator does not vanish, as follows from \eqref{e:greenasw} and the proof of \cite[Theorem 2.1(ii)]{ASW06}. The proof also shows that $z \mapsto \zeta^{z}(b)$ is holomorphic on $\C \setminus (a,\infty)$ and real-valued on $(-\infty,a]$.

In the case of combinatorial trees, $\zeta^{z}(b)$ coincides with what was denoted $\zeta^{z}_{o_b}(t_b)$ in \cite{AS2}. In fact, by the multiplicative property of the Green function, we have $\frac{G^{z}(o_b,t_b)}{G^{z}(o_b,o_b)} = \frac{G^{z}(o_b,o_b)\zeta_{o_b}^{z}(t_b)}{G^{z}(o_b,o_b)} = \zeta_{o_b}^{z}(t_b)$.

In the case that $\mathbf{T}$ is the $(q+1)$-regular tree with equilateral edges, with identical coupling constants and potentials, then $\zeta^{z}(b)$ is the quantity $\mu^-(z)$ in \cite{Car97}, and is independent of $b$. Moreover, the limit $\mu^-(\lambda) = \lim_{\eta \downarrow 0} \mu^-(\lambda+\ii\eta)$ exists in this case, provided that $\lambda$ is not in the Dirichlet spectrum, i.e., that $\sin (\lambda L)\neq 0$.

Finally, for the quantum Cayley graphs of \cite{Car17}, the $\zeta^{z}(b)$ coincide with the multipliers $\mu_m(z)$. Hence, there are finitely many distinct $\zeta^{z}(b)$. Moreover, in this setting, $\zeta^{z}(\hat{b}) = \zeta^{z}(b)$, and $\zeta^{z}(gb) = \zeta^{z}(b)$, where $g$ is an element of the group acting on the graph.
\end{rem}

Given an oriented edge $b$, we will denote by $\cN_b^+$ the set of outgoing edges from $b$, i.e.\ the set of $b'$ with $o_{b'}=t_b$ and $b'\neq \hat{b}$.

\begin{lem}
Let $z\in \C^+$. We have the following relations between $\zeta^{z}$ and the WT functions $R_{z}^{\pm}$:
\begin{equation}\label{e:zetawt}
\zeta^{z}(b) = C_{z}(L_b) + R_{z}^+(o_b)S_{z}(L_b)\,, \qquad \zeta^{z}(\hat{b}) = S_{z}'(L_b) + R_{z}^-(t_b) S_{z}(L_b) \,,
\end{equation}
\begin{equation}\label{e:r+-id}
R^+_{z}(t_b) = \frac{S_{z}'(L_b)}{S_{z}(L_b)} - \frac{1}{S_{z}(L_b)\zeta^{z}(b)}\,, \qquad R_{z}^-(o_b) = \frac{C_{z}(L_b)}{S_{z}(L_b)} - \frac{1}{S_{z}(L_b) \zeta^{z}(\hat{b})} \,.
\end{equation}
Moreover,
\begin{equation}\label{e:1}
 \frac{1}{\zeta^{z}(b) S_{z}(L_b)} + \sum_{b^+\in \cN_b^+} \frac{\zeta^{z}(b^+)}{S_{z}(L_{b^+})} = \sum_{b^+\in \cN_b^+} \frac{C_{z}(L_{b^+})}{S_{z}(L_{b^+})} + \frac{S_{z}'(L_b)}{S_{z}(L_b)} + \alpha_{t_b} \,,
\end{equation}
\begin{equation}\label{e:zetainv}
\frac{1}{\zeta^{z}(b)} - \zeta^{z}(\hat{b}) = \frac{S_{z}(L_b)}{G^{z}(t_b,t_b)}\,, \qquad \frac{\zeta^{z}(\hat{b})}{\zeta^{z}(b)} = \frac{G^{z}(o_b,o_b)}{G^{z}(t_b,t_b)} \,,
\end{equation}
and
\begin{equation}\label{e:2}
 \sum_{b^+\in \cN_b^+} \frac{C_{z}(L_{b^+})}{S_{z}(L_{b^+})} + \frac{S_{z}'(L_b)}{S_{z}(L_b)} + \alpha_{t_b} = \sum_{t_{b'}\sim t_b} \frac{\zeta^{z}(b')}{S_{z}(L_{b'})} + \frac{1}{G^{z}(t_b,t_b)} \,,
\end{equation}
where $b'=(t_b,t_{b'})$. Given a non-backtracking path $b_1,\ldots, b_k$ (that is to say, $o_{b_{i+1}} = t_{b_i}$ and $t_{b_{i+1}} \neq o_{b_i}$ for all $i \in \{1,\ldots,k-1\}$), we have the multiplicative property

\begin{equation}\label{e:greenmul}
G^{z}(o_{b_1},t_{b_k}) = G^{z}(o_{b_1},o_{b_1})\zeta^{z}(b_1)\cdots \zeta^{z}(b_k) = G^{z}(t_{b_k},t_{b_k}) \zeta^{z}(\hat{b}_1) \cdots \zeta^{z}(\hat{b}_k)\,.
\end{equation}
Finally, for any path $b_1,\ldots, b_k$, we have
\begin{equation}\label{e:sym}
G^{z}(o_{b_1}, t_{b_k}) = G^{z}(t_{b_k}, o_{b_1})\,.
\end{equation}
\end{lem}
\begin{proof}
By \eqref{e:greenasw}, $\zeta^{z}(b) = \frac{V_{z;o}^+(t_b)}{V_{z;o}^+(o_b)} \cdot \frac{\cW_{v,o}^{z}(o_b)}{\cW_{v,o}^{z}(t_b)}$. But the Wronskian is constant on $b$, as checked by differentiating it. Moreover, since $V_{z;o}^+(x_b)$ is a $z$-eigenfunction on $b$, we have $V_{z;o}^+(t_b) = C_{z}(L_b) V_{z;o}^+(o_b) + S_{z}(L_b) (V_{z;o}^+)'(o_b)$. Hence, $\zeta^{z}(b) = C_{z}(L_b) + R^+_{z}(o_b)S_{z}(L_b)$ as claimed. Next, $\zeta^{z}(\hat{b}) = \frac{G^{z}(t_b,o_b)}{G^{z}(t_b,t_b)} = \frac{U_{z;v}^-(o_b)}{U_{z;v}^-(t_b)}$ again by constancy of $\cW_{v,o}^{z}$ on $b$. Writing \eqref{e:endvalue} in the form
\begin{equation}\label{e:invalue}
\begin{pmatrix} \psi(o_b,z) \\ \psi'(o_b,z) \end{pmatrix} = M_{z}(b)^{-1} \begin{pmatrix} \psi(t_b) \\ \psi'(t_b) \end{pmatrix} = \begin{pmatrix} S_{z}'(L_b) & -S_{z}(L_b) \\ -C_{z}'(L_b) & C_{z}(L_b) \end{pmatrix} \begin{pmatrix} \psi(t_b) \\ \psi'(t_b) \end{pmatrix} \,,
\end{equation}
we get $U_{z;v}^-(o_b) = S_{z}'(L_b) U_{z;v}^-(t_b) -S_{z}(L_b)(U_{z;v}^-)'(t_b)$, so $\zeta^{z}(\hat{b}) = S_{z}'(L_b) + R^-_{z}(t_b) S_{z}(L_b)$ as claimed.

Next, $V_{z;o}^+(y_b) = V_{z;o}^+(o_b) C_{z}(y_b) + (V_{z;o}^+)'(o_b) S_{z}(y_b)$, so
\begin{displaymath}
R_{z}^+(t_b) = \frac{R_{z}^+(o_b)S_{z}'(L_b)+C_{z}'(L_b)}{R_{z}^+(o_b)S_{z}(L_b)+C_{z}(L_b)}.
\end{displaymath}
Using $R_{z}^+(o_b) = \frac{\zeta^{z}(b)-C_{z}(L_b)}{S_{z}(L_b)}$, we get
\begin{displaymath}
R_{z}^+(t_b) = \frac{\zeta^{z}(b)S_{z}'(L_b) - C_{z}(L_b)S_{z}'(L_b)+C_{z}'(L_b)S_{z}(L_b)}{\zeta^{z}(b)S_{z}(L_b)}.
\end{displaymath}
The first part of \eqref{e:r+-id} follows by the Wronskian identity $C_z(x)S_z'(x)-C_z'(x)S_z(x)=1$.

For the second part, by \eqref{e:invalue},
\begin{align*}
  R_{z}^-(o_b) = \frac{-(U_{z;v}^-)'(o_b)}{U_{z;v}^-(o_b)} &= \frac{C_{z}'(L_b)U_{z;v}^-(t_b) - C_{z}(L_b)(U_{z;v}^-)'(t_b)}{S_{z}'(L_b)U_{z;v}^-(t_b)-S_{z}(L_b)(U_{z;v}^-)'(t_b)} \\
  &= \frac{C_{z}(L_b)R_{z}^-(t_b)+C_{z}'(L_b)}{S_{z}(L_b)R_{z}^-(t_b)+S_{z}'(L_b)}.
\end{align*}
The claim now follows as before using \eqref{e:zetawt}.

Since $V_{z;o}^+$ satisfies the $\delta$-conditions, we have
\begin{equation}\label{e:rdelta}
\sum_{b^+\in \cN_b^+} R_{z}^+(o_{b^+}) = R_{z}^+(t_b) + \alpha_{t_b} = \frac{S_{z}'(L_b)}{S_{z}(L_b)} - \frac{1}{S_{z}(L_b)\zeta^{z}(b)} + \alpha_{t_b}.
\end{equation}
Recalling \eqref{e:zetawt}, this proves \eqref{e:1}.

By \eqref{e:greenasw},
\begin{displaymath}
\frac{1}{G^{z}(t_b,t_b)} = \frac{\cW_{v,o}^{z}(t_b)}{U_{z;v}^-(t_b)V_{z;o}^+(t_b)} = \frac{(U_{z;v}^-)'(t_b)}{U_{z;v}^-(t_b)} - \frac{(V_{z;o}^+)'(t_b)}{V_{z;o}^+(t_b)},
\end{displaymath}
so
\begin{equation}\label{e:greener}
\frac{1}{G^{z}(t_b,t_b)} = -(R_{z}^+(t_b) + R_{z}^-(t_b)) \,.
\end{equation}
We have $R_{z}^+(t_b) = \frac{S_{z}'(L_b)}{S_{z}(L_b)} - \frac{1}{S_{z}(L_b)\zeta^{z}(b)}$ and $R^-_{z}(t_b) = \frac{\zeta^{z}(\hat{b})-S_{z}'(L_b)}{S_{z}(L_b)}$ by \eqref{e:zetawt} and \eqref{e:r+-id}. Hence,
\begin{equation}\label{e:greener2}
R_{z}^+(t_b) + R_{z}^-(t_b) = \frac{-1}{S_{z}(L_b)\zeta^{z}(b)} + \frac{\zeta^{z}(\hat{b})}{S_{z}(L_b)},
\end{equation}
proving the first part of \eqref{e:zetainv}.

For the second part, we showed that $\frac{S_{z}(L_b)}{G^{z}(t_b,t_b)} = \frac{1-\zeta^{z}(b)\zeta^{z}(\hat{b})}{\zeta^{z}(b)}$, so replacing $b$ by $\hat{b}$ we deduce that $\frac{S_{z}(L_b)}{G^{z}(o_b,o_b)} = \frac{1-\zeta^{z}(\hat{b})\zeta^{z}(b)}{\zeta^{z}(\hat{b})}$, so $\frac{\zeta^{z}(\hat{b})}{\zeta^{z}(b)} = \frac{G^{z}(o_b,o_b)}{G^{z}(t_b,t_b)}$.

It follows from (\ref{e:greener}) and (\ref{e:greener2}) that
\begin{displaymath}
\frac{1}{\zeta^{z}(b)S_{z}(L_b)} = \frac{\zeta^{z}(\hat{b})}{S_{z}(L_b)} + \frac{1}{G^{z}(t_b,t_b)}.
\end{displaymath}
Inserting this expression in \eqref{e:1}, we obtain \eqref{e:2}.

As previously observed, the Wronskian is constant on each $b$, so $\zeta^{z}(\hat{b}) = \frac{U_{z;v}^-(o_b)}{U_{z;v}^-(t_b)}$. Hence,
\begin{align*}
  G^{z}(t_{b_k},t_{b_k})\zeta^{z}(\hat{b}_1)\cdots\zeta^{z}(\hat{b}_k) &=  \frac{U_{z;v}^-(v_k)V_{z;o}^+(v_k)}{\cW_{v,o}^{z}(v_k)} \frac{U_{z;v}^-(v_0)}{U_{z;v}^-(v_1)}\cdots\frac{U_{z;v}^-(v_{k-1})}{U_{z;v}^-(v_k)}\\
  &= \frac{U_{z;v}^-(v_0)V_{z;o}^+(v_k)}{\cW_{v,o}^{z}(v_k)} = G^{z}(o_{b_1},t_{b_k}).
\end{align*}
By \eqref{e:zetainv}, $\zeta^{z}(b_1)\cdots\zeta^{z}(b_k) = \zeta^{z}(\hat{b}_1)\cdots\zeta^{z}(\hat{b}_k) \cdot \frac{G^{z}(t_{b_k},t_{b_k})}{G^{z}(o_{b_1}, o_{b_1})}$, proving the other equality.

Finally, by the first part of \eqref{e:greenmul} we have $G^{z}(t_{b_k},o_{b_1})=G^{z}(t_{b_k},t_{b_k}) \zeta^{z}(\hat{b}_k)\cdots\zeta^{z}(\hat{b}_1) = G^{z}(o_{b_1},t_{b_k})$ by the second part.
\end{proof}

For the following lemma, fix $o,v\in V$ and consider the WT functions \eqref{e:WT}. Assume that $o_b,t_b\in V(\mathbf{T}_o^+)\cap V(\mathbf{T}_v^-)$, that is to say, that $b\in B(\mathbf{T}_o^+)$ and $\hat{b}\in B(\mathbf{T}_v^-)$. Let $\mathbf{T}_{o_b}^+ \subseteq \mathbf{T}_o^+$ and $\mathbf{T}_{t_b}^- \subseteq \mathbf{T}_v^-$ be the subtrees starting at $o_b$ and $t_b$, respectively. Let $G_{\mathbf{T}_{o_b}^+}(x,y)$ be the Green kernel of the $\delta$-problem on $\mathbf{T}_{o_b}^+$. This means the usual $\delta$-conditions at $v\in V(\mathbf{T}_{o_b}^+)$, with $\alpha_{o_b}=0$.
Similarly, $G_{\mathbf{T}_{t_b}^-}(x,y)$ is the Green kernel of the $\delta$-problem on $\mathbf{T}_{t_b}^-$. 

We will need the notion of Herglotz functions \cite{Pick} throughout the paper. A \emph{Herglotz function} (a.k.a. \emph{Nevanlinna} function
or \emph{Pick} function) is an analytic function from $\C^+$ to $\C^+$. Herglotz functions form a positive cone: if $f_1$, $f_2$ are Herglotz and $a_1, a_2$ are positive constants, then $a_1f_1 + a_2f_2$
is Herglotz.  Composition of two Herglotz functions is again a Herglotz function.  The functions $z\mapsto\sqrt{z}$ and $z\mapsto-1/z$ for example are Herglotz.

Every Herglotz function $f$ has a canonical representation \cite[Theorem II.I]{Pick} of the form
\begin{equation}
  \label{eq:herglotz}
  f(z) = Az + B + \int_\R \left(\frac{1}{t-z} - \frac{t}{1+t^2} \right) \dd\mathfrak{m}(t),
\end{equation}
where $A$ and $B$ are constants and $\mathfrak{m}$ is a Borel measure
satisfying $\int_\R (1+t^2)^{-1}\,\dd\mathfrak{m}<\infty$.

\begin{lem}\label{lem:ASW}
Let $b\in \mathbf{T}$ and $z\in \C^+$. Let $o,v\in V$ be such that $b\in \mathbf{T}_o^+$ and $\hat{b}\in \mathbf{T}_v^-$. Then we may express
\begin{equation}\label{e:WTtronq}
R_{z}^+(o_b) = \frac{-1}{G_{\mathbf{T}_{o_b}^+}^{z}(o_b,o_b)} \qquad \text{and} \qquad R_{z}^-(t_b) = \frac{-1}{G_{\mathbf{T}_{t_b}^-}^{z}(t_b,t_b)}\,,
\end{equation}
where $G_{\mathbf{T}_v^{\pm}}^z(v,v)$ are defined with the Neumann condition at $v$.

The functions $F(z) = R_{z}^+(o_b)$, $R_{z}^-(t_b)$ and $G^{z}(v,v)$
are Herglotz functions.  If all $W_v\ge 0$ and $\alpha_v\ge 0$, then
$\widetilde{F}(z) = \frac{R_{z}^+(o_b)}{\sqrt{z}}$,
$\frac{R_{z}^-(t_b)}{\sqrt{z}}$ are also Herglotz.

Moreover, we have the following ``current'' relations:
\begin{equation}\label{e:ASW}
\sum_{b^+\in \cN_b^+} \Im R_{z}^+(o_{b^+}) \le \frac{\Im R_{z}^+(o_b)}{|\zeta^{z}(b)|^2} \qquad \text{and} \qquad \sum_{b^-\in \cN_b^-} \Im R_{z}^-(t_{b^-}) \le \frac{\Im R_{z}^-(t_b)}{|\zeta^{z}(\hat{b})|^2} \,.
\end{equation}
Equality holds in both cases if $\Im z=0$, whenever defined. 
\end{lem}
Most statements of this lemma appear in \cite{ASW06}. We give the proof in Appendix~\ref{sec:app1} for completeness. We also deduce that $-\frac{S_z'(L_b)}{S_z(L_b)}$ and $S_{z}(L_b)\zeta^{z}(b)\in \C^+$, see Remark~\ref{rem:hj}.

The following corollary says that the inverse moments of the imaginary part of the WT functions essentially control all relevant spectral quantities on the tree~:

\begin{cor}\label{cor:grencontrol}
Let $I\subset \R$ be compact, $I\cap \mathscr{D}=\emptyset$, and $z\in \C^+$. Fix $c_1,c_2,c_3>0$ such that for all $z\in I+\ii[0,1]$, $L_b\in [L_{\min},L_{\max}]$,
\begin{equation}\label{e:cj}
c_1\le |S_z(L_b)|\le c_2 \qquad \text{and} \qquad |C_z(L_b)|\le c_3 \,.
\end{equation}
Then for any $p\ge 1$, and $b\in \T$,
\[
|G^z(o_b,o_b)|^p\le |\Im R_z^+(o_b)|^{-p}\,, \qquad |\zeta^z(b)|^p \le c_1^{-p}\sum_{b^+\in \cN_b^+}|\Im R_z^+(o_{b^+})|^{-p}
\]
and
\[
|R_z^+(o_b)|^p\le c_1^{-p}2^{p-1}\left(c_1^{-p}\sum_{b^+\in \cN_b^+}|\Im R_z^+(o_{b^+})|^{-p}+c_3^p\right).
\]
\end{cor}
Also note that $R_z^-(t_b) = R_z^+(o_{\widehat{b}})+\frac{C_z(L_b)-S_z'(L_b)}{S_z(L_b)}$ using \eqref{e:zetawt}, so\footnote{If $W_b$ is symmetric, i.e. $W_b(L_b-x_b)=W_b(x_b)$, then $S'_z(L_b)=C_z(L_b)$, so $R_z^-(t_b) = R_z^+(o_{\widehat{b}})$.} up to choosing $c_4>0$ with $|S_z'(L_b)|\le c_4$, a control over all $R_z^+(o_b)$, $b\in \T$, implies a control over all $R_z^-(t_b)$.
\begin{proof}
We have by \eqref{e:greener}, $|G^z(o_b,o_b)|^p = |R_z^+(o_b)+R_z^-(o_b)|^{-p}\le |\Im R_z^+(o_b)|^{-p}$. By \eqref{e:rdelta},
\begin{align*}
|\zeta^z(b)|^p \le c_1^{-p}|S_z(L_b)\zeta^z(b)|^p &\le c_1^{-p}\Big|\alpha_{t_b}+\frac{S_z'(L_b)}{S_z(L_b)}-\sum_{b^+\in \cN_b^+}R_z^+(o_{b^+})\Big|^{-p} \\
&\le c_1^{-p}\Big|\Im\Big(\alpha_{t_b}+\frac{S_z'(L_b)}{S_z(L_b)}-\sum_{b^+\in \cN_b^+}R_z^+(o_{b^+})\Big)\Big|^{-p} \\
&\le c_1^{-p}\Big( \sum_{b^+\in\cN_b^+}\Im R_z^+(o_{b^+})\Big)^{-p} \le c_1^{-p}|\Im R_z^+(o_{b^+})|^{-p},
\end{align*}
where we used that $-\frac{S_z'(L_b)}{S_z(L_b)}$ and $R_z^+(o_e)$ are Herglotz in the last line, with $b^+\in \cN_b^+$ arbitrary. Hence,
\begin{align*}
|R_z^+(o_b)|^p &= \Big|\frac{\zeta^z(b)-C_z(L_b)}{S_z(L_b)}\Big|^p \le c_1^{-p}2^{p-1}(|\zeta^z(b)|^p+c_3^p)\\
&\le c_1^{-p}2^{p-1}\left(c_1^{-p}|\Im R_z^+(o_{b^+})|^{-p}+c_3^p\right). \qedhere
\end{align*}
\end{proof}

\section{AC spectrum for the unperturbed tree}\label{sec:acunper}

The aim of this section is to prove Theorems~\ref{thm:det0} and \ref{thm:det}.

Let $\mathbf{T}$ be a quantum tree of finite cone type, with the structure described in \S~\ref{sec:condefs}. Given $(v_-,v)\in B(\T_{b_o}^+)$, we denote
\[
\zeta^{z}(v) = \frac{G^{z}(v_-,v)}{G^{z}(v_-,v_-)} \,.
\]
This notation is simply analogous to the one introduced in Section~\ref{subsec:Or}, and does not mean that $\zeta^{z}$ is a function of the terminus alone. It simply means that each discrete edge in $\T_{b_o}^+$ can be specified by indicating the terminus alone. We also let $\zeta^z(t_{b_o})=\zeta^z(b_o)$.

Denote $\zeta_j^{z} = \zeta^{z}(v)$ if $\ell(v)=j$. Then \eqref{e:1} says that for each $j\in\mathfrak{A}^+$,
\begin{equation}\label{e:coneprelim}
\frac{1}{\zeta_j^{z}S_{z}(L_j)} + \sum_{k=1}^m M_{j,k} \frac{\zeta^{z}_k}{S_{z}(L_k)} = \sum_{k=1}^m M_{j,k} \frac{C_{z}(L_k)}{S_{z}(L_k)} + \frac{S_{z}'(L_j)}{S_{z}(L_j)} + \alpha_j\,.
\end{equation}
The matrix elements $M_{j,k}$ were defined in
\S\ref{sec:condefs}(b).
The system \eqref{e:coneprelim} is reminiscent of the finite system of equations that appears in the combinatorial case \cite{KLW2,AS4} for $\zeta_j^{z} = \zeta_{v_-}^{z}(v)$. In order to put it in a nicer form, we denote $h_j = S_{z}(L_j)\zeta_j^{z}$. Then we get the following system of polynomial equations:
\begin{equation}\label{eq:pols}
\sum_{k=1}^m \frac{M_{j,k}}{S_{z}^2(L_k)} h_k h_j -F_j(z) h_j + 1 = 0\,, \qquad j=1,\dots, m
\end{equation}
where $F_j(z) = \alpha_j + \sum_{k=1}^m M_{j,k} \frac{C_{z}(L_k)}{S_{z}(L_k)} + \frac{S_{z}'(L_j)}{S_{z}(L_j)}$.

An analogous system of equations involving the matrix $N=(N_{i,j})$ arises when considering cones in $\T_{b_o}^-$. We restrict ourselves to the above system; the other one is analyzed similarly.

We mention that a similar system of equations in a more special framework appeared recently in \cite[eq. (4.8)]{Car17}. In this case, one has $M_{j,j} =1$ for each $j$ and $M_{j,k}=2$ for $k\neq j$.

Our aim in the following is to control the values of $\zeta^{\lambda+\ii \eta}_j$ as $\eta \downarrow 0$. For the models \cite{Car97,Car17}, the $\zeta^{z}_j$ are uniformly bounded. The following simple criterion gives a sufficient condition for this to happen.  Note the condition $M_{j,j}>0$ below implies that each vertex of label $j$ has at least one offspring of its own type.  Later we will
relax that restriction.

\begin{lem} \label{lem:Mjj>0}
  Suppose $M_{j,j}>0$ for some $j$. Then $|\zeta_j^{z}|<1$ for any
  $z\in \C\setminus \R$. In fact, $|\zeta_j^{z}|^2<\frac{1}{M_{j,j}}$.
\end{lem}
This lemma parallels the combinatorial case \cite[Lemma 3]{KLW2}, see \cite[Lemma 3.9]{Car17} for a special case. 
\begin{proof}
Let $z\in \C^+$ and $b$ with $\ell(t_b)=j$. Then \eqref{e:ASW} becomes $\sum_{k=1}^m M_{j,k} \Im R_z^+(k) \le \frac{\Im R_z^+(j)}{|\zeta_j^z|^2}$, where $R_z^+(k):=R_z^+(o_e)$ if $\ell(t_e)=k$. The inequality is actually strict if $\Im z>0$, as seen from the proof of \eqref{e:ASW}. Thus, $|\zeta_j^z|^2 < \frac{\Im R_z^+(j)}{M_{j,j} \Im R_z^+(j)} =\frac{1}{M_{j,j}}$.

The case $\Im z<0$ can be adapted without difficulty, in this case $\Im R_z^+(o_e)$ should be replaced by $|\Im R_z^+(o_e)|$ in \eqref{e:ASW}.
\end{proof}

The lemma implies in particular that $|\zeta_j^{\lambda+\ii 0}| \le \frac{1}{M_{j,j}}$ for any $\lambda \in \R$.

There are many models of interest for which the condition of Lemma
\ref{lem:Mjj>0} is not satistfied, so we next consider the general
case.  Now the limit $\zeta_j^{\lambda+\ii 0}$ may no longer exist,
but we aim to show this problem can only occur on a discrete subset of
$\R$.

\begin{prp}\label{prp:Lang}
There is a discrete set $ \mathfrak{D}\subset \R$ such that, for all $j=1,\dots,m$, the solutions $h_j(\lambda+\ii\eta) = S_{\lambda+\ii\eta}(L_j) \zeta_j^{\lambda+\ii\eta}$ of \eqref{eq:pols} have a finite limit as $\eta \downarrow 0$ for all $\lambda\in \R\setminus  \mathfrak{D}$. The map $\lambda\mapsto S_{\lambda}(L_j)\zeta_j^{\lambda+\ii0}$ is continuous on $\R\setminus \mathfrak{D}$, and there is a discrete set $\mathfrak{D}'$ such that it is analytic on $\R\setminus (\mathfrak{D}\cup\mathfrak{D}')$.
\end{prp}
\begin{proof}
We follow the strategy in \cite[\S 4]{AS4}. The aim is essentially to decouple the system \eqref{eq:pols} and show that each $h_j$ satisfies an algebraic equation $Q_j(h_j)=0$. For this, we will use an algebraic tool from \cite{Lang}.

Let $\lambda_0\in \R$, and let $P_j(h_1,\dots,h_m) = \sum_{k=1}^m\frac{M_{j,k}}{S_{z}^2(L_k)} h_k h_j -F_j(z) h_j + 1 $. Clearly, $P_j \in K[h_1,\dots,h_m]$, where $K=\mathscr{K}_{\lambda_0}$ is the field of functions $f(z)$ possessing a convergent Laurent series $f(z)=\sum_{j=-n_0}^{\infty}a_j(\lambda_0) (z-\lambda_0)^j$ in some neighbourhood $N_{\lambda_0}\subset \C$ of $\lambda_0$.

Let $K'=\mathscr{J}_{\lambda_0}$ be the field of functions $f$ which are meromorphic on $N_{\lambda_0}\cap \C^+$ for some neighbourhood $N_{\lambda_0}$ of $\lambda_0$. Then $K'$ is an extension of $K$, and we know that $S_{z}(L_j)\zeta_j^{z}$ belongs to $K'$ (see Remark~\ref{rem:zeta}) and satisfy $P_j(S_{z}\zeta_1^{z},\dots,S_{z}\zeta_m^{z})=0$. Calculating the Jacobian $\left(\frac{\partial P_j}{\partial h_k}(h)\right)$, we find
\begin{displaymath}
  \frac{\partial P_j}{\partial h_k} = \left\{
    \begin{array}{ll}
      \frac{M_{j,k}}{S_z^2(L_j)}h_j, & k\neq j, \\
      \sum_{\ell=1}^m \frac{M_{\ell,j}}{S_z^2(L_\ell) }h_\ell +
      \frac{M_{j,j}}{S_z^2(L_j)} h_j - F_j(z), & k=j.
    \end{array}
\right.
\end{displaymath}
Let
\begin{displaymath}
  J^{z} = \left.\det\left(\frac{\partial P_j}{\partial h_k} \right)
  \right|_{h=S_{z}(L_k)\zeta_k^{z}}.
\end{displaymath}
We will show that $z\mapsto J^{z}$ is not the zero element of $K'$. For this, we first study the asymptotics of $J^{z}$ as $z \to -\infty$.

Take $z=-r^2$ with $r>0$ large. We remark that
\begin{equation}\label{e:asympto}
\lim_{r\to \infty} \frac{C_{-r^2}(L_k)}{rS_{-r^2}(L_k)} = 1 \qquad \text{and} \qquad \lim_{r\to \infty} \frac{S_{-r^2}'(L_k)}{r S_{-r^2}(L_k)} =1 \,.
\end{equation}
This follows from classical estimates \cite[Chapter 1]{PT87}. In fact,
\begin{displaymath}
S_{-r^2}(L) \approx \frac{\sin \ii rL}{\ii r} = \frac{\sh rL}{r}
,\qquad
C_{-r^2}(L) \approx \cos \ii rL = \ch rL \approx S_{-r^2}'(L).
\end{displaymath}
More precisely, we write
\begin{displaymath}
\frac{C_{-r^2}(L)}{rS_{-r^2}(L)} = \frac{\ch rL + R(r,L)}{\sh rL + r R'(r,L)},
\end{displaymath}
where $R(r,L) = C_{-r^2}(L) - \ch rL$ and $R'(r,L) = S_{-r^2}(L) - \frac{\sh rL}{r}$. By \cite[p. 13]{PT87}, $\frac{rR'(r,L)}{\sh rL} \to 0$ and $\frac{R(r,L)}{\sh rL} \to 0$ as $r\to \infty$. Since $\frac{\ch rL}{\sh rL} \to 1$, \eqref{e:asympto} follows. Hence,
\begin{equation}\label{Eq:DefCj}
  F_j(-r^2) \sim \al_j + r + r{\textstyle \sum_k M_{j,k}}
  \sim C_{j} r \quad \text{as } r\to \infty.
\end{equation}
On the other hand, since $h_j$ is Herglotz (see Remark \ref{rem:hj}), it has
a representation of the form \eqref{eq:herglotz}. If $t_0=\inf\sigma(\cH_{\mathbf{T}})$, we also know from Remark~\ref{rem:zeta} that $h_j(\lambda)$ is well-defined and real-valued for $\lambda<t_0$. By \cite[Theorem 3.23]{Teschl}, the measure $\mathfrak{m}$ is thus supported on $[t_0,\infty)$. Hence, for large $r$ (say $r^2>-t_0+1$),
\[
  h_j(-r^2) = -Ar^2 + B + \int_{t_0}^{\infty} \frac{1-r^2t}{t+r^2}\,
           \frac{\dd\mathfrak{m}(t)}{1+t^2},
\]
where we used that $h_j(-r^2)= \lim_{\eta \downarrow 0}h_j(-r^2+\ii \eta)$ and dominated convergence (recall that $\frac{\dd\mathfrak{m}(t)}{1+t^2}$ is a finite measure). Thus,
\[
\frac{h_j(-r^2)}{-r^2} = A-\frac{B}{r^2} + \int_{t_0}^{\infty} \frac{t-\frac{1}{r^2}}{t+r^2}\,\frac{\dd\mathfrak{m}(t)}{1+t^2} \,.
\]
Using dominated convergence again, we see that $h_j(-r^2)/(-r^2)\to A$ as $r\to\infty$.
This implies that
\begin{displaymath}
\frac{h_j(-r^2)}{S_{-r^2}^2(L_k)} = O(r^4 \ee^{-2rL_k}) = O(\ee^{-rL_k}).
\end{displaymath}
Therefore, recalling that the $C_j$ were defined in (\ref{Eq:DefCj}), we find that as $r\to \infty$,
\[
  J^{-r^2} =\det
  \begin{pmatrix}
    -C_{1} r + O(\ee^{-rL_1}) & O(\ee^{-rL_2}) & \cdots & O(\ee^{-rL_m}) \\
    O(\ee^{-rL_1}) & -C_2 r + O(\ee^{-rL_2}) & \cdots & O(\ee^{-rL_m}) \\
    \vdots & \vdots & \ddots & \vdots \\
    O(\ee^{-rL_1}) & O(\ee^{-rL_2}) & \cdots & -C_{m}r +O(\ee^{-rL_m})
  \end{pmatrix}
\]
Hence, $J^{-r^2} \sim (-1)^mC_1\cdots C_m r^m = Cr^m\neq 0$ for $r$ large enough. Since $z\mapsto J^{z}$ is holomorphic on $\C\setminus [a_0,\infty)$, it follows that it cannot vanish identically on any neighbourhood $N_{\lambda_0}\cap \C^+$. Hence, $J^{z}$ is not the zero element of $K'$.

It follows by \cite[Proposition VIII.5.3]{Lang} that each $S_{z}\zeta_j^{z}$ is algebraic over $K$. By the Newton-Puiseux theorem (see e.g. \cite[Theorem 3.5.2]{Sim15}), each $h_j$ thus has an expansion of the form
\begin{equation}\label{eq:puiseux}
h_j(z)=\sum_{n\geq m} a_n (z-\lambda_0)^{n/d}
\end{equation}
in some neighbourhood $N_{\lambda_0}$ of $\lambda_0$. Here $m\in \Z$, $d\in \N$, and the entire series $\sum_{n\geq 0} a_n z^n$ has a positive radius of convergence. In particular, $z \mapsto S_{z}\zeta_j^{z}$ is analytic near any $\lambda\in N_{\lambda_0}\setminus \{\lambda_0\}$. The set $ \mathfrak{D}$ corresponds to those $\lambda_0$ for which $m<0$ in the Newton-Puiseux expansion at $\lambda_0$, and the set $\mathfrak{D}'$ corresponds to those $\lambda_0$ for which $d>1$.
\end{proof}

Our next aim is to show that all WT functions have a positive imaginary part on most of the spectrum.

Let $\sigma_D$ be the union of the Dirichlet spectra:
\[
\sigma_D = \bigcup_{j=1}^m \{\lambda\in \R : S_{\lambda}(L_j)=0\} \,.
\]
We would like to index the WT functions by vertices, but the notation $R_z^+(v)$ is a bit ambiguous since $R_z^+(t_b)\neq R_z^+(o_{b^+})$ even if $t_b=o_{b^+}=v$. So we take the convention that
\begin{equation}\label{e:rv}
R_z^{\pm}(v) := \begin{cases} R_z^{\pm}(o_b) & \text{if } b=(v_-,v)\in \T_{b_o}^+,\\ R_z^{\pm}(t_b) &\text{if }b=(v,v_+)\in \T_{b_o}^-.\end{cases}
\end{equation}
Here $(t_{b_o})_-=o_{b_o}$ and $(o_{b_o})_+=t_{b_o}$. This keeps with the convention of \S~\ref{subsec:Or} of indexing functions $\psi(b)$ by their terminus\footnote{The notation is probably a bit awkward since for $v\in \T_{b_o}^+$, $v$ is the terminus of $b$ yet $R_z^{\pm}(v):=R_z^{\pm}(o_b)$. We stress however that $R_z^\pm$ is not a function of the vertex $o_b$ alone but depends on the whole directed edge $b$, so it should really be read as a function $\psi(b)$, which we index by the terminus.} on $\T_{b_o}^+$ and their origin on $\T_{b_o}^-$.

As there are finitely many types of $\zeta^z(b)$ for $b=(v_-,v)\in B(\T_{b_o}^+)$, we see by \eqref{e:zetawt} there are finitely many types of $R_z^+(o_b)$ (this may not be true for $R_z^-(o_b)$ for such $b$). We denote $R_{z}^+(j) := R_{z}^+(o_b)$ if $b=(v_-,v)\in \T_{b_o}^+$ and $\ell(v)=j\in\mathfrak{A}^+$. Similarly, we denote $R_{z}^-(k) = R_{z}^-(t_b)$ if $b=(v,v_+)\in\T_{b_o}^-$ and $\ell(v) = k\in\mathfrak{A}^-$.

By some abuse of notation, we assume the discrete sets $\mathfrak{D}$, $\mathfrak{D}'$ of Proposition~\ref{prp:Lang} are the same for the system analogous to \eqref{eq:pols} which involves the matrix $(N_{i,j})$.

\begin{rem} \label{rem:wtcones}
Denote $R_{\lambda}^{\pm}:=R_{\lambda+\ii0}^{\pm}$. Then the limits $R_{\lambda}^+(j)$ exist for $\lambda\in \R \setminus (\mathfrak{D}\cup \sigma_D)$ and $j\in\mathfrak{A}^+$. This follows from Proposition~\ref{prp:Lang} and \eqref{e:zetawt}, which implies that $R_{\lambda}^+(j) = \frac{\zeta_j^{\lambda}-C_{\lambda}(L_j)}{S_{\lambda}(L_j)}$.  Similarly, the limits $R_{\lambda}^-(k)$ exist for $k\in\mathfrak{A}^-$.

It follows that $R_{\lambda}^{\pm}(v)$ exist for $v\in \{o_{b_o},t_{b_o}\}$ and $\lambda\notin \mathfrak{D}\cup \sigma_D$. In fact, $\zeta^{\lambda}(b_o)=\frac{h_{j_o}}{S_{\lambda}(L_o)}$ for some $j_o\in\mathfrak{A}^+$, which exists by Proposition~\ref{prp:Lang}, so $R_{\lambda}^+(o_{b_o})$ exists by \eqref{e:zetawt}. Similarly the result for the $(N_{ij})$ system implies the existence of $\zeta^{\lambda}(\hat{b}_o)$ and $R_{\lambda}^-(t_{b_o})$. Finally if $t_{b_o}$ has type $j_o\in \mathfrak{A}^+$, then $R_{\lambda}^+(t_{b_o})=\sum_{k=1}^mM_{j_o,k}R_{\lambda}^+(k)-\alpha_{t_{b_o}}$ by \eqref{e:rdelta}, which exists by the previous paragraph. Similarly $R_\lambda^-(o_{b_o})=\sum_{k=1}^nN_{j'_o,k}R_{\lambda}^-(k)+\alpha_{o_{b_o}}$ exists. 

Proposition~\ref{prp:Lang} tells us moreover that $R_{\lambda+\ii 0}^{\pm}(v)$ are analytic on $\R \setminus (\mathfrak{D}\cup \mathfrak{D}' \cup \sigma_D)$. In particular, their zeroes do not accumulate. Hence, there is a discrete set $\mathfrak{D}''$ such that $R_{\lambda}^+(v) + R_{\lambda}^-(v) \neq 0$ for $\lambda\in \R \setminus (\mathfrak{D}\cup \mathfrak{D}'\cup \mathfrak{D}'' \cup \sigma_D)$ and $v\in \{o_{b_o},t_{b_o}\}$. We therefore define
\begin{equation*}
\mathfrak{D}_0 := \mathfrak{D} \cup \mathfrak{D}'\cup \mathfrak{D}'' \cup \sigma_D.
\end{equation*}
\end{rem}

We may actually generalize the result of Remark~\ref{rem:wtcones} as follows.

\begin{lem}\label{lem:revi}
\phantomsection
\begin{enumerate}[\rm (a)]
\item If $\lambda \notin\mathfrak{D}\cup\sigma_D$, then $R_{\lambda}^{\pm}(o_b)$ exists for all $b=(v_-,v)\in \T_{b_o}^+$ and $R_{\lambda}^{\pm}(t_b)$ exists for all $b=(w,w_+)\in\T_{b_o}^-$.
\item If moreover $\lambda\notin \mathfrak{D}_0$, then $G^{\lambda}(v,v)$ exists for any $v\in \T$, and $G^{\lambda}(v,v)\neq 0$.
\end{enumerate}
\end{lem}
\begin{proof}
By symmetry it suffices to prove (a) for $\T_{b_o}^+$. Consider $b=(v_-,v)\in \T_{b_o}^+$. We already know that $R_{\lambda}^+(o_b)=R_{\lambda}^+(j)$ exists from Remark~\ref{rem:wtcones}. Next, we show by induction that $\frac{1}{\zeta^{\lambda}(\widehat{b})S_{\lambda}(L_b)}$ is finite. Note that we already know $\frac{1}{\zeta^{\lambda}(\widehat{b_o})S_{\lambda}(L_o)}$ is finite by \eqref{e:r+-id}. So consider any oriented edge $b=(t_{b_o},v)$. Applying \eqref{e:1} to $\widehat{b}$ instead of $b$, we may express $\frac{1}{\zeta^z(\widehat{b})S_z(L_b)}$ in terms of some $C_z,S_z$ functions, plus $\zeta^z(b')$ for $b'\in \cN_{\widehat{b}}^+$. One of them is $\zeta^z(\widehat{b}_o)$, whose limit on the real axis exists from Remark~\ref{rem:wtcones}. The rest are precisely those with $b'\in \cN_{b_o}^+\setminus \{b\}$, which also exist by Proposition~\ref{prp:Lang}. Thus, $\frac{1}{\zeta^\lambda(\widehat{b})S_\lambda(L_b)}$ exists for any $b=(t_{b_o},v)$. By induction we get existence for any $(v_-,v)\in \T_{b_o}^+$. It follows from \eqref{e:r+-id} that $R_{\lambda}^-(o_b)$ exists for all such $b=(v_-,v)\in \T_{b_o}^+$.

(b) By Remark~\ref{rem:wtcones},
$R_{\lambda}^+(v) + R_{\lambda}^-(v) \neq 0$ for
$v=o_{b_o},t_{b_o}$. Using \eqref{e:greener}, this implies
$G^{\lambda}(v,v)$ exists. We now observe that
\begin{equation}\label{e:greenrecpu}
G^{z}(v_+,v_+) = S_{z}(L_v)\zeta^{z}(v,v_+) + \zeta^{z}(v,v_+)^2G^{z}(v,v)\,.
\end{equation}
In fact, by \eqref{e:zetainv},
\begin{displaymath}
G^{z}(v_+,v_+) = \frac{\zeta^{z}(v,v_+)}{\zeta^{z}(v_+,v)}G^{z}(v,v) = \Big(\zeta^{z}(v,v_+)+\frac{S_{z}(L_v)}{G^{z}(v,v)}\Big)\zeta^{z}(v,v_+)G^{z}(v,v)
\end{displaymath}
as claimed. Using Proposition~\ref{prp:Lang}, we thus deduce the existence of $G^{\lambda}(w,w)$ for all $w\in \T_{b_o}^+$. Similarly the existence of $w\in \T_{b_o}^-$ follows from the analog of Proposition~\ref{prp:Lang} with the $(N_{i,j})$ system. Finally using \eqref{e:greener} we see that (a) implies $G^{\lambda}(v,v)\neq 0$.
\end{proof}

We now observe that under \textbf{(C1*)}, all the WT functions are related as follows:

\begin{lem}\label{lem:samewt}
Suppose $\mathbb{T}$ satisfies \emph{\textbf{(C1*)}} and let $\lambda\in \R \setminus(\mathfrak{D}\cup \sigma_D)$.
\begin{enumerate}[\rm (i)]
\item If $\Im R_{\lambda+\ii 0}^\pm(j) = 0$ for some $j\in\mathfrak{A}^\pm$, then $\Im R_{\lambda+\ii 0}^\pm(j) = 0$ for all $j\in \mathfrak{A}^\pm$.
\item Assume $\lambda\in \R \setminus\mathfrak{D}_0 $.

\noindent If $\Im R_{\lambda+\ii 0}^+(j) = 0$ for some $j\in \mathfrak{A}^+$ and $\Im R_{\lambda+\ii 0}^-(k) = 0$ for some $k\in \mathfrak{A}^-$, then $\Im G^{\lambda+\ii0}(v,v) = 0$ and $\Im R_{\lambda+\ii 0}^{\pm}(v) = 0$ for all $v\in \T$.

The same conclusion holds if $\Im G^{\lambda+\ii0}(w,w)=0$ for some $w\in \T$.
\item For any $v\in \T$, we have
\[
\sigma(\cH_0)\setminus \mathfrak{D}_0 = \overline{\{ \lambda\in \R\setminus \mathfrak{D}_0 : \Im G^{\lambda+\ii0}(v,v)>0\}}\setminus\mathfrak{D}_0 \, .
\]
\end{enumerate}
\end{lem}
\begin{proof}
  We first note that $\zeta_j^{\lambda}\neq 0$, due to the relation
  \begin{displaymath}
\zeta_j^{\lambda} = \frac{h_j}{S_{\lambda}(L_j)} = \frac{-1/S_{\lambda}(L_j)}{\sum_{k=1}^m\frac{M_{j,k}}{S_{\lambda}^2(L_k)} h_k-F_j(\lambda)}
\end{displaymath}
and the fact that the $h_k$ are finite.

Denote $R_k^{\lambda}=R_{\lambda+\ii0}^+(k).$
  Suppose that $\Im R_k^{\lambda}>0$ for some $k\in\mathfrak{A}^+$ and let $l\in\mathfrak{A}^+$. Then by \textbf{(C1*)}, $(M^n)_{l,k}\ge 1$, so if $v$ has label $l$ and $w$ has label $k$, there is a path $(u_0,\dots,u_r)$ with $u_1=v$ and $u_r=w$. Denote $b_j = (u_{j-1},u_j)$. Then applying \eqref{e:ASW} repeatedly,
\begin{align*}
	\Im R_l^{\lambda} = \Im R^{\lambda}(o_{b_1}) &\ge \sum_{(e_2;e_r)}|\zeta^{\lambda}(e_1)\cdots \zeta^{\lambda}(e_{r-1})|^2\Im R^{\lambda}(o_{e_r}) \\
	& \ge |\zeta^{\lambda}(b_1)\cdots \zeta^{\lambda}(b_{r-1})|^2\Im R^{\lambda}(o_{b_r})  > 0,
\end{align*}
where the sum runs over all $(r-1)$-paths $(e_2;e_r)$ outgoing from $b_1$, and the last inequality holds because $\Im R^{\lambda}(o_{b_r}) = \Im R_k^{\lambda}>0$ and all $\zeta_j^{\lambda}\neq 0$.

So under \textbf{(C1*)}, if $\Im R_j^{\lambda}>0$ for some $j\in \mathfrak{A}^+$, then $\Im R_k^{\lambda}>0$ for all $k\in\mathfrak{A}^+$. So if $\Im R_j^{\lambda}=0$ for some $j\in\mathfrak{A}^+$, then it must be zero for all $j\in\mathfrak{A}^+$.

The proof for $R_{\lambda+\ii 0}^-$ is the same.

For (ii), say $t_{b_o}$ has type $j_o$. Since
$R_{\lambda}^+(t_{b_o}) = \sum_{k=1}^m M_{j_o,k} R_{\lambda}^+(k) -
\alpha_{t_{b_o}}$ by \eqref{e:rdelta}, we get $\Im R_{\lambda}^+(t_{b_o}) = 0$ by
(i). Using \eqref{e:r+-id}, this implies
$\Im \zeta^{\lambda}(b_o) = 0$, which by \eqref{e:zetawt} implies that
$\Im R_{\lambda}^+(o_{b_o}) = 0$. Similarly, if $o_{b_o}$ has type
$j$, then
$R_{\lambda}^-(o_{b_o}) = \sum_{k=1}^n N_{j,k} R_{\lambda}^-(k) +
\alpha_{o_{b_o}}$, so $\Im R_{\lambda}^-(o_{b_o}) = 0$ by (i),
also implying $\Im R_{\lambda}^-(t_{b_o})=0$ \textit{via}
\eqref{e:r+-id}, \eqref{e:zetawt}. Now
$R_{\lambda}^+(v) + R_{\lambda}^-(v) \neq 0$ for
$v=o_{b_o},t_{b_o}$. Using \eqref{e:greener}, this implies
$G^{\lambda}(v,v)$ exists, and $\Im G^{\lambda}(v,v) = 0$ for
$v=o_{b_o},t_{b_o}$. Since $\Im G^{\lambda}(t_{b_o},t_{b_o})=0$ and
$\Im \zeta^{\lambda}(t_{b_o},v_+)=\Im \zeta^{\lambda}_j=0$, then using \eqref{e:greenrecpu}, we get
$\Im G^{\lambda}(v_+,v_+)=0$ for any $v_+\in\cN_{t_{b_o}}^+$, so
$\Im G^{\lambda}(w,w)=0$ for all $w\in \T_{b_o}^+$ by
induction. Similarly, we may use
$\Im G^{\lambda}(o_{b_o},o_{b_o}) = 0$ along with \eqref{e:zetainv} to
deduce that $\Im G^{\lambda}(v,v) = 0$ for all $v\in \T_{b_o}^-$. This
proves claim (ii) for the Green function.

Next, if $v\in \T_{b_o}^+$, we know that $\Im R_{\lambda}^+(v) = \Im R_{\lambda}^+(j) = 0$. By Lemma~\ref{lem:revi} $G^{\lambda}(v,v)$ exists, and we showed $\Im G^{\lambda}(v,v)=0$. Using \eqref{e:greener}, it follows that $\Im R_{\lambda}^-(v) = 0$. Hence, $\Im R_{\lambda}^{\pm}(v) = 0$ for $v\in \T_{b_o}^+$. The claim for $\T_{b_o}^-$ follows similarly.

Now suppose that $\Im G^{\lambda}(w,w)=0$ for some $w\in \T$. By symmetry we may assume $w\in\T_{b_o}^+$. Recall that $G^{\lambda}(v,v),R_{\lambda}^\pm(v)$ exist by Lemma~\ref{lem:revi}. By \eqref{e:greener}, we get $\Im R_{\lambda}^+(w)=\Im R_{\lambda}^-(w)=0$. Consider $w_+\in\cN_w^+$. By Proposition~\ref{prp:Lang}, $\zeta^\lambda(w,w_+)$ exists, so using \eqref{e:zetawt}, we get $\Im \zeta^{\lambda}(w,w_+)=0$, hence $\Im G^\lambda(w_+,w_+)=0$. On the other hand $\Im R_\lambda^+(w)=0$ implies $\Im R_\lambda^+(w_-)=0$ by (i), so we similarly get $\Im \zeta^\lambda(w_-,w)=0$ and $\Im G^\lambda(w_-,w_-)=0$. This shows that $\Im G^\lambda(v,v)=0$ for all $v\in\T_{b_o}^+$ and also for $v=o_{b_o}$, since by definition \eqref{e:rv}, $\Im R_\lambda^+(o_{b_o})=:\Im R_{\lambda}^+(j_o)=0$ for $v=t_{b_o}$, if $\ell(t_{b_o})=j_o$. Now if $u\in \cN_{o_{b_o}}^-$, then as in \eqref{e:greenrecpu} we have $G^\lambda(u,u)=S_\lambda(L_u)\zeta^\lambda(o_{b_o},u)+\zeta^\lambda(o_{b_o},u)^2G^\lambda(o_{b_o},o_{b_o})$, but $\Im R_\lambda^-(o_{b_o})=0$ implies $\sum_{b^-\in \cN_{b_o}^-}\Im R_\lambda^-(t_{b^-})=\Im R_\lambda^-(o_{b_o})=0$, so $\Im R_\lambda^-(t_{b^-})=0$ for each $b^-$, so $\Im R_\lambda^-(k)=0$ for all $k\in\mathfrak{A}^-$ by (i) and we deduce again that $\Im G^\lambda(u,u)=0$ for $u\in \cN_{o_{b_o}}^-$, hence for all $u\in\T_{b_o}^-$ by induction. 

Finally, to prove (iii), recall that if $E_{\cH_0}$ is the projection-valued measure $E_{\cH_0}(J) = \chi_J(\cH_0)$, then $\sigma(\cH_0) = \supp E_{\cH_0}$. Moreover, $E_{\cH_0}(J)=0$ if and only if $\mu_f(J)=0$ for all $f\in L^2(\mathbf{T})$, where $\mu_f(J) = \langle f, \chi_J(\cH_0) f\rangle$. By \cite[Lemma 3.13]{Teschl}, we know that $\supp \mu_f = \overline{\{\lambda\in \R : \Im \langle f,G^{\lambda} f\rangle >0\}}$. Since $\mathfrak{D}_0$ is a discrete set, we deduce that $\supp \mu_f \setminus \mathfrak{D}_0 = \overline{\{\lambda\in \R\setminus \mathfrak{D}_0 : \Im \langle f,G^{\lambda} f\rangle >0\}}\setminus \mathfrak{D}_0$.

Let $\lambda_0\in \R \setminus \mathfrak{D}_0$ and suppose there is $\eps$ such that $\Im G^{\lambda}(v,v)=0$ for all $I:=(\lambda_0-\eps,\lambda_0+\eps)$. Then $\Im R_{\lambda}^{\pm}(w)=0$ for all $w$ by (ii), so $\Im \langle f,G^{\lambda} f\rangle=0$ for any $f$ by Lemma~\ref{lem:ASWrep}. Thus, $\mu_f(I)=0$ for all $f\in L^2(\mathbf{T})$, so $E_{\cH_0}(I)=0$ and thus $\lambda_0\notin \sigma(\cH_0)$.

Conversely, fix $\lambda\in \R\setminus \mathfrak{D}_0$ and $v\in \T$. If $\Im G^{\lambda}(v,v)>0$, then $\Im R_{\lambda}^+(v)>0$ or $\Im R_{\lambda}^-(v)>0$, say the former holds and let $v=o_e$. Then by Lemma~\ref{lem:ASWrep}, $\Im \langle f, G^{\lambda} f\rangle >0$ for $f=\phi_{\lambda;e}^-$, since $g_{\phi_{\lambda;e}^-}^-(\lambda) \ge \frac{1}{6} \|\phi_{\lambda;e}^-\|^4>0$. In fact, $\|\phi_{\lambda;e}^-\|=0$ would imply $C_{\lambda}(x)=R_{\lambda}^-(o_e)S_{\lambda}(x)$ for all $x\in e$, contradicting the fact that $C_{\lambda}$ and $S_{\lambda}$ are linearly independent. We thus showed that $\overline{\{\lambda\in \R\setminus \mathfrak{D}_0:\Im G^{\lambda}(v,v)>0\}}\setminus \mathfrak{D}_0 \subseteq \supp \mu_f \setminus \mathfrak{D}_0 \subseteq \sigma(\cH_0)\setminus \mathfrak{D}_0$.
\end{proof}

\begin{lem}\label{lem:cuteproperties}
If $\mathbb{T}$ satisfies \emph{\textbf{(C1*)}}, then:
\begin{enumerate}[\rm(i)]
\item For any $j\in\mathfrak{A}^+$, $k\in\mathfrak{A}^-$, the map $\sigma(\cH)\setminus  \mathfrak{D}_0\ni \lambda \mapsto \Im R_{\lambda+\ii0}^+(j) + \Im R_{\lambda+\ii 0}^-(k)$ has a discrete set of zeroes. The same holds for $\sigma(\cH)\setminus \mathfrak{D}_0\ni \lambda \mapsto \Im G^{\lambda+\ii0}(v,v)$, for any $v\in\T$.
\item $\sigma(\cH)$ is a union of closed intervals and isolated points, $\bigcup_r I_r \cup \mathfrak{P}$. The limits $G^{\lambda+\ii0}(v,v)$ exist in the interior $\mathring{I}_r$ and satisfy $\Im G^{\lambda+\ii0}(v,v)>0$, for any $v\in\T$.
\item The spectrum of $H_{\mathbf{T}}$ is purely absolutely continuous in any compact subset $K\subset \mathring{I}_r$.
\end{enumerate}
\end{lem}
\begin{proof}
We know that $f(\lambda)= R_{\lambda}^+(j)+ R_{\lambda}^-(k)$ is analytic on $\R\setminus\mathfrak{D}_0$, so if $\Im f(\lambda_0)=0$ for some $\lambda_0\notin \mathfrak{D}_0$, we may expand $\Im f(\lambda) = \sum_{n\ge 0} b_n(\lambda-\lambda_0)^n$ for $\lambda\in (\lambda_0-\eps,\lambda_0+\eps)$, where $b_n=\Im a_n$ and $(a_n)$ are the coefficients for $f(\lambda)$. Suppose $\lambda_0\in \sigma(\cH)\setminus\mathfrak{D}_0$. If all $b_n=0$ then $\Im f$ is identically zero on $(\lambda_0-\eps,\lambda_0+\eps)$. In view of Lemma~\ref{lem:samewt} (ii)-(iii), this contradicts that $\lambda_0\in \sigma(\cH)$. Hence let $k$ be the smallest index with $b_k\neq 0$. Then $\Im f(\lambda) = (\lambda-\lambda_0)^kg(\lambda)$, where $g(\lambda)=\sum_{n\ge 0}b_{n+k}(\lambda-\lambda_0)^n$. Clearly $g(\lambda_0)=b_k\neq 0$ and $g$ is continuous, so we may find $\eps'\le\eps$ such that both $(\lambda-\lambda_0)^k$ and $g(\lambda)$ are nonzero on $(\lambda_0-\eps',\lambda_0+\eps')\setminus \{\lambda_0\}$. This shows that $\lambda_0$ is an isolated zero of $\Im f$, as required.

This proves the first part of (i). For the second part, suppose $\Im G^{\lambda}(w,w)=0$ for some $w\in \T$. By Lemma~\ref{lem:samewt}, this implies $\Im R_{\lambda}^+(j)+\Im R_{\lambda}^-(k)=0$. Hence, $\lambda$ must lie in the preceding discrete set of zeroes.

For (ii), recall that $\sigma(\cH_0)\setminus \mathfrak{D}_0 = \overline{\{ \lambda\in \R\setminus \mathfrak{D}_0 : \Im G^{\lambda}(v,v)>0\}}\setminus\mathfrak{D}_0$ by Lemma~\ref{lem:samewt}. By Proposition~\ref{prp:Lang} and \eqref{e:zetawt}, we know $\lambda \mapsto R_{\lambda}^{\pm}(v)$ is continuous, so $\R\setminus\mathfrak{D}_0 \ni \lambda \mapsto G^{\lambda}(v,v)$ is continuous by \eqref{e:greener}. Hence, $\{\lambda\in\R\setminus \mathfrak{D}_0 : \Im G^{\lambda}(v,v)>0\}$ is a  union of intervals $\bigcup_r J_r$ which is independent of $v$ by Lemma~\ref{lem:samewt}. We take $I_r$ as the closure of $J_r$ and $\mathfrak{P} = \sigma(H)\cap \mathfrak{D}_0$.

Finally, if $K$ is a compact subset of $\mathring{I}_r$, we know that $G^{\lambda}(v,v)$ is uniformly bounded, and the same holds for $R_{\lambda}^{\pm}(v)$. In particular, if $v=o_e$ and $\psi$ is supported in $e$, we get using respresentation \eqref{e:othergreen} along with \eqref{e:greener} that $\sup_{\lambda\in K}|\langle \psi, G^{\lambda} \psi\rangle|<\infty$. The claim follows by the density of the linear span of such $\psi$.
\end{proof}

This completes the proof of Theorem~\ref{thm:det0}. We next move to Theorem~\ref{thm:det}.

\begin{rem}\label{rem:c1c1toile}
Condition \textbf{(C1)} implies \textbf{(C1*)}. In fact, as remarked in \cite{AS4}, all cone types are indexed by the directed edges of the finite graph $\mathbf{G}$.
If we consider the universal cover $\mathbf{T}$ rooted at the midpoint $o$ of some $b_o\in B(\mathbf{G})$ (here $o$ is not viewed as an added vertex, just a reference point), this means that the type of each vertex $v\in \T$ is determined by a directed edge, so there are at most $|B(\mathbf{G})|$ types.
 By \cite[Lemma 3.1]{OW07}, we know the non-backtracking matrix of $B(\mathbf{G})$ is irreducible. This implies that if $\mathbf{T}$ is considered in the \emph{twisted} view, and if $M$ is the single matrix over some alphabet $\mathfrak{A}$ encoding all cone types, then $M$ satisfies: for any $k,l\in\mathfrak{A}$, there is $n(k,l)$ such that $(M^n)_{k,l}\ge 1$. In particular, \textbf{(C1*)} holds if we take the matrices $\tilde{M},N$ encoding the types in $\T_{b_o}^+$ and $\T_{b_o}^-$, respectively.
\end{rem}

\begin{proof}[Proof of Theorem~\ref{thm:det}]
Since \textbf{(C1)} implies \textbf{(C1*)}, we already know that $\sigma(\cH_0)$ has the structure given in Theorem~\ref{thm:det0}. Let $\lambda\in \mathring{I}_r$ be in the interior of an AC band.

Within the twisted view, all vertices are offspring of $o$ and we deal with the single, combined alphabet $\mathfrak{A}$. Under the stronger assumption \textbf{(C1)}, we know the larger matrix $M$ is irreducible. Consequently, if we suppose that $\Im R_{\lambda+\ii0}^+(j)=0$ for some $j\in \mathfrak{A}$, then the statement in Lemma~\ref{lem:samewt} (i) now implies that $\Im R_{\lambda+\ii0}^+(j)=0$ for all $j\in \mathfrak{A}$.

Now let $v\in \T$. We know $\Im G^{\lambda}(v,v)>0$, so by \eqref{e:greener}, $\Im R_\lambda^+(v)+\Im R_\lambda^-(v)>0$, so either $\Im R_\lambda^+(v)>0$ or $\Im R_{\lambda}^-(v)>0$ by the Herglotz property. In the former case we are done. Suppose that $\Im R_\lambda^+(v)=0$. Say $v=t_b$ for some $b\in B(\T)$ and $\ell(v)=j$. Then $0=\Im R_\lambda^+(t_b) = \sum_{k=1}^m M_{j,k}\Im R_\lambda^+(k)$ implies that $\Im R_\lambda^+(k)=0$ for some, hence all, $k\in\mathfrak{A}$. But by \eqref{e:zetawt}, $R_{\lambda}^-(t_b) = R_{\lambda}^+(o_{\widehat{b}})+\frac{C_{\lambda}(L_b)-S_\lambda'(L_b)}{S_\lambda(L_b)}$, so $\Im R_\lambda^-(t_b) = \Im R_\lambda^+(o_{\widehat{b}})$. As mentioned in Remark~\ref{rem:c1c1toile}, under \textbf{(C1)}, all cone types are indexed by the directed edges of $\mathbf{G}$, in particular $\T_{\widehat{b}}^+$ is one of the finitely many nonisomorphic cones\footnote{This property is why we need \textbf{(C1)}; it is not necessarily true under \textbf{(C1*)}. cf. footnote in \S~\ref{sec:condefs}.}. In other words, $\Im R_\lambda^+(o_{\widehat{b}}) = \Im R_\lambda^+(r)$ for some $r\in\mathfrak{A}$. Hence, $\Im R_\lambda^-(t_b)=0$. We thus get $\Im R_\lambda^+(v)+\Im R_\lambda^-(v)=0$, a contradiction. Thus, $\Im R_\lambda^+(v)>0$.
\end{proof}

\section{Examples of nontrivial spectrum}\label{sec:nontrivialspec}

For Theorem~\ref{thm:det} to be interesting, we'll need to know that $\sigma(\cH_0)$ is not reduced to the isolated points $\mathfrak{P}$. Our aim in this section is to give some examples in which this can be proved. We believe the phenomenon to be true for a wider class of examples.

\medskip

\subsection{Equilateral trees} Let $G$ be a discrete graph of minimal degree $\ge 2$ and $\T=\widetilde{G}$ its universal cover. We know from \cite[Section 1.6]{BSV17} that the spectrum of the adjacency matrix $\sigma(\cA_{\T})$ has a continuous part. Actually their argument remains valid for the normalized adjacency matrix $Pf(x) = \frac{1}{d(x)} (\cA f)(x)$ (and also if we add potentials). Consequently, using \cite[Theorem 3.18]{BGP08}, the induced quantum tree with equilateral edge length, identical symmetric potentials, and identical coupling constants, will also have some continuous spectrum. Using Theorem~\ref{thm:det}, we can now conclude:

\smallskip

\emph{If $G$ is a graph of degree $\ge 2$, if $\T=\widetilde{G}$ is its universal cover, and we endow each edge of $\T$ with the same length $L$ and potential $W$, and each vertex with the same coupling constant $\alpha$, then $\sigma(\cH_{\mathbf{T}})$ consists of non-empty bands of purely absolutely continuous spectrum, and possibly some isolated eigenvalues.}

\smallskip

This generalizes the case of regular trees previously considered in \cite{Car97,Solomyak}.

We may easily extend this to graphs with several lengths which are rationally dependent. More precisely, if in $\mathbf{G}$, we have $L_j = n_j L$ for some $n_j\in\N^{\ast}$, add $n_j$ vertices of degree $2$ to the edge $e_j$, with Kirchhoff-Neumann conditions. Then using the previous claim, we see that $\mathbf{T}$ also has nontrivial AC spectrum in this case.

\subsection{An argument of Bordenave-Sen-Vir\'ag} We now consider the non-equilateral case. For this, we start by adapting an argument from \cite{BSV17} to quantum graphs.

We begin with some definitions, which appear in a more general framework in \cite{BSV17}.

Let $G$ be a discrete graph and $\T=\widetilde{G}$ its universal cover.

A \emph{labeling} (or \emph{colouring}) of the vertices of $\T$ is a
map $\eta : V(\T)\to \Z$. With respect to a given labeling, we call a
vertex $v$:
\begin{enumerate}[\rm (a)]
\item \emph{prodigy} if it has a neighbour $w$ with $\eta(w)<\eta(v)$ and such that all other neighbours of $w$ also have label less than $\eta(v)$,
\item \emph{level} if it is not a prodigy and if all of its neighbours have the same or lower labels,
\item \emph{bad} if it is neither prodigy nor level.
\end{enumerate}

The tree $\T=\widetilde{G}$ is equipped with a natural unimodular\footnote{A
  measure is unimodular if it invariant under the moving of the root. More precisely, it should satisfy $\int \sum_{o'\sim o} f([G,o,o'])\,\dd\mathbb{P}([G,o])=\int\sum_{o'\sim o}f([G,o',o])\,\dd\mathbb{P}([G,o])$. See
  \cite[\S1.4]{BSV17} for details.}
measure on the space of rooted graphs, namely
$\prob = \frac{1}{|G|} \sum_{x\in G} \delta_{[\widetilde{G},\tilde{x}]}$.

We say the labeling $\eta$ on $\T$ is \emph{invariant} if there exists a unimodular probability measure on the set of coloured rooted graphs, which is concentrated on $\{[\T,v,\eta]\}_{v\in\T}$. See e.g. \cite[Appendix A]{AS2} for some background on coloured rooted graphs.

Let $S\subset \ell^2(\widetilde{G})$ be a subspace and let $P_S$ be the orthogonal projection onto $S$. We say that $S$ is \emph{invariant} if $P_S(gv,gw)=P_S(v,w)$ for any $g\in \Gamma$, where $\Gamma$ is the group of covering transformations with $\widetilde{G}/ \Gamma \equiv G$.

Given an invariant subspace $S\subset \ell^2(\widetilde{G})$, we define its \emph{von-Neumann dimension} by
\[
\dim_{\mathrm{VN}} S = \expect_{\prob} [\langle \delta_o,P_S \delta_o\rangle] = \frac{1}{|G|}\sum_{x\in G} P_S(\tilde{x},\tilde{x}) \,.
\]

A \emph{line ensemble} in $\T$ is a disjoint union of bi-infinite lines $(l_i)$. More precisely, $\cL:V(\T)\times V(\T)\to \{0,1\}$ is a line ensemble if:
\begin{itemize}
\item $\cL(u,v)=0$ if $\{u,v\}\notin E(\T)$,
\item $\cL(u,v)=\cL(v,u)$,
\item for any $v\in V(\T)$, we have $\sum_u \cL(u,v) \in \{0,2\}$.
\end{itemize}

Abusing notation, we then let $\cL = \{e: \cL(e)=1\}$, which gives a subgraph consisting of disjoint lines.

We say a line ensemble $\cL$ is \emph{invariant} if there exists a unimodular probability measure on the space of weighted rooted graphs, which is concentrated on $\{[\T,v,\cL]\}_{v\in\T}$.

We say that $\T$ is \emph{Hamiltonian} if there exists an invariant line ensemble $\cL$ that contains the root with probability $1$.

\begin{rem}\label{rem:hami}
Recall that a finite graph $G$ is Hamiltonian if there is a cycle in $G$ which visits each vertex exactly once. If $G$ is Hamiltonian, then $\widetilde{G}$ is Hamiltonian in the above sense. In fact, if $C=(x_0,\dots,x_m)$ is a cycle in $G$, then its lift to $\widetilde{G}$ is a line ensemble $\cL$ which generally consists of a disjoint union of countable lines $(l_i)$, where $l_i=(\dots,\tilde{x}_0,\dots,\tilde{x}_m,\tilde{x}_0,\dots)$ (see Figure \ref{fig:lift}). Since it is a lift, this line ensemble is invariant. More precisely, if $[H,v,R]$ denotes an equivalence class of graph $H$ with root $v$ and edge weight $R(e)$ for $e\in E(H)$, then $[\widetilde{G},v,\cL] = [\widetilde{G},gv,\cL]$ for any covering transformation $g$. This by  definition of the universal cover and $\cL$. So the measure $\tilde{\prob} = \frac{1}{|G|}\sum_{x\in G} \delta_{[\widetilde{G},\tilde{x},\cL]}$ is well-defined and unimodularity follows from $\sum_{(x,y)\in B(G)} f(x,y) = \sum_{(x,y)\in B(G)} f(y,x)$.

\begin{figure}[h]
  \centering
  \definecolor{CB_blue}{HTML}{a6cee3}
  \definecolor{CB_Blue}{HTML}{1f78b4}
  \definecolor{CB_green}{HTML}{b2df8a}
  \definecolor{CB_Green}{HTML}{33a02c}
  \definecolor{CB_pink}{HTML}{fb9a99}
  \definecolor{CB_Pink}{HTML}{e31a1c}
  \definecolor{CB_orange}{HTML}{fdbf6f}
  \definecolor{CB_Orange}{HTML}{ff7f00}
  \definecolor{CB_purple}{HTML}{cab2d6}
  \definecolor{CB_Purple}{HTML}{6a3d9a}
  \definecolor{CB_Brown}{HTML}{a65628}
  \begin{tikzpicture}
    [vertex/.style={fill, circle, inner sep=2.5pt, outer sep=0pt}]
    %
    %
    \begin{scope}[xshift=-5cm]
    \draw[ultra thick]
    (-1,1) node[vertex , label=below:$1$] (K1) {} --
    (0,1) node[vertex , label=below:$2$] (K2) {} --
    (-1,0) node[vertex , label=below:$3$] (K3) {} --
    (0,0) node[vertex, label=below:$4$] (K4) {} --
    (-1,-1) node[vertex , label=below:$5$] (K5) {} --
    (0,-1) node[vertex , label=below:$6$] (K6) {} --
    (-1,-2) node[vertex , label=below:$7$] (K7) {} --
    (0,-2) node[vertex , label=below:$8$] (K8) {};
    \draw[ultra thick] (K8) .. controls ++(1,-1) and (-1,-3) .. (-1.5,-2.5)
    .. controls (-2,-2) and (-1.5,0.5) .. (K1);
    \draw (K2) .. controls (1,2) and (-1,2) .. (-1.5, 1.5) .. controls
    (-2,1) and (-1.5,-1.5) .. (K7);
    \draw (K8) -- (K3);
    \draw (K8) -- (K5);
    \draw (K6) -- (K1);
    \draw (K6) -- (K3);
    \draw (K4) -- (K1);
    \draw (K4) -- (K7);
    \draw (K2) -- (K5);
    \node at (-0.5,-3.5) {Hamiltonian cycle in $K_{4,4}$.};
  \end{scope}
  %
  %
  \node[vertex, label=right:$1$] (v0) at (0,0) {};
  \node[vertex, label=below:$2$] (v1) at (45:1) {};
  \node[vertex, label=left:$8$] (v2) at (-45:1) {};
  \node[vertex, label=above:$6$] (v3) at (200:1) {};
  \node[vertex, label=below:$4$] (v4) at (130:1) {};
  %
  %
  \path (v1) ++(-20:1) node[vertex, label=below:$3$] (v5) {};
  \path (v1) ++(30:1) node[vertex, label=below:$5$] (v16) {};
  \path (v1) ++(60:1) node[vertex, label=below:$7$] (v15) {};
  \path (v15) ++(130:1) node[vertex, label=right:$8$] (v45) {};
  \path (v15) ++(90:1) node[vertex, label=right:$6$] (v46) {};
  \path (v15) ++(55:1) node[vertex, label=right:$4$] (v47) {};
  \path (v16) ++(40:1) node[vertex, label=right:$4$] (v48) {};
  \path (v16) ++(20:1) node[vertex, label=right:$6$] (v49) {};
  \path (v16) ++(0:1) node[vertex, label=right:$8$] (v50) {};
  \path (v5) ++(-60:1) node[vertex, label=below:$4$] (v17) {};
  \path (v5) ++(20:1) node[vertex, label=below:$8$] (v51) {};
  \path (v5) ++(-25:1) node[vertex, label=below:$6$] (v52) {};
  \path (v51) ++(-10:1) node[vertex, label=right:$1$] (v58) {};
  \path (v51) ++(10:1) node[vertex, label=right:$5$] (v57) {};
  \path (v51) ++(30:1) node[vertex, label=right:$7$] (v56) {};
  \path (v52) ++(0:1) node[vertex, label=right:$1$] (v59) {};
  \path (v52) ++(-30:1) node[vertex, label=right:$5$] (v60) {};
  \path (v52) ++(-55:1) node[vertex, label=right:$7$] (v61) {};
  %
  %
  \path (v2) ++(-20:1) node[vertex, label=above:$7$] (v6) {};
  \path (v2) ++(-55:1) node[vertex, label=above:$5$] (v7) {};
  \path (v2) ++(-90:1) node[vertex, label=left:$3$] (v8) {};
  \path (v6) ++(-10:1) node[vertex, label=right:$6$] (v18) {};
  \path (v6) ++(-30:1) node[vertex, label=right:$4$] (v19) {};
  \path (v6) ++(-60:1) node[vertex, label=right:$2$] (v20) {};
  \path (v7) ++(-45:1) node[vertex, label=right:$6$] (v21) {};
  \path (v7) ++(-65:1) node[vertex, label=below:$4$] (v22) {};
  \path (v7) ++(-85:1) node[vertex, label=below:$2$] (v23) {};
  \path (v8) ++(-70:1) node[vertex, label=below:$6$] (v24) {};
  \path (v8) ++(-90:1) node[vertex, label=below:$4$] (v25) {};
  \path (v8) ++(-110:1) node[vertex, label=below:$2$] (v26) {};
  \path (v22) ++(-10:1) node[vertex, label=right:$7$] (v53) {};
  \path (v22) ++(-35:1) node[vertex, label=right:$3$] (v54) {};
  \path (v22) ++(-60:1) node[vertex, label=right:$1$] (v55) {};
  %
  %
  \path (v3) ++(-80:1) node[vertex, label=right:$7$] (v9) {};
  \path (v3) ++(-120:1) node[vertex, label=right:$5$] (v10) {};
  \path (v3) ++(-155:1) node[vertex, label=right:$3$] (v11) {};
  \path (v9) ++(-50:1) node[vertex, label=below:$8$] (v27) {};
  \path (v9) ++(-70:1) node[vertex, label=below:$4$] (v28) {};
  \path (v9) ++(-90:1) node[vertex, label=below:$2$] (v29) {};
  \path (v10) ++(-75:1) node[vertex, label=below:$8$] (v30) {};
  \path (v10) ++(-95:1) node[vertex, label=below:$4$] (v31) {};
  \path (v10) ++(-115:1) node[vertex, label=below:$2$] (v32) {};
  \path (v11) ++(-120:1) node[vertex, label=left:$8$] (v33) {};
  \path (v11) ++(-150:1) node[vertex, label=left:$4$] (v34) {};
  \path (v11) ++(-180:1) node[vertex, label=left:$2$] (v35) {};
  %
  %
  \path (v4) ++(170:1) node[vertex, label=below:$7$] (v12) {};
  \path (v4) ++(130:1) node[vertex, label=below:$5$] (v13) {};
  \path (v4) ++(90:1) node[vertex, label=right:$3$] (v14) {};
  \path (v12) ++(200:1) node[vertex, label=left:$8$] (v36) {};  
  \path (v12) ++(180:1) node[vertex, label=left:$6$] (v37) {};  
  \path (v12) ++(160:1) node[vertex, label=left:$2$] (v38) {};
  \path (v13) ++(170:1) node[vertex, label=left:$8$] (v39) {};
  \path (v13) ++(135:1) node[vertex, label=left:$6$] (v40) {};
  \path (v13) ++(100:1) node[vertex, label=above:$2$] (v41) {};
  \path (v14) ++(120:1) node[vertex, label=above:$8$] (v42) {};
  \path (v14) ++(90:1) node[vertex, label=above:$6$] (v43) {};
  \path (v14) ++(60:1) node[vertex, label=above:$2$] (v44) {};
  %
  %
  \draw[ultra thick, CB_Blue] (v18) -- (v6) -- (v2) -- (v0) -- (v1) -- (v5)
                                -- (v17);
  \draw (v33) -- (v11) -- (v3) -- (v0) -- (v4) -- (v12) -- (v38);
  \draw[ultra thick, CB_blue] (v45) -- (v15) -- (v46);
  \draw[ultra thick, CB_Green] (v56) -- (v51) -- (v58);
  \draw[ultra thick, CB_pink] (v49) -- (v16) -- (v50);
  \draw[ultra thick, CB_Orange] (v61) -- (v52) -- (v60);
  \draw (v59) -- (v52) -- (v5) -- (v51) -- (v57);
  \draw (v47) -- (v15) -- (v1) -- (v16) -- (v48);
  \draw[ultra thick, CB_purple] (v54) -- (v22) -- (v7) -- (v21);
  \draw[ultra thick, CB_Brown] (v26) -- (v8) -- (v25);
  \draw (v55) -- (v22) -- (v53);
  \draw (v20) -- (v6) -- (v19);
  \draw (v24) -- (v8) -- (v2) -- (v7) -- (v23);
  \draw[ultra thick, CB_green] (v27) -- (v9) -- (v3) -- (v10) -- (v31);
  \draw[ultra thick, CB_Pink] (v34) -- (v11) -- (v35);
  \draw (v29) -- (v9) -- (v28);
  \draw (v32) -- (v10) -- (v30);
  \draw[ultra thick, CB_orange] (v36) -- (v12) -- (v37);
  \draw[ultra thick, CB_Purple] (v40) -- (v13) -- (v4) -- (v14) -- (v44);
  \draw (v39) -- (v13) -- (v41);
  \draw (v42) -- (v14) -- (v43);
\end{tikzpicture}
  \caption{The lift of a Hamiltonian cycle is a line ensemble covering all
    vertices of $\mathbb{T}$. (Each coloured bold line is infinite.)}
  \label{fig:lift}
\end{figure}

Moreover, $\tilde{\prob}(o\in \cL) = \frac{1}{|G|}\sum_{x\in G} 1_{\tilde{x}\in \cL} = \frac{1}{|G|}\sum_{x\in G}1_{x\in C} = \frac{|C|}{|G|}$. If $C$ covers $G$, we thus get $\tilde{\prob}(o\in \cL)=1$.

In particular, the $(q+1)$-regular tree $\T_q$ is Hamiltonian, since it covers the complete bipartite $(q+1)$-regular graph on $2(q+1)$ vertices, which is Hamiltonian.
\end{rem}

We may now state our adaptation of \cite[Theorem 1.5]{BSV17} to quantum trees. Here, if $G$ is a finite graph, we denote by $\mathbf{G}=\mathbf{G}(\alpha,\mathbb{L},\mathbb{W})$ the quantum graph obtained by endowing each edge with a length $L_e$, a potential $W_e$ and each vertex a coupling constant $\alpha_v$, so the Schr\"odinger operator $\cH=-\Delta+W$ acts with $\delta$-conditions. We say that $\mathbf{T}=\widetilde{\mathbf{G}}$ if $\T= \widetilde{G}$ is endowed with the lifted structure $\alpha_v = \alpha_{\pi v}$, $L_{(u,v)} = L_{(\pi u,\pi v)}$ and $W_{(u,v)}=W_{(\pi u,\pi v)}$.

Recall definition \eqref{e:csfun} of $S_z(x_b)$.

\begin{prp}\label{prp:bsv}
Suppose $G$ is a finite Hamiltonian graph. Endow $G$ with a quantum structure $\mathbf{G}$ with $\delta$-conditions and let $\mathbf{T}=\widetilde{\mathbf{G}}$. If $\lambda$ is an eigenvalue of $\cH_{\mathbf{T}}$, then $\lambda$ must be a Dirichlet value, i.e.\ $S_{\lambda}(L_b) = 0$ for some $b$.
\end{prp}
\begin{proof}
Suppose $\mathcal{H}_\mathbf{T}$ has an eigenvalue, say $\cH_{\mathbf{T}} \varphi_{\lambda}=\lambda\varphi_{\lambda}$ for some $\varphi_{\lambda}\in L^2(\mathcal{T})$.

Suppose on the contrary that $S_{\lambda}(L_b)\neq 0$ for all $b$.

Let $\mathring{\varphi} = \varphi|_{V}$. We claim that
\[
(\cA_{\lambda}\mathring{\varphi}_{\lambda})(v)=W_{\lambda}(v)\mathring{\varphi}_{\lambda}(v) \,,
\]
where
\[
(\cA_{\lambda}\psi)(v) = \sum_{u\sim v}\frac{\psi(u)}{S_{\lambda}(L_{uv})} \quad \text{and}\quad W_{\lambda}(v) = \alpha_v + \sum_{u\sim v} \frac{C_{\lambda}(L_{uv})}{S_{\lambda}(L_{uv})}
\]
and $uv:=(u,v)$. In fact,
\begin{align*}
\sum_{u\sim v} \frac{\varphi_{\lambda}(u)}{S_{\lambda}(L_{vu})} = \sum_{u\sim v} \frac{\varphi_{\lambda}(t_{vu})}{S_{\lambda}(L_{vu})} &= \sum_{u\sim v}\left[\frac{\varphi_{\lambda}(o_{vu})C_{\lambda}(L_{vu})}{S_{\lambda}(L_{vu})}+\varphi_{\lambda}'(o_{vu})\right] \\
&= \varphi_{\lambda}(v)\sum_{u\sim v} \frac{C_{\lambda}(L_{vu})}{S_{\lambda}(L_{vu})} + \alpha_v\varphi_{\lambda}(v).
\end{align*}

Let $\mathfrak{E}_{\lambda}\subset \ell^2(\mathbb{T})$ be the set of functions satisfying this eigenvalue equation, i.e.\ $\psi\in \mathfrak{E}_{\lambda}$ if and only if
\begin{equation}\label{e:eququan}
W_{\lambda}(w)\psi(w) = \sum_{u\sim w}\frac{\psi(u)}{S_{\lambda}(L_{uw})} \,.
\end{equation}

Note that $\mathfrak{E}_{\lambda}$ is invariant. In fact, let $M_{\lambda} = \cA_{\lambda}-W_{\lambda}$. Then $M_{\lambda}$ is self-adjoint. This is because all weights are real-valued and symmetric. Moreover, $\mathfrak{E}_{\lambda} = \ker M_{\lambda}$. Now, if $g\in \Gamma$, let $(U_gf)(v)=f(g^{-1}v)$. Using that $\alpha_{gv} = \alpha_v$ $L_{(gu,gv)} = L_{(u,v)}$, $W_{(gu,gv)}=W_{(u,v)}$, it easily follows that $U_g^{-1}M_{\lambda}U_g = M_{\lambda}$. Standard arguments imply that $U_g^{-1} P_{\mathfrak{E}_{\lambda}} U_g = P_{\mathfrak{E}_{\lambda}}$, so $\mathfrak{E}_{\lambda}$ is indeed invariant.

Let $C$ be a Hamiltonian cycle in $G$, so its lift $\cL$ is a line ensemble as in Remark~\ref{rem:hami}. Using the line ensemble, we may use the construction of \cite[Theorem 1.5]{BSV17} to define for any $k\in \N^{\ast}$ an invariant labeling $\eta_k:V(\T)\to\Z_k$ of the vertices of $\T$ by integers which satisfies:
\begin{itemize}
\item $b:=\prob(o \text{ is bad }) \le 1/k$,
\item vertices in $\cL$ with $\eta_k(v)\neq 0$ are prodigy. Vertices in $\cL$ with $\eta_k(v) = 0$ are bad,
\item vertices outside $\cL$ are level.
\end{itemize}

In our case, all vertices are in $\cL$, so there are no level vertices.

We now argue as in \cite[Theorem 2.3]{BSV17}. Here the situation is simpler as there are no level vertices. Let  $\mathfrak{B}$ be the space of vectors which vanish on the set of bad vertices. Then $\dim_{\mathrm{VN}} \mathfrak{B} = \prob(o \text{ is not bad }) = 1-b$. Let $\mathfrak{E}'=\mathfrak{E}_{\lambda} \cap \mathfrak{B}$. Then using $\dim_{\mathrm{VN}} (R \cap Q)\ge \dim_{\mathrm{VN}} R + \dim_{\mathrm{VN}} Q -1$, we have
\[
\dim_{\mathrm{VN}} \mathfrak{E}_{\lambda} \le b + \dim_{\mathrm{VN}} \mathfrak{E}' \,.
\]
We show $\mathfrak{E}'$ is the trivial subspace by induction on the label $j$, showing that from low to high, any $f\in \mathfrak{E}'$ vanishes on vertices with label $j$. Remember vertices $v$ can only be prodigy or bad.

Recall that we have finitely many labels. Let $j_0$ be the smallest label and let $v$ be of label $j_0$. If $v$ is a bad vertex, then $f(v)=0$ since $f\in \mathfrak{B}$. Note that $v$ cannot be a prodigy. Hence $f(v)=0$ on vertices of smallest label.

Now assume $f\in \mathfrak{E}'$ vanishes on all vertices with label strictly below $j$. Since $f\in \mathfrak{B}$, we know $f$ vanishes on bad vertices. If $v$ is a prodigy vertex of label $j$, then $v$ has a neighbour $w$ such that $f$ vanishes on $w$ and all neighbours of $w$, except perhaps $v$. But \eqref{e:eququan} gives $\frac{f(v)}{S_{\lambda}(L_{vw})} = W_{\lambda}(w)f(w) - \sum_{u\sim w,u\neq v} \frac{f(u)}{S_{\lambda}(L_{uv})}$, so if the RHS is zero, then $f(v)=0$.

Recalling that $b\le 1/k$, we have showed that $\dim \mathfrak{E}_{\lambda} \le 1/k$. As $k$ is arbitrary, we get $\dim \mathfrak{E}_{\lambda}=0$. It follows that $\mathfrak{E}_{\lambda} = \{0\}$. Indeed, we have $P_{\mathfrak{E}_{\lambda}}(v,v)=0$ for all $v$, so $\tr P_{\mathfrak{E}_{\lambda}}=0$, so $\|P_{\mathfrak{E}_{\lambda}}\|_{op} \le \|P_{\mathfrak{E}_{\lambda}}\|_1 =0$, implying $\mathfrak{E}_{\lambda}=\{0\}$. It follows that there is no $\ell^2$ function on $\mathbb{T}$ such that $\cA_{\lambda} \psi = W_{\lambda} \psi$. By \cite{Ex97}, it follows that there is no $L^2$ function on $\mathbf{T}$ such that $H_{\mathbf{T}} \varphi = \lambda \varphi$. In other words, $\lambda$ is not an eigenvalue of $H_{\mathbf{T}}$ (contradiction).
\end{proof}

\subsection{More examples}
Let $\mathbf{T}=\widetilde{\mathbf{G}}$. We now show the spectral bottom $a_0 = \inf\sigma(H_{\mathbf{T}})$ is strictly below the smallest Dirichlet value. By virtue of Proposition~\ref{prp:bsv}, if $G$ is Hamiltonian, this implies $a_0$ is not an eigenvalue. In particular, $a_0$ is not an isolated spectral value, so in view of Theorem~\ref{thm:det}, there is some pure AC spectrum near $a_0$.

Recall that if $Q_j$ are the quadratic forms associated to operators $H_j$, then $H_1 \ge H_2$ if $D(Q_1) \subseteq D(Q_2)$ and $Q_1(f,f) \ge Q_2(f,f)$ for $f\in D(Q_1)$. If $H_1 \ge H_2$ then $\inf \sigma(H_1) \ge \inf \sigma(H_2)$.

In $\mathbf{T}$ there are finitely many different kinds of edges
(lengths and potentials).  To each oriented edge $b$, we associate the smallest
Dirichlet: the smallest $\mathcal{E}$ such that $S_{\mathcal{E}}(L_b)=0$.
Denote the least of those values by $\ED$ and let $(v,w)$ be the edge
on which it is attained (choose one of them if there is more than one edge
with the same lowest Dirichlet value).

Consider the quantum star graph around $v$, with the
usual $\delta$-condition at $v$, and Dirichlet conditions at the extremities
$w'\sim v$. Denote this (compact) graph by $\bigstar$ and let $E_0$ be
its smallest eigenvalue. We claim that
\[
a_0 \le E_0 < \ED \,.
\]

For the first inequality, let $H_{\star} f = E_0 f$. Then $E_0 = \frac{Q_{\star}(f,f)}{\|f\|^2}$, where $Q_\star$ is the quadratic form associated to $H_\star$. Let $\widetilde{f}$ be the extension of $f$ by zero to $\mathcal{T}$. Then $\widetilde{f} \in D(Q_{\mathbf{T}})$, where $Q_{\mathbf{T}}$ corresponds to $H_{\mathbf{T}}$. So
\begin{displaymath}
a_0 = \inf_{g\neq 0} \frac{Q_{\mathbf{T}}(g,g)}{\|g\|^2} \le \frac{Q_{\mathbf{T}}(\widetilde{f},\widetilde{f})}{\|\widetilde{f}\|^2} = \frac{Q_{\star}(f,f)}{\|f\|^2} = E_0.
\end{displaymath}

For the second inequality, note that if $f\in D(H_{\star})$ and
$H_{\star} f = E f$, then on the edge $b$,
\begin{displaymath}
  f(x_b) = A C_E(x_b) + B S_E(x_b).
\end{displaymath}
Due to the Dirichlet conditions at extremities of the star, $f(L_b)=0$.
If $E<\ED$ then $S_E(L_b)\neq 0$ and $B=-A C_E(L_b)/S_E(L_b)$. Evaluating at $x_b=0$, the centre of the star, reveals that $A=f(v)$. Thus,
\begin{displaymath}
  f(x_b) = f(v) \frac{C_E(x_b)S_E(L_b) - C_E(L_b)S_E(x_b)}{S_E(L_b)}.
\end{displaymath}
The condition at $v$, $\sum_{b\in\star} f'(o_b) = \al_v f(v)$, leads to
\begin{equation}
  \label{eq:ZE}
  \sum_{w'\sim v} -\frac{C_E(L_{vw'})}{S_E(L_{vw'})} = \al_v.
\end{equation}
Denote the left hand side of \eqref{eq:ZE} as a function of $E$ as $Z(E)$.
A solution $Z(E)=\al_v$ will be an eigenvalue of the star graph.

Let us consider the behaviour of $Z(E)$ as $E$ increases to $\ED$. We first show that as $E$ approaches $\ED$ from below, $S_E(L_{vw})\to0$ from above. For this, note that:
\begin{itemize}
\item by \cite[Theorem 6(a)]{PT87}, $S_\ED(x)$ has exactly two zeros on $[0,L_{vw}]$. These are thus $\{0,L_{vw}\}$. If $E<\ED$, since $S_E(0)=0$ and $S_E'(0)=1$, we know that $S_E$ is positive near $0$. If we show that its first zero on $(0,\infty)$ occurs after $L_{vw}$, this will imply that $S_E(L_{vw})>0$, which is what we seek.
\item We thus check that if $E<\ED$, then the first zero of $S_E$ on $(0,\infty)$ occurs after the first zero of $S_{\ED}$. For this, let $f(x) = S_E(x)S'_{\ED}(x)-S'_E(x)S_{\ED}(x)$. Then $f'(x) = S_E(x)S_{\ED}(x)(E-\ED)$, which is negative until the first zero of $S_E$ or $S_{\ED}$. Suppose for contradiction that the first zero of $S_E$ (call it $L_E$) is before the first zero of $S_{\ED}$. Then we get that $f(0)=0$, $f(L_E)>0$, but $f'(x)<0$ on $(0,L_E)$, which is absurd. This proves the claim.
\end{itemize} 

Next, $S'_\ED(L_{vw})<0$ because $L_{vw}$ is the first zero of $S_{\ED}$ after $x=0$.  By
the Wronskian relation $C_\ED(L_{vw})S'_\ED(L_{vw})=1$, so $C_\ED(L_{vw})<0$
too. Therefore, since $\ED$ is the smallest
Dirichlet eigenvalue, all the terms in $Z(E)$ are either finite or diverge to $+\infty$ as $E\nearrow\ED$, so that $Z(E)$ diverges to $+\infty$
as $E\nearrow\ED$.

On the other hand, from the proof of Proposition \ref{prp:Lang}, we
know that $C_E(L_b)/S_E(L_b) \to+\infty$ as $E\to-\infty$, so $Z(E)\to
-\infty$ in the same limit.

Together we have
\begin{displaymath}
  Z(E) \to -\infty \text{ as }E\to-\infty\qquad\text{and}\qquad
  Z(E) \to +\infty \text{ as }E\nearrow \ED.
\end{displaymath}
Since $Z$ is continuous, there is a solution to $Z(E)=\al_v$ strictly
below $\ED$, as claimed.

\section{AC spectrum under perturbations}

We now aim to prove Theorem~\ref{thm:random}.  
For this, we adapt the approach of \cite{KLW}, see also \cite{FHS} for some earlier ideas.

A very sketchy outline of the argument is as follows. Our results in the previous sections tell us that we have a good control over the unperturbed operator: it has pure AC spectrum in $\Sigma$, and all relevant spectral quantities such as the Green's functions and WT functions have a limit on $\Sigma$, which has a strictly positive imaginary part. Let $H_v^z$ be such a spectral quantity to be chosen later, where $z\in \C^+$ and $v\in V(\T)$. The aim is now to prove an $L^p$-continuity estimate in mean with respect to the disorder $\eps$. More precisely, in some semi-metric $\gamma$ on $\h$, see \eqref{e:gasm}, we aim to show that $\lim_{\eps\downarrow 0} \sup_{z\in I+\ii(0,1)}\expect(\gamma(h_v^z,H_v^z)^p)=0$, where $h_v^z$ is the analogous spectral quantity for the perturbed operator (equal to $H_v^z$ when $\eps=0$). Such a uniform stability result directly implies almost-sure \emph{pure} AC spectrum in the combinatorial case, by classical results. In our case we will have to work further (Section~\ref{sec:acran} and Appendix~\ref{app:A2}).

The important question now is which Herglotz function to choose for $H_v^z$. In the combinatorial case it is natural to take $\zeta^z(b)$, with $b=(v_-,v)$. We considered something close in Proposition~\ref{prp:Lang}, namely $S_z(L_b)\zeta^z(b)$. For the present continuity considerations, $H_v^z = \frac{R_z^+(o_b)}{\sqrt{z}}$ seems to behave better. Still, the function $\sqrt{z}S_z(L_b)\zeta^z(b)$ will also play an important role in fixing the disorder window later on (Appendix~\ref{app:uni}). Of course it can be argued that all such quantities are related in Section~\ref{sec:greenquan}, but one needs to be careful because the aim is roughly to prove strict contraction estimates on $\expect(\gamma(h_v^z,H_v^z)^p)$ in terms of $\sum_{w\in \cN_v^+}\expect(\gamma(h_w^z,H_w^z)^p)$, so adding/multiplying terms to $h_w^z$ is not a very good operation, although there are partial answers (Lemma~\ref{lem:11} and Lemma~\ref{lem:Kelphd}).

We have not discussed how the $L^p$-continuity actually proceeds; we outline the proof in \S~\ref{sec:unifcont} after giving some important expansion estimates in \S~\ref{sec:2step}.

\subsection{The two step expansion estimate}\label{sec:2step}

Consider the hyperbolic disc $\mathbb{D}=\{z\in\mathbb{C}: |z|<1\}$
equipped with the usual hyperbolic distance metric
\[
d(z, z')=\cosh^{-1} (1+\delta(z, z'))\,,
\]
where
\[
\delta(z, z')=2\,\frac{|z-z'|^2}{(1-|z|^2)(1-|z'|^2)} \,,\quad|z|, |z'|<1.
\]

We will use the M\"obius transformation $\cC(z)=\frac{z-\ii}{z+\ii}$ that sends the upper half plane model $\h$ isometrically to the disk model. Its inverse is $\cC^{-1}(u)=\ii \frac{1+u}{1-u}$. Note that if, for $g,h\in \mathbb{H}$, we set
\begin{equation}\label{e:gasm}
\gamma(g,h) = \frac{|g-h|^2}{\Im g\Im h} \,,
\end{equation}
then
\begin{equation}\label{e:gammadelta}
\gamma(g, h)=2\delta(\cC(g), \cC(h)) \,.
\end{equation}
In fact,
\begin{displaymath}
\delta(\cC g,\cC h) = 2 \frac{|(g-\ii)(h+\ii)-(g+\ii)(h-\ii)|^2}{(|g+\ii|^2-|g-\ii|^2)(|h+\ii|^2-|h-\ii|^2)} = \frac{1}{2}\gamma(g,h).
\end{displaymath}

The following is a more adequate replacement of \cite[Lemma 1]{KLW} to our framework.

\begin{lem}\label{lem:11}
Let $K$ be a compact subset of the hyperbolic disc $\mathbb{D}$. Then there exists a continuous function $C_K: \mathbb{R}^+\longrightarrow \mathbb{R}^+$, such that
$C_K(0)=0$, and
\[
\delta( \lambda_1 z,\lambda_2 z')\le \left(|\lambda_1|^2+C_K(|\lambda_1-\lambda_2|)\right) \delta(z, z') + C_K(|\lambda_1-\lambda_2|)
\]
for all $z\in K$ and for all $z'\in \mathbb{D}$, for all $\lambda_i\in \C$ such that $|\lambda_i|\leq 1$.

More explicitly, if $r_K = \max_{z\in K} |z|<1$, we can take $C_K(t) = \frac{8}{(1-r_K)^2}\cdot t$.
\end{lem}
\begin{proof}
First assume $\lambda_1\neq 0$. Suppose $z\neq z'$. We have
\begin{align*}
\delta(\lambda_1z,\lambda_2 z') & = |\lambda_1|^2\frac{|z-\lambda_2\lambda_1^{-1} z'|^2}{|z-z'|^2}\cdot \frac{(1-|z|^2)(1-|z'|^2)}{(1-|\lambda_1z|^2)(1-|\lambda_2 z'|^2)}\cdot \delta(z,z') \\
& \le |\lambda_1|^2\frac{|z-\lambda_2\lambda_1^{-1} z'|^2}{|z-z'|^2}\cdot  \delta(z,z')\\
& \leq |\lambda_1|^2\left(1+|1-\lambda_2\lambda_1^{-1}|\frac{|z'|}{|z-z'|}\right)^2\delta(z,z') \,,
\end{align*}
where, to obtain the first inequality, we used that $|\lambda_i|\le 1$.
Let $\delta_K=1-r_K$. If $|z-z'|\ge \frac{\delta_K}{2}$, we have
\begin{align}
\delta(\lambda_1z,\lambda_2 z')& \le \left(|\lambda_1|+|\lambda_1-\lambda_2|\Big[1+\frac{|z|}{|z-z'|}\Big]\right)^2\delta(z,z') \nonumber \\
& \le \left(|\lambda_1|+(1+2r_K\delta_K^{-1})|\lambda_1-\lambda_2|\right)^2\delta(z,z') \,. \label{e:firb}
\end{align}
Now assume $|z-z'|<\frac{\delta_K}{2}$. Then $|z'| < \frac{\delta_K}{2} + r_K = \frac{1+r_K}{2}$, so
\begin{align}
\delta(\lambda_1z,\lambda_2 z') &\le 2|\lambda_1|^2\frac{(|z-z'|+|1-\lambda_2\lambda_1^{-1}|\cdot |z'|)^2}{(1-|z|^2)(1-|z'|^2)}  \nonumber\\
& = |\lambda_1|^2\delta(z,z') + 2\, \frac{2|z-z'|\cdot |\lambda_1||\lambda_1-\lambda_2|\cdot|z'|+|\lambda_1-\lambda_2|^2\cdot|z'|^2}{(1-|z|^2)(1-|z'|^2)} \nonumber \\
&\le |\lambda_1|^2\delta(z,z') + 2\, \frac{\frac{\delta_K(1+r_K)}{2} \cdot |\lambda_1-\lambda_2| + (\frac{1+r_K}{2})^2|\lambda_1-\lambda_2|^2}{(1-r_K^2)(1-(\frac{1+r_K}{2})^2)} \,. \label{e:secb}
\end{align}
The first assertion follows. For the explicit formula, we check the relevant terms in \eqref{e:firb}, \eqref{e:secb} are bounded by $4\frac{1+r_K}{(1-r_K)^2}|\lambda_1-\lambda_2|$. As $r_K<1$, this will complete the proof.

Note that $|\lambda_1-\lambda_2|\le 2$, so
\begin{multline*}
  2(1+2r_K\delta_K^{-1})|\lambda_1||\lambda_1-\lambda_2| + (1+2r_K\delta_K^{-1})^2|\lambda_1-\lambda_2|^2 \\
   \le 2\Big\{1+\frac{2r_K}{1-r_K}+\Big(1+\frac{2r_K}{1-r_K}\Big)^2\Big\}|\lambda_1-\lambda_2|= 4\,\frac{1+r_K}{(1-r_K)^2}|\lambda_1-\lambda_2|\,.
\end{multline*}
For the other term, using $\frac{1}{1-x^2}\le \frac{1}{1-x}$ for $x\in [0,1]$, we have
\begin{align*}
  2\, \frac{\frac{\delta_K(1+r_K)}{2} \cdot |\lambda_1-\lambda_2| + (\frac{1+r_K}{2})^2|\lambda_1-\lambda_2|^2}{(1-r_K^2)(1-(\frac{1+r_K}{2})^2)}
  &\le  \frac{1-r_K^2 + (1+r_K)^2}{(1-r_K)(1-(\frac{1+r_K}{2}))}\cdot |\lambda_1-\lambda_2|\\
  &= 4\,\frac{1+r_K}{(1-r_K)^2}|\lambda_1-\lambda_2|\,.
\end{align*}

Finally, if $\lambda_1=0$, $\delta(\lambda_1z,\lambda_2z')=2\frac{|\lambda_2z'|^2}{1-|\lambda_2z'|^2}\le \frac{2|\lambda_2|}{1-|z'|^2}$. If $|z-z'|\ge \frac{\delta_K}{2}$ then $1\le \frac{4|z-z'|^2}{(1-r_K)^2}$ and the claim follows. Otherwise $|z'|<\frac{1+r_K}{2}$, so $\frac{1}{1-|z'|^2}\le \frac{2}{1-r_K}$, proving the claim.
\end{proof}

We will also need \cite[Lemma 2.16]{Kel}, which we recall below:

\begin{lem}\label{lem:Kelphd}
For any $g,h,z\in \h$,
\[
\max\left\{\gamma(g,h+z),\gamma(g+z,h)\right\} \le (1+c_g(z))\gamma(g,h) + c_g(z)\,,
\]
where $c_g(z) = \frac{4|z|}{\Im g} + \frac{4|z|^2}{(\Im g)^2}$.
\end{lem}

We now consider $\mathbf{T}$ in the \emph{twisted} view. If $\T$ is a tree with parameters $(\{\alpha_v^0\}, \{L_v^0\})$ before perturbation and $(\{\alpha_v^{\omega}\}, \{L_v^{\omega}\})$ after perturbation, and if $b=(v_-,v)\in B(\mathbf{T}) = B(\mathbf{T}_{o}^+)$, we set
\[
h_v^z = \frac{R_{z,\omega}^+(o_b)}{\sqrt{z}} \quad \text{and} \quad H_v^z = \frac{R_{z,0}^+(o_b)}{\sqrt{z}}  \,,
\]
where $R_{z,\omega}^+(o_b)$ is the WT function of $(\T,\{L_v^{\omega}\},\{\alpha_v^{\omega}\})$ and $R_{z,0}^+(o_b)$ the WT function of $(\T,\{L_v^0\},\{\alpha_v^0\})$. The notation may seem a bit confusing since the WT function is evaluated at $v_-$ instead of $v$. However, this is in accordance with the notations of \S~\ref{sec:defs}, where the oriented edges are indexed by their endpoint, and where we write $\phi_v$ instead of $\phi_b$ if $v=t_b$. We also define for $b=(v_-,v)$,
\[
\mathrm{g}_v^z=\frac{R_{z,\omega}^+(L_b^{\omega})}{\sqrt{z}} \quad \text{and} \quad \Gamma_v^z=\frac{R_{z,0}^+(L_b^0)}{\sqrt{z}} \,.
\]
Then the $\delta$-conditions \eqref{e:kir+} applied to $V_{z;o}^+(x)$ give
\begin{equation}\label{e:kiro}
\sum_{v_+\in\cN_v^+} h_{v_+}^z = \mathrm{g}_v^z + \frac{\alpha_v^{\omega}}{\sqrt{z}} \quad \text{and}\quad \sum_{v^+\in\cN_v^+} H_{v_+}^z = \Gamma_v^z+\frac{\alpha_v^0}{\sqrt{z}} \,.
\end{equation}

We shall assume the coupling constants $\alpha_v\ge 0$ and potentials $W\ge 0$. In this case we can ensure that $h_v^z$ and $H_v^z$ are Herglotz functions (Lemma~\ref{lem:ASW}), so their Cayley transform lies in $\mathbb{D}$. In this case, $\mathrm{g}_v^z$ and $\Gamma_v^z$ are also Herglotz by \eqref{e:kiro}.

\begin{rem}\label{rem:cayley}
  Assume there is no potential on the edges: $W_v\equiv 0$. Then the functions $h_v^z$ and $\mathrm{g}_v^z$ are also connected by the following relations: if $b=(v_-,v)$ and we expand $V_{z;o}^+(x_v)$ in the basis $C_z(x_v)=\cos\sqrt{z}x_v$ and $S_z(x_v)=\frac{\sin\sqrt{z}x_v}{\sqrt{z}}$, we get
\begin{displaymath}
R_z^+(t_b) = \frac{R_z^+(o_b)S_z'(L_b)+C_z'(L_b)}{R_z^+(o_b)S_z(L_b)+C_z(L_b)} = \frac{R_z^+(o_b)\cos\sqrt{z}L_b-\sqrt{z} \sin \sqrt{z}L_b}{R_{z}^+(o_b)\frac{\sin\sqrt{z}L_b}{\sqrt{z}}+\cos\sqrt{z}L_b}\,,
\end{displaymath}
so $\mathrm{g}_v^z = \frac{h_v^z\cos\sqrt{z}L_v-\sin\sqrt{z}L_v}{h_v^z\sin\sqrt{z}L_v+\cos\sqrt{z}L_v}$. From this, we find
\[
\cC(\mathrm{g}_v^z) = \frac{\mathrm{g}_v^z-\ii}{\mathrm{g}_v^z+\ii}= \cC(h_v^z)\cdot \frac{\cos\sqrt{z}L_v-\ii\sin\sqrt{z}L_v}{\cos\sqrt{z}L_v+\ii\sin\sqrt{z}L_v} = \ee^{-2\ii\sqrt{z}L_v} \cC(h_v^z)
\]
as previously observed in \cite{ASW06}. We also remark that we can invert the M\"obius identity defining $\mathrm{g}_v^z$ in terms of $h_v^z$ to get $h_v^z = \frac{\mathrm{g}_v^z\cos\sqrt{z}L_v+\sin\sqrt{z}L_v}{-\mathrm{g}_v^z\sin\sqrt{z}L_v+\cos\sqrt{z}L_v}$.
\end{rem}

We finally define
\[
\gamma_v(h)=\gamma(h_v^z, H_v^z) \,.
\]

The following plays the analog of \cite[Lemma 2]{KLW}.

\begin{lem}\label{lem:2}
Let $K$ be a compact subset of the hyperbolic disc, and assume that $z$ varies in a compact subset such that $H_v^z\in \h$ and $\cC(\Gamma_v^z)\in K$ for all $v$.
Then (see equations \eqref{eq:defq}, \eqref{eq:defAlpha} and \eqref{e:Q} below for the definition of the quantities $q$, $\alpha$ and $Q$) 
\begin{multline*}
\shoveleft
\gamma_v(h)\le (1+  C_K(\eps'))(1+c_H(\eps))\sum_{v'\in \cN_v^+}\frac{\Im H_{v'}^z\cdot \gamma_{v'}(h)}{\sum_{u\in \cN_v^+}\Im H_u^z}
\Big(\sum_{w\in \cN_v^+}q_w(h)Q_{v', w}(h) \cos\alpha_{v', w}(h)\Big) \\
+   2C_K(\eps') + (1+C_K(\eps'))c_H(\eps) \,, 
\end{multline*}
if $\sup_{z,v} |\frac{\alpha_v^{\omega} - \alpha_v^0}{\sqrt{z}}|\le \eps$ and $\sup_{z, v}|\ee^{2\ii\sqrt{z}L_v^{\omega}}-\ee^{2\ii\sqrt{z}L_v^0}|\le \eps'$, where $C_K(t)=\frac{8}{(1-r_K)^2}\cdot t$ and
\[
c_H(t) = \sup_{z,v}\frac{4t}{\sum_{v_+\in \cN_v^+} \Im H_{v_+}^z}\left(1+\frac{t}{\sum_{v_+\in \cN_v^+} \Im H_{v_+}^z}\right) \,.
\]
\end{lem}

We will apply this when $K = \bigcup_v\{\cC(\Gamma_v^z) : z\in I + \ii[0,1]\}$, where $I$ is a compact interval on which $\Gamma_v^{\lambda+\ii 0}$ exists and $\Im \Gamma_v^{\lambda+\ii 0}>0$. Note that the union here is finite as the unperturbed model is of finite cone type. For the same reason, the supremum in $c_H$ is a maximum.

The quantities $q, Q, \alpha$ are defined by formulas similar to those in \cite{KLW}: for $x,y\in \cN_v^+$,
\begin{equation}\label{eq:defq}
q_y(h)=\frac{\Im h_y}{\sum_{u\in \cN_v^+} \Im h_u}
\end{equation}
\begin{equation}\label{eq:defAlpha}
\cos\alpha_{x,y}(h)= \cos\arg(h_x-H^z_x)\overline{(h_y-H^z_y)}
\end{equation}
\begin{equation}\label{e:Q}
Q_{x,y}(h)=\frac{\sqrt{\Im h_x\Im h_y\Im H^z_x\Im H^z_y \gamma_x(h)\gamma_y(h)}}{\frac12(\Im h_x\Im H^z_y \gamma_y(h)+\Im h_y\Im H^z_x \gamma_x(h))} \,.
\end{equation}
assuming $h_x\neq H^z_x$ and $h_y\neq H^z_y$, otherwise we let $Q_{x,y}(h)=\cos\alpha_{x,y}(h)=0$.
\begin{proof}[Proof of Lemma~\ref{lem:2}]
Using \eqref{e:gammadelta} and Remark~\ref{rem:cayley}, we have
\[
\gamma_v(h) = 2\delta(\cC(H_v^z), \cC (h_v^z)) = 2\delta\left(\ee^{2\ii\sqrt{z} L_v^0} \cC (\Gamma_v^z), \ee^{2\ii\sqrt{z} L_v^{\omega}}\cC (\mathrm{g}_v^z)\right) .
\]

By hypothesis, $\cC(\Gamma_v^z)$ stays in $K$. Applying Lemma~\ref{lem:11} and \eqref{e:gammadelta}, we deduce that
\begin{equation}\label{e:gme}
\gamma_v(h) \le (1+C_K(\eps'))\gamma(\Gamma_v^z,\mathrm{g}_v^z) + 2C_K(\eps') \,.
\end{equation}
Suppose $\alpha_v^{\omega}\ge \alpha_v^0$. Using \eqref{e:kiro}, we need to estimate
\begin{equation}\label{e:interme}
\gamma\left(\sum_{v_+\in \cN_v^+} H_{v_+}^z - \frac{\alpha_v^0}{\sqrt{z}}, \sum_{v_+\in \cN_v^+} h_{v_+}^z - \frac{\alpha_{v}^{\omega}}{\sqrt{z}}\right) \le \gamma\left(\sum_{v_+\in \cN_v^+} H_{v_+}^z,\sum_{v_+\in\cN_v^+} h_{v_+}^z - \frac{\alpha_{v}^{\omega} - \alpha_{v}^0}{\sqrt{z}}\right) ,
\end{equation}
where the inequality holds by $\gamma(\xi+z,\xi'+z) \le \gamma(\xi,\xi')$ if $\Im z\ge 0$, as easily checked. Since we are assuming $\alpha_v^{\omega}\ge \alpha_v^0 \geq 0$, we have $-\frac{\alpha_v^{\omega}-\alpha_v^0}{\sqrt{z}}\in \h$. Using Lemma~\ref{lem:Kelphd}, we may drop the term $\frac{\alpha_v^{\omega}-\alpha_v^0}{\sqrt{z}}$. If $\alpha_{v}^0\ge \alpha_{v}^{\omega}$, we simply replace \eqref{e:interme} by
\[
\gamma\left(\sum_{v_+\in \cN_v^+} H_{v_+}^z - \frac{\alpha_{v}^0}{\sqrt{z}}, \sum_{v_+\in \cN_v^+} h_{v_+}^z - \frac{\alpha_{v}^{\omega}}{\sqrt{z}}\right) \le \gamma\left(\sum_{v_+\in \cN_v^+} H_{v_+}^z - \frac{\alpha_{v}^0 - \alpha_{v}^{\omega}}{\sqrt{z}},\sum_{v_+\in\cN_v^+} h_{v_+}^z \right)
\]
and remove $\frac{\alpha_v^0-\alpha_v^{\omega}}{\sqrt{z}}$ using Lemma~\ref{lem:Kelphd}. Finally, as calculated in \cite[Lemma 2]{KLW},
\[
\Big|\sum_{v_+\in \cN_v^+} H_{v_+}^z - \sum_{v_+\in \cN_v^+} h_{v_+}^z\Big|^2 = \sum_{v'\in \cN_v^+}\Im H_{v'}^z\cdot \gamma_{v'}(h) \Big(\sum_{w\in \cN_v^+}\Im h_w^z Q_{v', w}(h) \cos\alpha_{v', w}(h)\Big)\,.
\]
Dividing by $(\Im\sum_{v_+} H_{v_+})(\Im\sum_{v_+} h_{v_+})$ completes the proof. 
\end{proof}

\begin{rem}\label{rem:small}
Note that $\sum_{v_+} q_{v_+}=1$. On the other hand, $Q_{v',w}$ is a quotient of a geometric and arithmetic mean, so $0\le Q_{v',w} \le 1$. Since $-1\le \cos \alpha_{v',w} \le 1$, this shows that $-1\le \sum_{w\in\cN_v^+} q_w Q_{v',w} \cos\alpha_{v',w}\le 1$.
\end{rem}

We now deduce a ``two-step expansion''. We will assume our tree satisfies \textbf{(C2)}. In other words, for each vertex $v$, there is a vertex $v'\in \cN_v^+$ such that every label found in $\cN_v^+$ can also be found in $\cN_{v'}^+$.
We then say that $v'$ is chosen w.r.t.\ \textbf{(C2)}.
 
From now on, we denote
\[
S_v = \cN_v^+ \,.
\]

If $\ast=t_{b_o}$ and $\ast'$ is the vertex chosen w.r.t.\ \textbf{(C2)}
corresponding to $\ast$, we let
\[
S_{\ast,\ast'} = \left(\cN_\ast^+\setminus \{\ast'\}\right) \cup \cN_{\ast'}^+ \,.
\]

Given $x\in S_{\ast}$, let
\begin{equation}\label{e:cx1}
c_x(h) = \sum_{y\in S_{\ast}} q_y(h)Q_{x,y}(h)\cos\alpha_{x,y}(h)
\end{equation}
and for $x\in S_{\ast'}$, let
\begin{align}\label{e:cx2}
c_x(h) & = \Big(\sum_{y\in S_{\ast}}q_y(h)Q_{\ast',y}(h)\cos\alpha_{\ast',y}(h)\Big)\Big(\sum_{y\in S_{\ast'}}q_y(h)Q_{x,y}(h)\cos\alpha_{x,y}(h)\Big) \\
& = c_{\ast'}(h)\sum_{y\in S_{\ast'}} q_y(h)Q_{x,y}(h)\cos\alpha_{x,y}(h).\nonumber
\end{align}

It follows from Remark~\ref{rem:small} that $c_x(h)\in [-1,1]$. We next define for $x\in S_{\ast}$,
\begin{equation}\label{e:px1}
p_x = \frac{\Im H_x^z}{\sum_{y\in S_{\ast}} \Im H_y^z} \,,
\end{equation}
and for $x\in S_{\ast'}$,
\begin{equation}\label{e:px2}
p_x = \frac{\Im H_{\ast'}^z\Im H_x^z}{(\sum_{y\in S_{\ast}}\Im H_y^z)(\sum_{u\in S_{\ast'}} \Im H_u^z)} = p_{\ast'}\cdot \frac{\Im H_x^z}{\sum_{u\in S_{\ast'}} \Im H_u^z} \,.
\end{equation}
Then $\sum_{x\in S_{\ast,\ast'}} p_x = 1$.  Note that $c_x(h)$ is a quantity
that depends on the random parameters of the perturbed graph, whereas
$p_x$ is non-random.

Recall definition \eqref{e:Sigma} of the set $\Sigma$. Using
\eqref{e:kiro}, we also have $\Im \Gamma_v^{\lambda+\ii0}>0$ on
$\Sigma$.

\begin{prp}\label{p:expansion}
Let $I\subset \Sigma$ be a compact interval. There exists a continuous function $C_{I,H}:[0,\infty)^2\to[0,\infty)$ with $C_{I,H}(0,0)=0$ such that if $\sup_{z\in I+\ii[0,1],v} |\frac{\alpha_v^{\omega}-\alpha_v^0}{\sqrt{z}}| \le \eps$ and $\sup_{z\in I+\ii [0,1],v} |\ee^{2\ii \sqrt{z}L_v^{\omega}}-\ee^{2\ii\sqrt{z}L_v^0}| \le \eps'$, then
\[
\gamma_\ast(h) \le (1+C_{I,H}(\eps,\eps'))\sum_{x\in S_{\ast,\ast'}} p_x	c_x(h) \gamma_x(h) + C_{I,H}(\eps,\eps') \,.
\]
\end{prp}
\begin{proof}
Let $K = \bigcup_e \{\cC(\Gamma_e^z):z\in I+\ii[0,1]\}$. We apply Lemma~\ref{lem:2} to $v=\ast$, then to $v=\ast'$. If $c_{I,H}(\eps,\eps') = 2C_K(\eps') + c_H(\eps) + C_K(\eps')c_H(\eps)$, the statement follows by taking $C_{I,H}(\eps,\eps') = 2c_{I,H}(\eps,\eps') + c_{I,H}(\eps,\eps')^2$.
\end{proof}

To use this result under assumption \textbf{(P0)}, note that
\begin{align*}
  \left|\ee^{2\ii\sqrt{z}L_v^{\omega}}-\ee^{2\ii\sqrt{z}L_v^0}\right|
  &= \left| \ee^{\ii\sqrt{z}(L_v^\omega + L_v^0)} \right|
    \left|\ee^{\ii\sqrt{z}(L_v^{\omega}-L_v^0)}-\ee^{\ii\sqrt{z}(L_v^0-L_v^\omega)}
    \right| \\
  &\le 2\sin \sqrt{z}(L_v^\omega - L_v^0) \\
  &\le 2c_I\eps.
\end{align*}
Thus, $|\ee^{2\ii\sqrt{z}L_v^{\omega}}-\ee^{2\ii\sqrt{z}L_v^0}| \le \eps'$, with
$\eps'=2c_I\eps$.

\subsection{\texorpdfstring{$\boldsymbol{L^p}$}{Lp} 
continuity of the WT function}\label{sec:unifcont}

The aim of this subsection is to establish the following uniform continuity result in $L^p$-norm:

\begin{thm}\label{thm:conti}
Let $\T$ satisfy \emph{\textbf{(C0), (C1), (C2)}}  and $(\alpha,L)$ satisfy \emph{\textbf{(P0)},
    \textbf{(P1)}} and \emph{\textbf{(P2)}}. For all compact $I\subset \Sigma$, $I\cap\mathscr{D}=\emptyset$, and $p>1$, there is $\eps_0(I,p)>0$, $\eta_0(I,\varepsilon_D)>0$ and $C_p:[0,\eps_0)\to[0,\infty)$ with $\lim_{\eps\to 0} C_p(\eps)=0$ such that for any $v\in \T$, $\eps\le \eps_0$,
\begin{equation}\label{e:realcool}
\sup_{z\in I+\ii(0,\eta_0]} \expect\left(\gm(h_{v}^z,H_{v}^z)^{p}\right) \leq C_p(\eps) \,, \qquad 
\sup_{z\in I+\ii (0,\eta_0]} \expect(|h_v^z-H_v^z|^p) \le C_p(\eps) \,.
\end{equation}
\end{thm}

Given the key results of \S~\ref{sec:2step}, much of the proof of Theorem~\ref{thm:conti} goes as in the combinatorial case \cite{KLW}, so we only outline the main ideas. The latter part of this proof becomes nevertheless more technical in the quantum setting, due to the more complicated relations among the Green's functions (see \eqref{e:goprime}), so we give the necessary modifications in Appendix~\ref{app:uni}.

Recall that $\mathbf{T}=\mathbf{T}_o^+$ is defined by a cone matrix $M$ on a set of labels $\mathfrak{A}$ satisfying \textbf{(C0)}, \textbf{(C1)} and \textbf{(C2)}. We previously denoted $\ast=t_{b_o}$. More generally, given $(v_-,v)\in B(\mathbf{T}_o^+)$ with $\ell(v)=j$, if we consider the subtree $\T_{(v_-,v)}^+$, we shall denote $v=\ast_j$. The set $S_{\ast_j,\ast_j'}$ is then constructed analogously, and all results of \S~\ref{sec:2step} apply without change (this works in particular for $\ast_j=o_{b_o}$). 

Let us now discuss the proof of Theorem~\ref{thm:conti} in several steps:

\textbf{Step 1:} The Euclidean bound follows from the hyperbolic one.

In fact, if the $\gamma$-bound is proved, then using the Cauchy-Schwarz inequality,
\begin{equation}\label{e:e}
\mathbb{E}\ap{|h_{v}^z-H_{v}^z|^{p}}^2\leq\mathbb{E}\ap{\gm(h_{v}^z,H_{v}^z)^{p}} \mathbb{E}\ap{\ap{\Im h_{v}^z\Im{H_{v}^z}}^{p}}\leq C(\eps)\mathbb{E}\left(|h_{v}^z|^{p}\right)|H_{v}^z|^{p}.
\end{equation}
To bound the moment $\mathbb{E}(|h_v^z|^p)$, one uses the simple inequality 
\begin{equation}\label{e:modgam}
|\xi|\leq 4\gm(\xi,\zeta)\Im \zeta+2|\zeta|,\qquad \xi,\zeta\in\h,
\end{equation}
applied to $\xi=h_v^z$ and $\zeta=H_v^z$.

\medskip

\textbf{Step 2:} To prove the $\gamma$-bound, it suffices to show that for each $j\in \mathfrak{A}$,
\begin{equation}\label{e:ste2}
\mathbb{E}\left(\gm(h_{\ast_j}^z,H_{\ast_j}^z)^p\right)
\leq(1+c_{1}(\eps))(1-\de_0) \sum_{k\in\mathfrak{A}}P_{j,k}\mathbb{E}\left(\gm_{\ast_k}(h^z)^{p}\right)+C(\eps),
\end{equation}
as long as $P=(P_{j,k})$ forms a nonnegative irreducible matrix (and $c_1(\eps),C(\eps)\to 0$ as $\eps\to 0$). Indeed, the Perron-Frobenius theorem then provides a positive eigenvector $u\in \R^{\mathfrak{A}}$ such that $P^\intercal u=u$. If we consider the vector $\mathbb{E}_p\gm:=(\mathbb{E}( \gamma[h_{\ast_j}^z,H_{\ast_j}^z]^{p} ))_{j\in\mathfrak{A}}$, then \eqref{e:ste2} implies that 
\[
 \langle u, \mathbb{E}_p\gamma\rangle_{\C^{\mathfrak{A}}} \leq(1-\delta) \langle u,P\mathbb{E}_p\gamma\rangle_{\C^{\mathfrak{A}}} + C(\eps)=(1-\delta) \langle u,\mathbb{E}_p\gamma\rangle_{\C^{\mathfrak{A}}} + C(\eps),
\]
so $\langle u, \mathbb{E}_p\gamma\rangle_{\C^{\mathfrak{A}}}\le \frac{C(\eps)}{\delta}$, and the $\gamma$-bound easily follows.

\medskip

\textbf{Step 3:} To prove \eqref{e:ste2}, we apply the two-step expansion (Proposition~\ref{p:expansion}) to get
\begin{equation}\label{e:ste3}
\mathbb{E}\Big(\gm(h_{\ast_j}^z,H_{\ast_j}^z)^p\Big) \le(1+C_{I,H})^{2p-1}\mathbb{E}\Big(\Big(\sum_{x\in S_{\ast_j,\ast'_j}}p_{x}c_{x}(h^z)\gm_{x}(h^z)\Big)^p\Big) +(1+C_{I,H})^{p-1}C_{I,H}
\end{equation}
The idea now is that for $p\ge 1$,
\begin{equation}\label{e:jen}
\Big({\sum_{x\in S_{\ast_j,\ast_j'}}p_{x}c_{x}(h)\gamma_{x}(h)}\Big)^p \le \sum_{x\in S_{\ast_j,\ast_j'}}p_{x}\gamma_{x}(h)^p \le \max_{x\in S_{\ast_j,\ast_j'}}\gamma(H_{x}^z,h_{x}^z)^p
\end{equation}
as follows from Jensen's inequality and the facts $\sum p_x=1$ and
$|c_x(h)|\le 1$. However the first inequality is generally strict for $p>1$, and this is what provides the $(1-\delta_0)$ in \eqref{e:ste2} if we choose $P=P(z)$ to be the matrix
\[
P_{j,k}:=\sum_{\substack{x\in S_{\ast_j,\ast_j'},\\ \ell(x)=k}} p_{x}
\]
for $j,k\in\mathfrak{A}$, which satisfies the requirements of Step 2.

To derive the strict contraction for \eqref{e:jen} more precisely, the authors in \cite{KLW} introduce an additional averaging over the permutations of $S_{\ast_j,\ast'_j}$ preserving the labels. Let
\[
\Pi_j:=\{\pi:S_{\ast_j,\ast'_j}\to S_{\ast_j,\ast'_j}\mid \pi \text{ is bijective and }\ell(\pi(x))=\ell(x) \text{ for all } x\in S_{\ast_j,\ast'_j} \}.
\]

Fix $\ast=\ast_j$. Recall we denote functions $\phi_b$ on $[0,L_b]$ by $\phi_v$ if $b=(v_-,v)$. If $v\in S_{\ast,\ast'}$ and $\pi\in\Pi$, then $\phi_{\pi v} := \phi_{(\ast,\pi v)}$ if $\pi v\in S_{\ast}$, and $\phi_{\pi v} := \phi_{(\ast',\pi v)}$ if $\pi v\in S_{\ast'}$.

Denote $H^z = (H_x^z)_{x\in S_{\ast,\ast'}}\in \C^{S_{\ast, \ast'}}$. Given $g\in\h^{S_{\ast,\ast'}}$, $\pi\in\Pi$, denote $g\circ\pi=(g_{\pi(x)})_{x\in S_{\ast,\ast'}}$. By \eqref{e:WTtronq} and the symmetry of the tree, we  get for the unperturbed WT function
\begin{equation}\label{e:Hinva}
H^z=H^z\circ\pi
\end{equation}
for all $\pi\in\Pi$. By \textbf{(P1)}, \textbf{(P2)} and \eqref{e:WTtronq}, $h_{x}^z$ and $h_y^z$ are independent and identically distributed for $x,y\in S_{\ast,\ast'}$ that carry the same label. Hence,
\begin{equation}\label{e:hinva}
\expect\big(f(h^z)\big)=    \expect\big(f(h^z\circ\pi)\big)
\end{equation}
for any integrable function $f$ and all $\pi\in\Pi$.

Recall from Remark~\ref{rem:cayley} that
\begin{displaymath}
h_v^z = \frac{\mathrm{g}_v^z\cos\sqrt{z}L_v+\sin\sqrt{z}L_v}{-\mathrm{g}_v^z\sin\sqrt{z}L_v+\cos\sqrt{z}L_v}\,.
\end{displaymath}
On the other hand, $\mathrm{g}_v^z = \sum_{v_+} h_{v_+}^z - \frac{\alpha_v}{\sqrt{z}}$. Given $g\in \h^{S_{\ast,\ast'}}$, it is therefore natural to define $g_{\ast'}=g_{\ast'}(z,\alpha,L)$ by
\begin{equation}\label{e:goprime}
g_{\ast'} = \frac{\phi_{\ast'}(g)\cos\sqrt{z}L+\sin\sqrt{z}L}{-\phi_{\ast'}(g)\sin\sqrt{z}L+\cos\sqrt{z}L} \,,\qquad \phi_{\ast'}(g)= - \frac{\alpha}{\sqrt{z}} + \sum_{x\in S_{\ast'}} g_x \,.
\end{equation}
In particular, for $g=H^z=(H_x^z)$, we get $(H^z)_{\ast'}(z,\alpha^0,L^0) = H^z_{\ast'}$.

Introduce the notation $g_x^{(\pi)} = g_{\pi(x)}$ for $x\in S_{\ast,\ast'}$. We also define
\begin{equation} \label{e:goprime2}
g_{\ast'}^{(\pi)} = (g^{\pi})_{\ast'} = \frac{\phi_{\ast'}(g^{(\pi)})\cos\sqrt{z}L+\sin\sqrt{z}L}{-\phi_{\ast'}(g^{(\pi)})\sin\sqrt{z}L+\cos\sqrt{z}L} \,.
\end{equation}
Note that $g_{\ast'}^{(\pi)} \neq g_{\pi \ast'}$, in fact $\pi \ast'$ is not even defined. On the other hand, $H_{\ast'}^{(\pi)}(z,\alpha_{\ast'}^0,L_{\ast'}^0) = H_{\ast'}$ by \eqref{e:Hinva}.

Given $g\in \h^{S_{\ast,\ast'}}$, $\pi\in \Pi$ and $g_{\ast'}(z,\alpha,L)$ as in \eqref{e:goprime}--\eqref{e:goprime2}, we denote
\begin{align*}
\gm_{x}^{(\pi)}(g) = \gamma_x(g^{(\pi)}) =\left\{\begin{array}{ll}
\gm(g_{\pi(x)},H_{x}^z) &: x\in S_{\ast,\ast'}, \\
\gm(g_{\ast'}^{(\pi)} ,H_{\ast'}^z) &: x=\ast'. \\
\end{array}
\right.
\end{align*}

Then in view of \eqref{e:hinva}, the mean in the RHS of \eqref{e:ste3} becomes
\begin{equation}\label{e:meanlab}
\mathbb{E}\Big(\Big(\sum_{x\in S_{\ast_j,\ast_j'}}p_{x}c_{x}(h)\gm_{x}(h)\Big)^p\Big) =\frac{1}{|\Pi_j|}\mathbb{E}\Big(\sum_{\pi\in\Pi_j}\Big(\sum_{x\in S_{\ast_j,\ast'_j}}p_{x}c_{x}^{(\pi)}(h)\gm_{x}^{(\pi)}(h)\Big)^p\Big).
\end{equation}

To implement the strict inequality in \eqref{e:jen}, given  $p\ge 1$, $z\in\h$, $\alpha,L>0$, one defines $\kappa_{\ast}^{(p)}:=\kappa_{\ast}^{(p)}(z,\alpha,L,\cdot):\h^{S_{\ast,\ast'}}\to\R$, by
\begin{equation}\label{eq:cont_coeff}
\kappa_\ast^{(p)}(g):=\frac{\sum_{\pi\in\Pi} (\sum_{x\in S_{\ast,\ast'}}p_{x}c_{x}^{(\pi)}(g)\gamma_{x}^{(\pi)}(g))^p} {\sum_{\pi\in\Pi} \sum_{x\in S_{\ast,\ast'}}p_{x}{\gamma_{x}^{(\pi)}(g)}^{p}}.
\end{equation}
Here, $\ka_{\ast}^{(p)}$ depends on $\alpha,L$ \textit{via} $c_x^{(\pi)}(g) := c_x(g^{(\pi)})$. Indeed\footnote{Note that $c_x^{(\pi)}(g)\neq c_{\pi x}(g)$ in general. For example, if $x\in S_{\ast}$ and $\pi x\in S_{\ast'}$, then $c_x^{(\pi)}(g)$ is defined by a single sum while $c_{\pi x}(g)$ is a product of two sums, as in \eqref{e:cx1}--\eqref{e:cx2}.}, recall that $c_y(g)$ involves $c_{\ast'}(g)$ for $y\in S_{\ast'}$, and $c_{\ast'}(g)$ involves $g_{\ast'}$, which in turns depends on $\alpha,L$.

The RHS of \eqref{e:meanlab} now becomes
\begin{equation}\label{e:almostthere}
\frac{1}{|\Pi_j|}\mathbb{E}\Big({\ka_{\ast_j}^{(p)}\big(z,\alpha^{\omega}_{\ast_j'},L^{\omega}_{\ast_j'},h\big){\sum_{\pi\in\Pi_j}\sum_{x\in S_{\ast_j,\ast'_j}}p_{x}\gm_{x}^{(\pi)}(h)^p}}\Big).
\end{equation}
If all $\gamma_x^{(\pi)}(h)\approx 0$ then we are done, so the nontrivial work is when $h\notin B_r(H)$, where
\[
B_r(H):=\{g\in\h^{S_{\ast_j,\ast'_j}}: \gamma(g_{x},H_{x}^z)\le r \ \forall x\in S_{\ast_j,\ast'_j}\}.
\]
This case is controlled by Proposition~\ref{p:ka} below. To state it, recall that we assume $I$ is away from the Dirichlet spectrum in \eqref{e:nodir}. This implies, for the $L$ in $\kappa_{\ast}^{(p)}$,
\begin{equation}\label{e:diravoid}
\min_{z\in I+\ii[0,1]}|\sin \left(\sqrt{z}L\right) \sin\left(\sqrt{z}L_{\ast'}^0\right)| \ge \varepsilon_D>0\,.
\end{equation}

\begin{prp}\label{p:ka}
Let $I\subset \Sigma$ be compact and satisfy \eqref{e:nodir} and let $p>1$. For any $\ast=\ast_j$, there exist $\de_\ast(I,p)>0$, $\eps_\ast(I,\varepsilon_D)>0$, $\eta_\ast(I,\varepsilon_D)>0$ and $R_{\ast}:[0,\eps_\ast]\to[0,\infty)$ with $\lim_{\eps\to 0}R_{\ast}(\eps)=0$ such that for all $\eps\in[0,\eps_\ast]$,
\[
\sup_{z\in I+\ii[0,\eta_\ast]}\sup_{|L-L_{\ast'}^0| \le \eps}\sup_{|\alpha-\alpha_{\ast'}^0|\le \eps}\sup_{g\in \h^{S_{\ast,\ast'}}\setminus B_{R_\ast(\eps)}(H)}\ka_{\ast}^{(p)}(z,\alpha,L,g)\le 1-\de_\ast.
\]
\end{prp}

Proposition~\ref{p:ka} is remarkable as it gives a \emph{uniform contraction estimate} on the random variable. Using it in \eqref{e:almostthere}, we get the $(1-\delta_0)$ we were seeking in \eqref{e:ste2} (take $\delta_0=\min_{j\in \mathfrak{A}}\delta_{\ast_j}$), thus completing the proof of Theorem~\ref{thm:conti}.

The proof of Proposition~\ref{p:ka} is given in Appendix~\ref{app:uni}.

\subsection{Proof of pure AC spectrum}\label{sec:acran}

We may finally conclude with the proof of Theorem~\ref{thm:random}, which we recall:

\begin{thm}\label{main3}
Let $\T$ satisfy \emph{\textbf{(C0), (C1), (C2)}}. For all compact $I\subset \Sigma$, $I\cap\mathscr{D}=\emptyset$, and $p>1$, there is $\eps_0(I,p)>0$, $\eta_0(I,\varepsilon_D)>0$, $C_p'(I,\varepsilon_D)>0$
such that for any $v\in \T$,
\begin{equation}\label{e:realcool2}
\sup_{z\in I+\ii (0,\eta_0]} \expect(|\Im R_z^{\pm}(v)|^{-p}) \le C_p' \quad \text{and} \quad \sup_{z\in I+\ii (0,\eta_0]}\expect(|R_z^{\pm}(v)|^{\pm p}) \le C_p'
\end{equation}
for all $\eps\in [0,\eps_0)$ and all $(\alpha^\omega,L^\omega)$ satisfying \emph{\textbf{(P0)}}, \emph{\textbf{(P1)}} and \emph{\textbf{(P2)}}.

In particular, $\cH^{\omega}_{\eps}$ has pure AC spectrum almost surely in $I$.
\end{thm}
\begin{proof}
We showed after \eqref{e:e} that $\sup_{z\in I+\ii(0,\eta_0]} \expect\ap{|h_{x}^z|^{p}}<\infty$. So for $b=(v_-,v)$,
\[
\expect\left(|R_{z,\omega}^+(v)|^p\right) = |z|^{p/2} \expect\left(|h_v^z|^p\right) \le C_p
\]
for any $z\in I+\ii(0,\eta_0]$. On the other hand, observe that
\[
\gamma\left(\frac{-1}{\ii \Im h_v^z}, \frac{-1}{\ii \Im H_v^z}\right) = \gamma\left(\ii\Im h_v^z, \ii\Im H_v^z\right) \le \gamma\left(h_v^z,H_v^z\right) .
\]
In particular, using \eqref{e:modgam} with $\xi=\ii/\Im h_v^z$
and $\zeta = \ii/\Im H_v^z$, we get
\[
\left|\frac{1}{\Im h_v^z}\right|^p \le 2^{p-1}\left\{4^p\frac{\gamma(h_v^z,H_v^z)^p}{(\Im H_v^z)^p} + 2^p \left|\frac{1}{\Im H_v^z}\right|^p\right\} .
\]
Hence,
\[
\expect\left(|\Im h_v^z|^{-p}\right) \le C_p'
\]
for any $z\in I+\ii(0,\eta_0]$, by \eqref{e:realcool}. In particular,
\[
\expect\left(|R_{z,\omega}^+(v)|^{-p}\right) = |z|^{-p/2} \expect\left(|h_v^z|^{-p}\right) \le C_p'' \,.
\]
On the other hand, for any $b$, $R_z^-(t_b) = R_z^+(o_{\widehat{b}})$, and we know all results hold true for the tree $\mathbf{T}_{t_b}^-=\mathbf{T}_{o_{\widehat{b}}}^+$ (recall we are assuming \textbf{(C1)}). So we also have $\expect \left(|R_{z,\omega}^-(v)|^{p}\right) <\infty$ for any $p\in \Z$.

Finally, for $(\Im R_z^+(v))^{-p}$, denote $\sqrt{z}=a+\ii b$. Then
\begin{displaymath}
\Im h_v^z = \frac{a\Im R_z^+(v) - b\Re R_z^+(v)}{|z|} =: \frac{F_z}{|z|}\,.
\end{displaymath}
Hence,
\begin{align*}
  \expect\Big(\frac{1}{\Im R_z^+(v)^p}\Big)
  &= a^p\expect\Big(\frac{1}{|F_z + b\Re R_z^+(v)|^p}\Big)\\
  &= a^p \expect\Big(\frac{1}{|F_z+b\Re R_z^+(v)|^p}
    \mathbbm{1}_{\{|F_z|> |2b \Re R_z^+(v)|\}}(\omega)\Big) \\
  &\qquad+ a^p\expect\Big(\frac{1}{|F_z+b\Re R_z^+(v)|^p}
    \mathbbm{1}_{\{|F_z| \le |2b \Re R_z^+(v)|\}}(\omega)\Big)\,.
\end{align*}
For the first expectation, we have
$|F_z+b\Re R_z^+(v)|\ge |F_z| - b|\Re R_z^+(v)| \ge
\frac{1}{2}|F_z|$. For the second expectation, we make use of the
lower-bound for $\Im R_z^+(v)/\Im z$ proved in the Appendix (Lemma
\ref{lem:lower_bound}) to conclude that there is some $C>0$ with
$|F_z+b\Re R_z^+(v)| = a\Im R_z^+(v)\ge aC\Im z$ for any
$z,v,\omega$. Summarizing, we get
\begin{align*}
\expect\left(\Im R_z^+(v)^{-p}\right) &\le (2a)^p\expect\left(|F_z|^{-p}\right) + C^{-p}(\Im z)^{-p}\prob\left(|F_z|\le |2b\Re R_z^+(v)|\right) \\
&\le (2a)^p|z|^{-p}\expect\left((\Im h_v^z)^{-p}\right) + C^{-p}|z|^{-p}(\Im z)^{-p}(2b)^{p}\expect\left(\left|\frac{\Re R_z^+(v)}{\Im h_v^z}\right|^p\right)\,,
\end{align*}
using Markov's inequality.  The right-hand side is bounded uniformly
in $z\in I+\ii(0,\eta_0]$, since
$\frac{2b}{\Im z} = \frac{2\Im \sqrt{z}}{\Im z} = \frac{1}{\Re
  \sqrt{z}}$. This completes the proof of \eqref{e:realcool2}.

The consequence on pure AC spectrum follows from Theorem~\ref{thm:accrit}.
\end{proof}

\subsection{Uniform inverse moments}\label{sec:uninversemoments}

We now aim to prove Theorem~\ref{thm:moments}. The arguments of this section are inspired from their analogs \cite{AW2,AS3} in case of regular combinatorial graphs.

Note that \textbf{(P0)} implies that the distributions $\nu_j$ of $\alpha_v^\omega$ have compact support $\supp \nu_j \subseteq [\alpha_{\min}-\eps,\alpha_{\max}+\eps]$. In particular, all moments $M_{\varsigma} = \max_v\expect[|\alpha^\omega_v|^{\varsigma}] < \infty$ exist.

As in Section~\ref{sec:acunper}, we denote $R_z^+(k):= R_z^+(o_b)$ if $\ell(t_b)=k$. Note that here $R_z^+(o_b)$ is random; even if $\ell(t_b)=\ell(t_{b'})$ this doesn't imply that $R_z^+(o_b)$ and $R_z^+(o_{b'})$ are equal. However their distribution are the same, which justifies the notation for the purposes of this section.
We also let $Q = \max_{v\in \T} \deg(v)-1$ and $q=\min_{v\in \T} \deg(v)-1$.

Let $I\subset \Sigma$ be compact with $I\cap \mathscr{D}=\emptyset$. Fix $\eta_0(I,\varepsilon_D)$ as in Theorem~\ref{main3}. Given $\delta\in (0,1)$ and $\eps>0$, we introduce
\begin{equation}\label{e:sigmaac}
\sigma_{ac}^{\eps}(\delta) = \{\lambda \in \R : \prob(\Im R_{\lambda+\ii\eta}^+(j)>\delta)>\delta \quad \forall j\in \mathfrak{A}, \ \forall \eta \in (0,\eta_0)\} \, .
\end{equation}

\begin{lem}\label{lem:iinsigma}
There exist $\eps_0,\delta>0$ such that $I\subseteq \sigma_{ac}^\eps(\delta)$ for $\eps\le \eps_0$.
\end{lem}
\begin{proof}
By Theorem~\ref{main3}, fixing $p=2$ we may find $\eps_0(I)$ and $C_2$ such that $\expect(|\Im R_z^\pm(j)|^{-2})\le C_2$ so by Jensen $\expect(\Im R_z^\pm(j))\ge \expect(|\Im R_z^\pm(j)|^{-2})^{-1/2}\ge C_2^{-1/2}>0$. 

Suppose that $\prob(\Im R_{\lambda+\ii\eta}^+(j) > \delta) \le \delta$ for some $\lambda\in I$, $j\in\mathfrak{A}$ and $\eta\in (0,\eta_0)$. Then
\begin{align*}
\expect\left[\Im R_{\lambda+\ii\eta}^+(j)\right] & = \expect\left[\Im R_{\lambda+\ii\eta}^+(j)1_{\Im R^+_{\lambda+\ii\eta}(j)>\delta}\right] \\
& \quad + \expect\left[\Im R_{\lambda+\ii\eta}^+(j)1_{\Im R^+_{\lambda+\ii\eta}(j)\le \delta}\right] \le (C_2\delta)^{1/2} + \delta< C_2^{-1/2}
\end{align*}
if $\delta>0$ is small enough, yielding a contradiction. Here, we used Cauchy-Schwarz and the bound $\expect[|\Im R^+_{\lambda+\ii\eta}(j)|^2] \le C_2$. Hence, $I\subseteq \sigma_{ac}^{\eps}(\delta)$ for some $\delta>0$ as claimed.
\end{proof}

Given $z\in \C^+$ and $j\in\mathfrak{A}$, let
\[
F_z^j(x) = \prob\left(\Im R_z^+(j)\le x\right), \quad F_z(x)=\max_{j\in\mathfrak{A}}F_z^j(x)  \quad \text{and} \qquad H_z^j(x) = \prob\left(| \zeta^{z}_j |\le x\right)  \, .
\]

\begin{lem}           \label{lem:F(x)}
Let $I \subset \R$ be compact, $I\cap\mathscr{D}=\emptyset$, $\varsigma\ge 3$, $\lambda\in I$, $\eta>0$ and $z=\lambda+\ii\eta$. Then
\[
F_z(x) \le F_z(xy^{-2})^q + C_{\nu,Q,\varsigma}\left(y^\beta\,F_z\Big(\frac{4Qc_2}{c_1^2}\,y\Big)^q + y^{\varsigma}\right)
\]
for all $x>0$ and $y\in \left(0,c_I\right]$, where $\beta$ is from \emph{\textbf{(P3)}}, $c_I=\frac{c_1}{4Qc_2c_3}$ and $c_1,c_2,c_3>0$ are chosen as in \eqref{e:cj}.
\end{lem}
\begin{proof}
By \eqref{e:ASW} we have
\begin{equation}               \label{eq.zet}
\Im R_z^+(o_b) \ge |\zeta^{z}(b)|^2\sum_{b^+\in \cN_b^+} \Im R_z^+(o_{b^+})  .
\end{equation}
So given $y>0$, if $\Im R_z^+(o_b) \le x$, then either $|\zeta^{z}(b)| \le y$, or $\Im R_z^+(o_{b^+}) \le xy^{-2}$ for all $b^+\in \cN_b^+$. Recalling the matrix $M=(M_{i,j})$ we thus have by independence,
\[
F_z^j(x) \le F_z^1(xy^{-2})^{M_{j,1}}\cdots F_z^m(xy^{-2})^{M_{j,m}} + H_z^j(y)\le F_z(xy^{-2})^q+H_z^j(y)  .
\]
Now, assuming $c_1\le |S_z(L_b)|\le c_2$ and $|C_z(L_b)| \le \frac{c_1}{4Qc_2y}\le \frac{1}{4y}$ we have by \eqref{e:1}
\begin{align*}
H_{z}^j(y) & = \prob\Big(\Big|S_z(L_b)\Big(\alpha_{t_b}+\frac{S_z'(L_b)}{S_z(L_b)}-\sum_{b^+\in \cN_b^+}R_z^+(o_{b^+})\Big)\Big| \ge y^{-1}\Big) \\
& \le \prob\Big( |\alpha_{t_b} S_z(L_b)| + |C_z(L_b)| + |S_z(L_b)|\sum_{b^+\in \cN_b^+} |R_z^+(o_{b^+})| \ge y^{-1}\Big) \\
& \le \prob\Big( c_2|\alpha_{t_b}| + (4y)^{-1} + c_2\sum_{b^+\in \cN_b^+} |R_z^+(o_{b^+})| \ge y^{-1}\Big) \\
& \le \prob\left(  |\alpha_{t_b} | \ge (4c_2y)^{-1}\right) + \prob\Big(\sum_{b^+\in \cN_b^+} |R_z^+(o_{b^+})| \ge (2c_2y)^{-1} \Big) \\
& \le (4 c_2y)^{\varsigma} M_{\varsigma} + \sum_{k=1}^mM_{j,k} \prob\left(|R_z^+(k)| \ge (2Qc_2y)^{-1}\right) ,
\end{align*}
where $M_{\varsigma} = \max_v\expect[|\alpha^\omega_v|^{\varsigma}] < \infty$. But for $t=\frac{1}{2Qc_2y}$,
\begin{multline}\label{e:Hgamma}
 \prob\left(|R_z^+(o_b)| \ge t\right) = \prob\left(\frac{|\zeta^z(b)-C_z(L_b)|}{|S_z(L_b)|}\ge t\right) \le \prob\left(\frac{|\zeta^z(b)|+(c_1/4Qc_2y)}{c_1}\ge t\right)  \\
 \le \prob\Big(|\zeta^z(b)|\ge \frac{c_1t}{2}\Big)\le \prob\Big(|\zeta^z(b)S_z(L_b)|\ge \frac{c_1^2t}{2}\Big)\\
\le \prob\Big(\big|\alpha_{t_b}+\Re \frac{S_z'(L_b)}{S_z(L_b)} - \sum_{b^+\in \cN_b^+} \Re R_z^+(o_{b^+})\big| \le \frac{2}{c_1^2t} \text{ and } \sum_{b^+\in \cN_b^+} \Im R_z^+(o_{b^+}) \le \frac{2}{c_1^2t} \Big) \\\
 \le \prob\Big(\alpha_{t_b} \text{ lies in an interval of length } \frac{4}{c_1^2t} \text{ and } \sum_{b^+\in \cN_b^+} \Im R_z^+(o_{b^+}) \le \frac{2}{c^2_1t} \Big) \\
 \le  C_{\nu} \Big(\frac{4}{c_1^2 t}\Big)^\beta\prob\Big(\sum_{b^+\in \cN_b^+} \Im R_z^+(o_{b^+}) \le \frac{2}{c_1^2t}\Big)  \le C_{\nu} \Big(\frac{4}{c_1^2 t}\Big)^\beta F_z\left( \frac{4Qc_2y}{c_1^2}\right)^q  .
\end{multline}
Here we used that $-\frac{S_{z}'(L)}{S_{z}(L)}$ is Herglotz and that $\{R_z^+(o_{b^+})\}$ are independent of $\alpha_{t_b}$, as follows from \eqref{e:WTtronq}, and bounded the probability by first conditioning over the random variables different from $\alpha_{t_b}^\omega$, so that the ``interval'' above is fixed.

Thus, if $c_I = \frac{c_1}{4qc_2c_3}$, where $ |C_z(L_b)|\le c_3$, then for any $0<y\le c_I \le \frac{c_1}{4qc_2|C_z(L_b)|}$,
\[
H_{z}(y) \le (4c_2)^{\varsigma}M_{\varsigma} \cdot y^{\varsigma} + C_{\nu}\Big(\frac{8Qc_2}{c_1^2}\Big)^\beta Q\cdot y^\beta \cdot F_z\Big(\frac{4Qc_2y}{c_1^2}\Big)^q  .    \qedhere
\]
\end{proof}

\begin{thm}                 \label{thm:F(x)}
Let $I \subset \sigma_{ac}^{\eps}(\delta)$ be bounded, $\lambda\in I$, $\varsigma\ge 3$, $\eta\in (0,\eta_0)$ and $z=\lambda+\ii\eta$. Then
\[
F_{z}(x) \le C_{\nu,Q,\delta,\varsigma}x^{\beta\varsigma/5}
\]
for all $x\in (0,x_{\delta}]$, for some $C_{\nu,Q,\delta}<\infty$ and $x_{\delta}=x(\delta,\nu,Q,\varsigma,I)>0$.
\end{thm}
\begin{proof}
The theorem is proved by gradually improving on the decay. First, take $y=x^{1/4}$ in Lemma~\ref{lem:F(x)} to get for $x\in (0,c_I^4]$,
\[
F_z(x) \le F_z(x^{1/2})^q + C_{\nu,Q}\Big(x^{\beta/4}F_z\Big(\frac{4Qc_2}{c_1^2}\,x^{1/4}\Big)^q + x^{\varsigma/4}\Big)  ,
\]
so using the bounds $x^{\beta/2}[F_z(\frac{4Qc_2}{c_1^2}\,x^{1/2})]^q \le x^{\beta/2}$ and $x^{\varsigma/2} \le x^{\beta/2}$ for small $x$, we get
\begin{equation}                                     \label{eq:fnal}
F_z(x^2) \le F_z(x)^q + C'_{\nu,Q}x^{\beta/2} \, .
\end{equation}
Since $\lambda\in \sigma_{ac}^{\eps}(\delta)$, we have $F_z(x) =\max_j \prob(\Im R_z^+(j)\le \delta) \le 1-\delta$ for any $x \in [0,\delta]$. Choose $x_0 \le \min(\delta, c_I^4,(\frac{\delta}{4C'_{\nu,Q}})^{8/\beta})$. For $\alpha$ small enough, we have $(1-\delta) \le (1-\frac{\delta}{2})x_0^{2\alpha}$. So there is some $\alpha_0 \in (0,\beta/16]$ such that
\[
F_z(x_0) \le \Big(1-\frac{\delta}{2}\Big)x_0^{2\alpha_0}  .
\]
Now define recursively $x_n = x_{n-1}^2$. We show by induction that $F(x_n) \le (1-\frac{\delta}{2}) x_n^{2\alpha_0}$. We know this for $n=0$. Next, for $n\ge 1$, using \eqref{eq:fnal}, induction and $q \ge 2$, we have
\begin{align*}
F_z(x_n) & \le F_z(x_{n-1})^2 + Cx_n^{\beta/4} \le \Big(1-\frac{\delta}{2}\Big)^2x_n^{2\alpha_0} +  x_n^{2\alpha_0}\cdot Cx_0^{\beta/4-2\alpha_0} \\
& \le \Big[\Big(1-\frac{\delta}{2}\Big)^2 + Cx_0^{\beta/8}\Big]x_n^{2\alpha_0} \le \Big(1-\frac{\delta}{2}\Big) x_n^{2\alpha_0}  ,
\end{align*}
where we used $x_n \le x_0$, $2\alpha_0\in (0,\beta/8]$ and $x_0 \le (\frac{\delta}{4C})^{8/\beta}$. Hence, if $x\in (0,x_0]$, so that  $x\in (x_{n+1},x_n]$ for some $n$, using that $F_z$ is monotone increasing, we get
\[
F_z(x) \le F_z(x_n) \le \Big(1-\frac{\delta}{2}\Big)x_{n+1}^{\alpha_0} \le \Big(1-\frac{\delta}{2}\Big)x^{\alpha_0} \le x^{\alpha_0}  .
\]
This proves a first power decay which we now improve. Suppose we have
\[
F_z(x) \le c x^{\alpha}
\]
for some $c>0$ and all $x\in (0,x_0]$. Taking $y=x^{\frac{1}{3}\frac{\varsigma}{1+\varsigma}}$ in Lemma~\ref{lem:F(x)}, we get
\begin{align*}
F_z(x) & \le F_z(x^{\frac{3+\varsigma}{3(1+\varsigma)}})^q + Cx^{\frac{\beta}{3}\frac{\varsigma}{1+\varsigma}}F_z\Big(\frac{4Qc_2}{c_1^2}\,x^{\frac{1}{3}\frac{\varsigma}{1+\varsigma}}\Big)^q + Cx^{\frac{1}{3}\frac{\varsigma^2}{1+\varsigma}} \\
& \le cx^{\frac{q}{3}\frac{3+\varsigma}{1+\varsigma}\alpha} + \tilde{C} x^{\frac{1}{3}\frac{\varsigma}{\varsigma+1}(\beta+q\alpha)} + Cx^{\frac{1}{3}\frac{\varsigma^2}{1+\varsigma}}   .
\end{align*}
If $q \ge 3$, then $\frac{q}{3}\frac{3+\varsigma}{1+\varsigma} \ge \frac{3+\varsigma}{1+\varsigma} >1$. Also, if $\alpha \le \frac{\beta\varsigma}{4}$, then $\frac{\beta+q\alpha}{3} \frac{\varsigma}{1+\varsigma} \ge \frac{\beta+3\alpha}{3}\frac{\varsigma}{1+\varsigma}>\alpha$, since $\beta\varsigma + 3\alpha\varsigma > 3\alpha + 3\alpha \varsigma$. Finally, $\frac{\varsigma}{3(1+\varsigma)} \ge \frac{1}{4} \ge \frac{\beta}{4}$, since $\varsigma \ge 3$, so $\frac{\varsigma^2}{3(1+\varsigma)} \ge \frac{\beta\varsigma}{4}$. We thus showed that if $q \ge 3$, then the decay power is strictly increased as long as $\alpha \le \frac{\beta\varsigma}{4}$. Iteration thus proves the theorem\footnote{The exponent of the first term is increased by $\frac{2}{1+\varsigma}$ at each step, the second one by at least $\frac{\alpha_0}{3(1+\varsigma)}$, as for the last one, it already has the required decay. After finitely many steps, the exponent thus reaches $\frac{\beta\varsigma}{4}$.} when $q\ge 3$ (because $x^{\beta\varsigma/4} \le x^{\beta\varsigma/5}$).

The rest of the proof is devoted to the case $q=2$, where we need to improve again. Using \eqref{e:ASW} twice, we have
\begin{align*}
&\Im R_z^+(o_b)  \ge |\zeta^z(b)|^2 \sum_{b^+} |\zeta^z(b^+)|^2\sum_{b^{++}} \Im R_z^+(o_{b^{++}}) \\
& \quad\ge \sum_{b^+} \frac{c_2^{-2}|\zeta^z(b^+)|^2}{(|\alpha_{t_{b}}| + |\frac{S_z'(L_{b})}{S_z(L_{b})}| + \sum_{b^{+} }\frac{|\zeta^z(b^{+})|}{|S_z(L_{b^{+}})|} + \sum_{b^{+}}|\frac{C_z(L_{b^{+}})}{S_z(L_{b^{+}})}|)^2}\sum_{b^{++}} \Im R_z^+(o_{b^{++}}) \\
&\quad \ge \sum_{b^+ } \frac{c_1^2c_2^{-2}|\zeta^z(b^+)|^2}{(c_1|\alpha_{t_{b}}| + |C_z(L_{b})| + \sum_{b^{+} }|\zeta^z(b^{+})| + \sum_{b^{+}}|C_z(L_{b^{+}})|)^2}\sum_{b^{++}} \Im R_z^+(o_{b^{++}}).
\end{align*}
Define the events $E_0 = \{ |\alpha_{t_{b}}| \le y^{-1}\}$,
\[
E_1 = \{|\zeta^z(b^{+})| >y^{-1} \text{ for at least two } b^{+}\}  ,
\]
\[
E_2 = \{|\zeta^z(b^{+}_0)|>y^{-1} \text{ for exactly one } b^{+}_0\} ,
\]
\[
E_3 = \{|\zeta^z(b^{+})| \le y^{-1} \text{ for all } b^{+}\}   .
\]
Using an estimate from \eqref{e:Hgamma}, we have
\[
\prob(E_1) \le \frac{Q(Q+1)}{2}\left[\prob(|\zeta^z(b^{+})|>y^{-1})\right]^2 \le c_{\nu,Q} \Big(\frac{2}{c_1y^{-1}}\Big)^{2\beta} F_z\Big(\frac{1}{c_1y^{-1}}\Big)^4  .
\]
For $E_0\cap E_2$, $|\zeta^z(b^{+})| \le y^{-1}$ for all $b^{+}\neq b_0^{+}$, so we have
\begin{multline*}
\frac{|\zeta^z(b_0^{+})|}{c_1|\alpha_{t_{b}}| + \sum_{t_{b'}\sim t_b}|C_z(L_{b'})| + \sum_{b^{+}}|\zeta^z(b^+)|} \\
= \frac{1}{1+\frac{c_1|\alpha_{t_{b}}| + \sum_{t_{b'}\sim t_b}|C_z(L_{b'})|+\sum_{b^{+}\neq b_0^{+}}|\zeta^z(b^{+})|}{|\zeta^z(b_0^{+})|}} \ge \frac{1}{Q+c_1+y(Q+1)c_3} \, ,
\end{multline*}
since $\frac{c_1|\alpha_{t_{b}}| + \sum_{t_{b'}\sim t_b}|C_z(L_{b'})|+\sum_{b^{+}\neq b_0^{+}}|\zeta^z(b^{+})|}{|\zeta^z(b_0^{+})|} \le \frac{c_1y^{-1} + (Q+1)c_3+Qy^{-1}}{y^{-1}}$. Assuming $y\le 1$, the RHS is $\ge c$. Hence, for $\tilde{c} = c_1^2c_2^{-2}c^2$ we have
\[
\prob\left(\{\Im R_z^+(o_b) \le x\} \cap E_0 \cap E_2\right) \le \prob\Big(\tilde{c}^2 \sum_{b_0^{++} \in \cN_{b_0^+}} \Im R_z^+(o_{b_0^{++}}) \le x\Big) \le F_z(\tilde{c}^{-2}x)^2 \, .
\]
Finally, for $E_0 \cap E_3$, we use \eqref{eq.zet} along with $|\zeta^z(b)| \ge \frac{c_1c_2^{-1}}{(Q+c_1)y^{-1}+(Q+1)c_3}$. If $y$ is small, this is $\ge Cy$, so we get
\[
\prob\left(\{\Im R_z^+(o_b) \le x\} \cap E_0 \cap E_3\right) \le \prob\Big(C^2y^2 \sum_{b^+} \Im R_z^+(o_{b^+})\le x\Big) \le F_z(C^{-2}xy^{-2})^2 \, .
\]
Estimating $\prob(E_0^c) \le y^{\varsigma}M_{\varsigma}$ by Chebyshev, we thus showed that
\[
F_z(x) \le Cy^{2\beta}F_z(c_1^{-1}y)^4 + F_z(\tilde{c}^{-2}x)^2 + F_z(C^{-2}xy^{-2})^2 + c' y^{\varsigma} \, .
\]
Assuming we showed that $F_z(x) \le  c_0 x^{\alpha}$, then choosing $y = x^{\frac{\varsigma}{3+4\varsigma}}$, we get
\[
F_z(x) \le \tilde{C} x^{\frac{2\beta\varsigma}{3+4\varsigma}} x^{\frac{4\alpha \varsigma}{3+4\varsigma}} + \tilde{c} x^{2\alpha} + \tilde{C}' x^{\frac{3+2\varsigma}{3+4\varsigma}(2\alpha)} + c'x^{\frac{\varsigma^2}{3+4\varsigma}} \, .
\]
To get $\frac{2\beta\varsigma+4\alpha\varsigma}{3+4\varsigma}>\alpha$ we must have $2\beta\varsigma + 4\alpha \varsigma> 3\alpha+4\alpha\varsigma$, i.e. $\alpha< \frac{2\beta\varsigma}{3}$, so $\alpha \le \frac{\beta\varsigma}{5}$ suffices. Next, $\frac{6+4\varsigma}{3+4\varsigma}>1$. Finally, $\frac{\varsigma}{3+4\varsigma} \ge \frac{1}{5} \ge \frac{\beta}{5}$ since $\varsigma \ge 3$. We thus showed the decay exponent can be strictly improved up to $\frac{\beta\varsigma}{5}$.
\end{proof}

\begin{proof}[Proof of Theorem~\ref{thm:moments}]
Given $p\ge 1$, choose $\varsigma$ such that $p<\frac{\beta\varsigma}{5}$. By Lemma~\ref{lem:iinsigma}, we have $I\subseteq \sigma_{ac}^\eps(\delta)$ for some $\delta>0$ and $\eps\le \eps_0$.
As $p\ge 1$, given $\lambda\in I$, we have by the layer-cake representation,
\[
\expect\left(|\Im R_z^+(o_b)|^{-p}\right) = p \int_0^{\infty} t^{p-1}\prob\left(|\Im R_z^+(o_b)|^{-1} \ge t\right) \dd t = p\int_0^{\infty}t^{p-1}F_{z}(t^{-1}) \,\dd t\, .
\]
By Theorem~\ref{thm:F(x)}, denoting $t_{\delta} = x_{\delta}^{-1}$, we know $F_{z}(t^{-1})\le C_{\nu,Q}t^{-\beta\varsigma/5}$ for all $t \ge t_{\delta}$. Hence
\[
\expect\left(|\Im R_z^+(o_b)|^{-p}\right) \le p\int_0^{t_{\delta}} t^{p-1}\,\dd t + pC \int_{t_{\delta}}^{\infty} t^{p-1-\beta\varsigma/5}\,\dd t = t_{\delta}^p + \frac{5p}{\beta\varsigma - 5p} \frac{C}{t_{\delta}^{\beta\varsigma/5 - p}}\, .
\]
We may assume $t_{\delta} \ge 1$ by taking a smaller $x_{\delta}$ if necessary. Since this holds for any $\lambda\in I$ and $\eta\in (0,\eta_0)$, we get
\[
\sup_{\lambda\in I} \sup_{\eta \in (0,\eta_0)}\expect\left(|\Im R_z^+(o_b)|^{-p}\right) \le t_{\delta}^p\,\Big(1+\frac{5p}{\beta\varsigma-5p}C_{\nu,Q}\Big) \, . \qedhere
\]
\end{proof}

\appendix

\section{Proofs of some technical facts}\label{sec:app1}

\subsection{General results}

\begin{proof}[Proof of Lemma~\ref{lem:A.2}]
Given $f\in L^2(\mathbf{T})$, define $G^{z} f := \int_{\mathcal{T}} G^{z}(x,y)f(y)\,\dd y$. We should show that $G^{z}f\in D(\cH_{\mathbf{T}})$ and $(\cH_{\mathbf{T}}-z)G^{z} f = f$.

Assume $f$ is continuous on $\mathcal{T}$ and supported in a ball $\Lambda\subset \mathbf{T}$. Let $x=(b,x_b)\in \mathcal{T}$, with $x_b\in (0,L_b)$. Fix $o,v\in \mathbf{T}$ such that $b \in B(\mathbf{T}_o^+)\cap B( \mathbf{T}_v^-)$ and $\Lambda\subset \mathbf{T}_o^+ \cap \mathbf{T}_v^-$. By definition,
\[
(G^{z}f)(x) = V_{z;o}^+(x)\int_{\mathcal{T}_x^-}\frac{U_{z;v}^-(y)}{\cW_{v,o}^{z}(y)}f(y)\,\dd y  + U_{z;v}^-(x)\int_{\mathcal{T}_x^+}\frac{V_{z;o}^+(y)}{\cW_{v,o}^{z}(y)}f(y)\,\dd y \,.
\]
Hence,
\begin{equation}\label{e:gprime}
(G^{z}f)'(x) = (V_{z;o}^+)'(x)\int_{\mathcal{T}_x^-}\frac{U_{z;v}^-(y)}{\cW_{v,o}^{z}(y)}f(y)\,\dd y  + (U_{z;v}^-)'(x)\int_{\mathcal{T}_x^+}\frac{V_{z;o}^+(y)}{\cW_{v,o}^{z}(y)}f(y)\,\dd y \,,
\end{equation}
where the term $\frac{V_{z;o}^+(x)U_{z;v}^-(x^-)}{\cW_{v,o}^{z}(x^-)}f(x^-) -\frac{U_{z;v}^-(x)V_{z;o}^+(x^+)}{\cW_{v,o}^{z}(x^+)}f(x^+)$ from Leibniz's rule canceled, all functions being continuous. Next,
\[
(G^{z}f)''(x) = (V_{z;o}^+)''(x)\int_{\mathcal{T}_x^-}\frac{U_{z;v}^-(y)}{\cW_{v,o}^{z}(y)}f(y)\,\dd y  + (U_{z;v}^-)''(x)\int_{\mathcal{T}_x^+}\frac{V_{z;o}^+(y)}{\cW_{v,o}^{z}(y)}f(y)\,\dd y - f(x) \,,
\]
where we used that $\frac{(V_{z;o}^+)'(x)U_{z;v}^-(x^-)}{\cW_{v,o}^{z}(x^-)}f(x^-) -\frac{(U_{z;v}^-)'(x)V_{z;o}^+(x^+)}{\cW_{v,o}^{z}(x^+)}f(x^+) = -f(x)$. Recalling that $\psi''=(W-z)\psi$ for $\psi=V_{z;o}^+$, $U_{z;v}^-$, we get $(G^{z}f)'' = (W-z)G^{z}f - f$, so $(\cH_{\mathbf{T}}-z)G^{z} f = f$.

For the boundary conditions, note that $G^{z}f(x) = \int_{\mathbf{T}} G^{z}(x,y)f(y)\,\dd y$ and $(G^{z}f)'(x) = \int_{\mathbf{T}} \partial_x G^{z}(x,y)f(y)\,\dd y$ by \eqref{e:gprime}. So it suffices to check that $x\mapsto G^{z}(x,y)$ satisfies the $\delta$-conditions. But this follows immediately since $V_{z;o}^+\in D(\cH_{\mathbf{T}_o^+}^{\max})$ and $U_{z;v}^-\in D(\cH_{\mathbf{T}_v^-}^{\max})$.

Finally, as $H_{\mathbf{T}}$ is self-adjoint, we have by the spectral theorem $\|f\|^2 = \|(\cH_{\mathbf{T}}-z)G^{z}f\|^2 = \int_{\mathbf{T}} |\lambda-z|^2\,\dd\mu_{G^{z}f} \ge |\Im z|^2\int_{\mathbf{T}} \dd\mu_{G^{z}f} = |\Im z|^2\|G^{z}f\|^2$, i.e.\ $\|G^{z}f\| \le \frac{1}{\Im z} \|f\|$. This holds on the subspace of continuous $f$ of compact support. By the density of such functions, $G^{z}$ extends to a bounded operator on $L^2(\mathbf{T})$ satisfying $\|G^{z}\|\le \frac{1}{\Im z}$.

We proved that $G^{z}f\in D(\cH_{\mathbf{T}})$ and $(\cH_{\mathbf{T}}-z)G^{z}f=f$ assuming $f$ is continuous of compact support. For general $f\in L^2(\mathbf{T})$, take a sequence $(f_j)$ of such functions with $f_j\to_{L_2} f$. We showed $G^{z}f_j\in W^{2,2}_{\max}(\mathbf{T})$ for each $j$; this space being complete, we obtain a subsequence $(G^{z}f_{j_k})$ converging in $W^{2,2}_{\max}(\mathbf{T})$. The limit must be $G^{z}f$ since $G^{z}f$ is the $L^2$ limit of $G^{z}f_j$. It follows that $(\cH_{\mathbf{T}}-z)G^{z}f=f$ a.e., and for each $x$, we have $G^{z}f(x) = \lim_k G^{z} f_{j_k}(x)$ and $(G^{z}f)'(x)=\lim_k (G^{z}f_{j_k})'(x)$, so we deduce that $G^{z}f$ satisfies the boundary conditions.
\end{proof}

\begin{rem}\label{rem:newgreenrep}
Fix an edge $b\in B(\T)$ and choose $o,v$ with $b\in B(\mathbf{T}_o^+)\cap B(\mathbf{T}_v^-)$. Then for $x= (b,x_b),y=(b,y_b)\in \mathcal{T}$, we can also express
\begin{equation}\label{e:othergreen}
G^{z}_{\mathbf{T}}(x,y) = \begin{cases} -\frac{\phi_{z;b}^-(x) \phi_{z;b}^+(y)}{R_{z}^+(o_b)+R_{z}^-(o_b)} &\text{if } y\in \mathcal{T}_x^+,\\ -\frac{\phi_{z;b}^-(y) \phi_{z;b}^+(x)}{R_{z}^+(o_b)+R_{z}^-(o_b)} &\text{if } y\in \mathcal{T}_x^-,\end{cases}
\end{equation}
where $\phi_{z;b}^-(x) = \frac{U_{z;v}^-(x)}{U_{z;v}^-(o_b)}$ and $\phi_{z;b}^+(x) = \frac{V_{z;o}^+(x)}{V_{z;o}^+(o_b)}$. This follows immediately from \eqref{e:greenasw}, \eqref{e:WT}, since the Wronskian $\cW_{v,o}^{z}(y)= \cW_{v,o}^{z}(b, y_b)$ is constant for $y_b\in [0, L_b]$. Note that
\begin{equation}\label{e:phiplumoi}
\phi_{z;b}^{\pm}(x) = C_{z}(x_b) \pm R_{z}^{\pm}(o_b) S_{z}(x_b)
\end{equation}
for $x= (b, x_b)\in \mathcal{T}$. On the other hand,
\[
G^{z}_{\mathbf{T}}(y,x) = G^{z}_{\mathbf{T}}(x,y)
\]
for  $x= (b,x_b),y=(b,y_b)\in \mathcal{T}$, since $x\in \mathcal{T}_y^{\pm}\iff y\in \mathcal{T}_x^{\mp}$. This implies
\[
\langle f_1, G^{z} f_2 \rangle = \langle f_2, G^{z} f_1 \rangle
\]
for any real-valued $f_j$ supported on $e(b)$. Hence, for any $f = f_1 + \ii f_2$ supported in $e$,
\begin{equation}\label{e:greenrealenough}
\langle f, G^{z} f \rangle = \langle f_1, G^{z} f_1 \rangle + \langle f_2, G^{z} f_2 \rangle \,.
\end{equation}
\end{rem}

\begin{lem}\label{lem:ASWrep}
Fix $b\in B(\T)$ and suppose $R_{\lambda}^{\pm}(o_b) := R_{\lambda+\ii0}^{\pm}(o_b)$ exist and are not both zero. Then for any $f$ supported in $e(b)$,
\begin{equation}\label{e:nicerep}
\Im \langle f, G^{\lambda} f\rangle = \frac{\Im R_{\lambda}^+(o_b)g_f^-(\lambda) + \Im R_{\lambda}^-(o_b)g_f^+(\lambda)}{|R_{\lambda}^+(o_b)+R_{\lambda}^-(o_b)|^2} \,,
\end{equation}
where
\[
g_f^{\pm}(\lambda) = \left|\langle f,\Re \phi^{\pm}_{\lambda;b}\rangle_{L^2[0, L_b]}\right|^2 + \left|\langle f,\Im \phi^{\pm}_{\lambda;b}\rangle_{L^2[0, L_b]} \right|^2 .
\]
In particular, if $\Im R_{\lambda}^{\pm}(o_b)=0$ for all $b\in B$, then $\Im \langle f, G^{\lambda}f\rangle=0$ for all $f\in L^2(\mathbf{T})$.
\end{lem}
This lemma was stated without proof in \cite[eq (A.15)]{ASW06}.
\begin{proof}
First note that it suffices to prove this for real-valued $f$. In fact, for $f=f_1+\ii f_2$, we then use \eqref{e:greenrealenough} and deduce the result, since $g_{f_1}^+(\lambda) + g_{f_2}^+(\lambda) = g_f^+(\lambda)$ as easily checked and similarly for $g_f^-(\lambda)$. So assume $f$ is real.

Since $b$ and $\lambda$ are fixed, we shall denote $R^{\pm} :=R_{\lambda}^{\pm}(o_b)$ and $\phi^{\pm} :=\phi_{\lambda;b}^{\pm}$, and drop $b$ indices in $x_b$ and $f_b$ for simplicity.

We have by \eqref{e:othergreen},
\begin{multline*}
\langle f, G^{\lambda} f\rangle = \frac{-1}{R^++R^-}\int_0^{L_b} f(x)\Big(\int_0^{x} \phi^+(x)\phi^-(y)f(y)\,\dd y \\
+ \int_{x}^{L_b} \phi^-(x)\phi^+(y)f(y)\,\dd y\Big)\,\dd x \,.
\end{multline*}
But by \eqref{e:phiplumoi},
\[
\phi^+(x)\phi^-(y) = C(x)C(y) - R^-C(x)S(y) + R^+S(x)C(y) - R^+R^-S(x)S(y) ,
\]
\[
\phi^-(x)\phi^+(y) = C(x)C(y) + R^+C(x)S(y)-R^-S(x)C(y)-R^+R^-S(x)S(y).
\]
Thus, $\phi^-(x)\phi^+(y) = \phi^+(x)\phi^-(y) + (R^++R^-)(C(x)S(y)-S(x)C(y))$. Hence,
\begin{multline*}
\langle f,G^{\lambda}f\rangle = \frac{-1}{R^++R^-} \int_{[0, L_b]^2} f(x)f(y)\phi^+(x)\phi^-(y)\,\dd y\,\dd x \\
- \int_0^{L_b}\int_x^{L_b}f(x)f(y)\left(C(x)S(y)-S(x)C(y)\right)\,\dd y\,\dd x\,.
\end{multline*}
Since $f$ is real-valued, we get
\begin{align}\label{e:concerningcross}
\Im \langle f,G^{\lambda} f\rangle &= \frac{\Im (R^++R^-)}{|R^++R^-|^2}\int_{[0, L_b]^2} f(x)f(y)\Re[\phi^+(x)\phi^-(y)]\,\dd y\,\dd x \nonumber \\
&\quad- \frac{\Re(R^++R^-)}{|R^++R^-|^2}\int_{[0, L_b]^2}f(x)f(y)\Im[\phi^+(x)\phi^-(y)]\,\dd y\,\dd x \nonumber \\
&= \frac{1}{|R^++R^-|^2} \big\{\Im(R^++R^-)[\langle f,\Re \phi^+\rangle\langle f,\Re \phi^-\rangle - \langle f,\Im \phi^+\rangle\langle f,\Im \phi^-\rangle] \\
&\quad-\Re(R^++R^-)[\langle f, \Im \phi^+\rangle\langle f,\Re \phi^-\rangle + \langle f, \Re \phi^+\rangle \langle f, \Im \phi^-\rangle]\big\}\,. \nonumber
\end{align}

Now $\Re \phi^{\pm} = C(x) \pm (\Re R^{\pm})S(x)$ and $\Im \phi^{\pm} = (\pm \Im R^{\pm}) S(x)$. Hence, the term in curly brackets is
\begin{align}\label{e:bigcalcu}
\Im (R^+&+R^-)\left[\langle f,C\rangle + \Re R^+\langle f,S\rangle\right]\left[\langle f,C\rangle - \Re R^-\langle f,S\rangle\right]\\
&\qquad\qquad-\Im (R^++R^-)\left[\Im R^+\langle f,S\rangle\right]\left[-\Im R^-\langle f,S\rangle\right]\nonumber\\
&\qquad-\Re(R^++R^-)\left[\Im R^+\langle f,S\rangle\right]\left[\langle f,C\rangle-\Re R^-\langle f,S\rangle\right]\nonumber\\
&\qquad-\Re (R^++R^-)\left[\langle f,C\rangle + \Re R^+\langle f,S\rangle\right]\left[-\Im R^-\langle f,S\rangle\right]\nonumber\\
&= \langle f,C\rangle^2\Im (R^++R^-) + \langle f,C\rangle\langle f,S\rangle \big[(\Im R^++\Im R^-)(\Re R^+-\Re R^-) \nonumber\\
&\qquad- (\Re R^++\Re R^-)(\Im R^+-\Im R^-)\big]\nonumber\\
&\qquad+\langle f,S\rangle^2 \big[(\Im R^++\Im R^-)(\Im R^+\Im R^- - \Re R^+\Re R^-)\nonumber\\
&\qquad+(\Re R^++\Re R^-)(\Im R^+\Re R^-+\Im R^-\Re R^+)\big]\nonumber\\
&= \langle f,C\rangle^2 \Im(R^++R^-) +2\langle f,C\rangle\langle f,S\rangle \left[\Im R^-\Re R^+ - \Im R^+\Re R^-\right]\nonumber\\
        &\quad+\langle f,S\rangle^2\big[(\Im R^+)^2\Im R^- + \Im R^+(\Im R^-)^2 +\Im R^-(\Re R^+)^2\nonumber\\
        &\qquad + \Im R^+(\Re R^-)^2\big] . \nonumber
\end{align}
On the other hand,
\begin{multline*}
\Im R^+ g_f^- + \Im R^- g_f^+ = \Im R^+\left[\left(\langle f,C\rangle - \Re R^-\langle f,S\rangle\right)^2 + (\Im R^-)^2\langle f,S\rangle^2\right]\\
+ \Im R^-\left[\left(\langle f,C\rangle + \Re R^+\langle f,S\rangle\right)^2+(\Im R^+)^2\langle f,S\rangle^2\right],
\end{multline*}
which is exactly the expression at the end of \eqref{e:bigcalcu}. This completes the proof of \eqref{e:nicerep}.

Finally, if $\Im R_{\lambda}^{\pm}(o_b) = 0$ for all $b\in B$, then $\Im \langle f_b,G^{\lambda}f_b\rangle=0$ for all $f_b$ supported in $b$, by \eqref{e:nicerep}. The same ideas show that $\Im \langle f_b, G^{\lambda} f_{b'}\rangle = 0$ in this case. In fact, we don't need to go through all the above calculations, just note that in \eqref{e:concerningcross}, we get $\Im \phi^{\pm}=(\pm \Im R^{\pm})S = 0$. 

It follows that $\Im\langle f,G^{\lambda}f\rangle=0$ for all $f\in L^2(\mathbf{T})$.
\end{proof}

\begin{proof}[Proof of Lemma~\ref{lem:ASW}]
We may find an $L^2$ solution $\widetilde{U}_{z;v}^-$ on $\mathbf{T}_v^-$ satisfying the $\delta$-conditions at vertices $w\notin \{o_b,v\}$, the Neumann condition at $o_b$ and $\widetilde{U}_{z;v}^-(v)=1$. As in \eqref{e:greenasw}, we get
  \begin{displaymath}
G_{\mathbf{T}_{o_b}^+}^{z}(x,y) = \begin{cases} \frac{\widetilde{U}_{z;v}^-(x) V_{z;o}^+(y)}{\cW^{z}_{v,o,o_b}(x)} &\text{if } y\in \mathcal{T}_x^+,\\ \frac{\widetilde{U}^-_{z;v}(y) V_{z;o}^+(x)}{\cW^{z}_{v,o,o_b}(x)} &\text{if } y\in \mathcal{T}_x^-,\end{cases}
\end{displaymath}
where $\cW_{v,o,o_b}^{z}(x) = V_{z;o}^+(x)(\widetilde{U}_{z;v}^-)'(x)-(V_{z;o}^+)'(x)\widetilde{U}^-_{z;v}(x)$.

In particular, $\frac{1}{G_{\mathbf{T}_{o_b}^+}^{z}(o_b,o_b)} = \frac{\cW_{v,o,o_b}^{z}(o_b)}{\widetilde{U}^-_{z;v}(o_b) V_{z;o}^+(o_b)} = 0 - \frac{(V_{z;o})'(o_b)}{V_{z;o}(o_b)} = -R_{z}^+(o_b)$.

Similarly, we find $\widetilde{V}_{z;o}^+$ on $\mathbf{T}_o^+$ satisfying Neumann's condition at $t_b$, yielding
\begin{displaymath}
G_{\mathbf{T}_{t_b}^-}^{z}(x,y) = \begin{cases} \frac{U_{z;v}^-(x) \widetilde{V}_{z;o}^+(y)}{\cW^{z}_{v,o,t_b}(x)} &\text{if } y\in \mathcal{T}_x^+,\\ \frac{U_{z;v}^-(y) \widetilde{V}^+_{z;o}(x)}{\cW^{z}_{v,o,t_b}(x)} &\text{if } y\in \mathcal{T}_x^-,\end{cases}
\end{displaymath}
so we get similarly $R_{z}^-(t_b)=\frac{-1}{G_{\mathbf{T}_{t_b}^-}^{z}(t_b,t_b)}$.

To see the Herglotz property \cite{HP08}, note that if $\mathcal{H}_{\mathbf{T}_u^{\pm}}^{\max}$ is as in \S~\ref{sec:greenquan}, then
\[
\langle V_{z;u}^+, \mathcal{H}_{\mathbf{T}_u^+}^{\max} V_{z;u}^+\rangle_{L^2(\mathcal{T}_u^+)} - \langle \mathcal{H}_{\mathbf{T}_u^+}^{\max} V_{z;u}^+, V_{z;u}^+\rangle_{L^2(\mathcal{T}_u^+)} = 2\ii \Im z \cdot  \|V_{z;u}^+\|_{L^2(\mathcal{T}_u^+)}^2 \,.
\]
On the other hand, the left-hand side can also be computed by integration by parts on every edge. All the boundary terms except the one at $u$ cancel thanks to the self-adjoint conditions. We thus obtain 
\begin{multline*}
\langle V_{z;u}^+, \mathcal{H}_{\mathbf{T}_u^+}^{\max} V_{z;u}^+\rangle_{L^2(\mathcal{T}_u^+)} - \langle \mathcal{H}_{\mathbf{T}_u^+}^{\max} V_{z;u}^+, V_{z;u}^+\rangle_{L^2(\mathcal{T}_u^+)} = \overline{V_{z;u}^+(u)}(V_{z;u}^+)'(u) - \overline{(V_{z;u}^+)'(u)}V_{z;u}^+(u).
\end{multline*}
Since $V_{z;u}^+(u)=1$, this reduces to $2\ii \Im (V_{z;u}^+)'(u) = 2\ii \Im \frac{(V_{z;o}^+)'(u)}{V_{z;o}^+(u)}$. We thus showed $\Im R_{z}^+(u) = \Im z\, \|V_{z;u}^+\|_{\mathbf{T}_u^+}^2$, implying the result by taking $u=o_b$. The claim for $R_{z}^-$ is similar: the preceding proof shows that in the \emph{twisted} view, $\frac{(U_{z;o}^+)'(u)}{U_{z;o}^+(u)}$ is Herglotz, and the negative sign is there to pass to the \emph{coherent} view. These claims in turn show by \eqref{e:greener} that $z\mapsto G^{z}(v,v)$ is Herglotz\footnote{Though it is known that $z\mapsto  \langle \psi, G^{z} \psi\rangle$ is Herglotz for any $\psi\in L^2(\mathbf{T})$ by the spectral theorem, we followed this somehow roundabout argument to deduce the same holds for $z\mapsto G^{z}(v,v)$. See the appendix of \cite{AISW2} for a more general result.}

To show that $\frac{R_{z}^+(o_b)}{\sqrt{z}}$ is Herglotz, we use the same approach. First,
\[
\langle V_{z;u}^+, \mathcal{H}_{\mathbf{T}_u^+}^{\max} V_{z;u}^+\rangle_{L^2(\mathcal{T}_u^+)} + \langle \mathcal{H}_{\mathbf{T}_u^+}^{\max} V_{z;u}^+, V_{z;u}^+\rangle_{L^2(\mathcal{T}_u^+)} = 2 \Re z \cdot  \|V_{z;u}^+\|_{L^2(\mathcal{T}_u^+)}^2 \,.
\]
On the other hand, the left-hand side is $2\Re \langle V_{z;u}^+,\mathcal{H}_{\mathbf{T}_u^+}^{\max} V_{z;u}^+\rangle$. Integrating by parts yields
\[
2\Big(\|(V_{z;u}^+)'\|^2_{L^2(\mathcal{T}_u^+)} + \langle V_{z;u}^+,W V_{z;u}^+\rangle_{L^2(\mathcal{T}_u^+)} + \sum_{v\in V(\T_u^+)\setminus \{u\}} \alpha_v |V_{z;u}^+(v)|^2 + \Re \overline{V_{z;u}^+(u)}(V^+_{z;u})'(u) \Big) .
\]
As before, $V_{z;u}^+(u)=1$ and $(V_{z;u}^+)'(u) = R_{z}^+(u)$. Hence, $\Re R_{z}^+(u) = \Re z \|V_{z;u}^+\|_{L^2(\mathcal{T}_u^+)}^2 - \|(V_{z;u}^+)'\|_{L^2(\mathcal{T}_u^+)}^2 - \langle V_{z;u}^+,WV_{z;u}^+\rangle_{L^2(\mathcal{T}_u^+)}-\sum_{v\in V(\T_u^+)\setminus\{u\}}\alpha_v |V_{z;u}^+(v)|^2$.

Let $\sqrt{z}=r+\ii s$. It follows that
\begin{align*}
\Im &\Big(\frac{R_{z}^+(u)}{\sqrt{z}}\Big) = \Im R_{z} \Re \frac{1}{\sqrt{z}} + \Im\frac{1}{\sqrt{z}}\Re R_{z} \\
&= \frac{\Im z \|V_{z;u}^+\|^2 r - s\Re z \|V_{z;u}^+\|^2 + s(\|(V_{z;u}^+)'\|^2 + \langle V_{z;u}^+,WV_{z;u}^+\rangle + \sum \alpha_v|V_{z;u}^+(v)|^2)}{r^2+s^2} \\
&= \frac{\Im(z \overline{\sqrt{z}})\|V_{z;u}^+\|^2+s(\|(V_{z;u}^+)'\|^2+\langle V_{z;u}^+,WV_{z;u}^+\rangle +\sum \alpha_v|V_{z;u}^+(v)|^2)}{|z|}.
\end{align*}
Since $\Im(z \overline{\sqrt{z}}) = |z|\Im\sqrt{z}$, the RHS is clearly positive if $W\ge0$ and $\al_v\ge0$ for all $v$.

\medskip

We next prove the current relations. Since $V_{z;o}^+$ satisfies the $\delta$-conditions, we have $\sum_{b^+\in \cN_b^+} R_{z}^+(o_{b^+}) = R_{z}^+(t_b) + \alpha_{t_b}$, so $\sum_{b^+\in \cN_b^+} \Im R_{z}^+(o_{b^+}) = \Im R_{z}^+(t_b)$. Similarly, as $\sum_{w_-\in \cN_w^-} U_{w_-}'(L_{w_-}) + \alpha_w U_w(0) =   U_w'(0)$, we get $\sum_{b^-\in \cN_b^-} \Im R_{z}^-(t_{b^-}) = \Im R_{z}^-(o_b)$.

Suppose $H f = z f$ and let $J_{z}(x_b) = \Im [\overline{f(x_b)}f'(x_b)]$. Then $J_{z}'(x_b) = \Im [|f'(x_b)|^2 + \overline{f(x_b)}[W(x_b)f(x_b)-z f(x_b)] = -\Im z |f(x_b)|^2$. Hence, $J(x_b)$ decreases on $b$, in particular, $J(t_b) \le J(o_b)$. Thus, $|V_{z;v}^+(t_b)|^2 \Im R_{z}^+(t_b) = \Im [\overline{V_{z;v}^+(t_b)}(V_{z;v}^+)'(t_b)] \le |V_{z;v}^+(o_b)|^2\Im R_{z}^+(o_b)$. Thus, $\Im R_{z}^+(t_b) \le \frac{|V_{z;v}^+(o_b)|^2}{|V_{z;v}^+(t_b)|^2}\Im R_{z}^+(o_b) = \frac{\Im R_{z}^+(o_b)}{|\zeta^{z}(b)|^2}$. This proves \eqref{e:ASW} for $R_{z}^+$. Since $|U_{z;v}^-(o_b)|^2\Im R_{z}^-(o_b) = - \Im [\overline{U_{z;v}^-(o_b)}(U_{z;v}^-)'(o_b)]$, the claim for $R_{z}^-$ follows similarly.

Finally, if the terms are defined for $\Im z=0$, then equality follows from $J'_{z}(x_b)=0$.
\end{proof}

\begin{rem}\label{rem:hj}
  The method above also shows that $z\mapsto -\frac{S_{z}'(L)}{S_{z}(L)}$ is
  Herglotz.

  In fact,
  $\Im(\frac{S_{z}'(L)}{S_{z}(L)}) = \frac{1}{|S_{z}(L)|^2} \Im
  [\overline{S_{z}(L)}S_{z}'(L)] \le \frac{1}{|S_{z}(L)|^2}
  \Im[\overline{S_{z}(0)}S_{z}'(0)] =0$. If the potentials are
  symmetric, so that $S_{z}'(L)=C_{z}(L)$, we also get
  $\Im (\frac{C_{z}(L)}{S_{z}(L)}) \le 0$.

Recall that
$\Im R_{z}^+(t_b) = \sum_{b^+\in \cN_b^+}\Im R_{z}^+(o_{b^+}) \ge
0$. Using \eqref{e:r+-id},
$\Im(\frac{-1}{S_{z}(L_b)\zeta^{z}(b)}) = \Im R_{z}^+(t_b) - \Im
\frac{S_{z}'(L_b)}{S_{z}(L_b)} \ge 0$. Thus, we also get
$\Im(S_{z}(L_b)\zeta^{z}(b)) \ge 0$ for any $b$.
\end{rem}

Above we have shown that $\Im R_{z}^+(v) = \Im z\,
\|V_{z;v}^+\|_{L^2(\mathcal{T}_v^+)}^2$.  We would like to replace the $L^2$ norm
by a lower bound that does not depend on $z, v$ or any of the potentials.
This is proved in the following lemma:
\begin{lem}
  \label{lem:lower_bound}
  Let the potential $W=W_b$ be bounded,
  $\|W\|_\infty:= \sup_{b\in B} \|W_b\|_{L^\infty([0, L_b])} <\infty$,
  and let $K\subset \C^+$ be a compact set. Then there exists
  $C = C(K, \|W\|_\infty)>0$ such that for all $z\in K$ and all
  $v\in V(\mathbf{T})$, we have $\Im R_{z}^+(v) \ge C \Im z.$
\end{lem}
\begin{proof}
  We begin with the following general fact: \emph{let
    $M,\ell>0$. There exists $c>0$ such that for any potential $Q$ on
    $[0,\ell]$ with $\|Q\|_{\infty}\le M$, and any solution of
    $-f'' + Qf = 0$ on $[0,\ell]$ with $f(0)=1$, we have
    $\|f\|_{L^2(0,\ell)}^2\ge c$.} Indeed, suppose to the contrary there
  are $Q_n,f_n$ with $\|Q_n\|_{\infty}\le M$ and $f_n$ solutions with
  $f_n(0)=1$ such that $\|f_n\|_{L^2(0,\ell)}\to 0$. Then
  $\|f_n''\|_{L^2} \to 0$. This implies $\|f_n\|_{W^{2,2}} \to 0$
  (because in one dimension, $\|\cdot\|_{W^{2,2}}$ is equivalent to
  $\|u\|_{L^2} + \|u''\|_{L^2}$). But
  $W^{2,2}(0,\ell)\subset C^1[0,\ell]$, in particular this implies
  $\|f_n\|_{C^0[0,\ell]}\to 0$, so $f_n(0)\to 0$, a contradiction.

  Taking $Q=W-z$ and $[0,\ell]$ the edge in $\mathbf{T}_v^+$ outgoing from $v$
  we find for $f=V_{z;v}^+$,
  \begin{displaymath}
    \| V_{z;v}^+\|_{\mathbf{T}_v^+}^2 \ge \|f\|_{L^2(0,\ell)}^2 \ge c. \qedhere
  \end{displaymath}
\end{proof}

We showed in Remark~\ref{rem:hj} that $z\mapsto \frac{-S_z'(L)}{S_z(L)}$ is Herglotz. We also have:

\begin{lem} \label{lem:cot_herglotz}
  If the potentials $W$ are all non-negative,
  the function $z\mapsto-\frac{S_z'(L)}{\sqrt{z}S_z(L)}$ is Herglotz.
\end{lem}
\begin{proof}
  Consider
  \begin{align*}
    &\frac{\dd}{\dd x}\left( \frac{\overline{S_z'(x)} S_z(x)}
      {\sqrt{\bar{z}}} -
      \frac{S_z'(x)\overline{S_z(x)}}{\sqrt{z}}\right)
    \\&\qquad= \frac{\overline{S_z''(x)}S_z(x)}{\sqrt{\bar{z}}} +
    \frac{|S_z'(x)|^2}{\sqrt{\bar{z}}} -
      \frac{|S_z'(x)|^2}{\sqrt{z}} - \frac{S_z''(x)\overline{S_z(x)}}{\sqrt{z}}
    \\
    &\qquad = \frac{(W(x)-\bar{z})|S_z(x)|^2}
      {\sqrt{\bar{z}}} - 2\ii |S_z'(x)|^2 \Im \left\{ \frac1{\sqrt{z}}
      \right\}  -\frac{(W(x)- {z} ) |S_z(x)|^2}
      {\sqrt{z}}\\
    &\qquad= -2\ii W(x) |S_z(x)|^2 \Im \left\{ \frac1{\sqrt{z}} \right\} -
      2\ii |S_z'(x)|^2 \Im \left\{ \frac1{\sqrt{z}} \right\}  -(\sqrt{\bar{z}} - \sqrt{z}) |S_z(x)|^2. 
  \end{align*}
  Integrating this from $0$ to $L$ we get
  \begin{align*}
    2\ii \Im\sqrt{z} \int_0^L |S_z(x)|^2\,\dd x &-  2\ii \Im \left\{
    \frac1{\sqrt{z}} \right\} \int_0^L \left(W(x)|S_z(x)|^2 + |S_z'(x)|^2\right)\dd x \\
    &=  \frac{\overline{S_z'(L)} S_z(L)} {\sqrt{\bar{z}}} -
      \frac{S_z'(L)\overline{S_z(L)}}{\sqrt{z}}
      -  \frac{\overline{S_z'(0)} S_z(0)} {\sqrt{\bar{z}}} +
      \frac{S_z'(0)\overline{S_z(0)}}{\sqrt{z}} \\
 & = -2\ii \Im \left\{ \frac{S_z'(L)\overline{S_z(L)}}{\sqrt{z}} \right\} = 2\ii |S_z(L)|^2 \Im \left\{ \frac{-S_z'(L)}{\sqrt{z}S_z(L)} \right\}.
  \end{align*}
As $W\ge 0$, we deduce that $\frac{-S_z'(L)}{\sqrt{z}S_z(L)}\in \C^+$.
\end{proof}

\subsection{A criterion for AC spectrum} \label{app:A2}

We recall the following fact \cite[Theorem XIII.20]{RS4}.

\medskip

\emph{Suppose $H$ is a self-adjoint operator on a Hilbert space $\mathscr{H}$, with resolvent $G^{z}$. Suppose there exists $p>1$ such that for any $\varphi$ in a dense subset of $\mathscr{H}$, we have}
\begin{equation}\label{e:simoncrit}
\liminf_{\eta \downarrow 0} \int_a^b |\Im \langle \varphi, G^{\lambda+\ii\eta}\varphi\rangle|^p\,\dd\lambda <\infty \,.
\end{equation}
\emph{Then $H$ has purely absolutely continuous spectrum in $(a,b)$.}

\medskip

This criterion also appeared later in \cite{Klein} for $p=2$. In \cite{RS4}, $\liminf_{\eta\downarrow 0}$ is replaced by $\sup_{0<\eta<1}$, but one sees from the proof that the above statement holds.

\begin{thm}\label{thm:accrit}
Suppose a Schr\"odinger operator $H_{\mathbf{T}}$ on a quantum tree satisfies the following: there exists $p>1$ such that for any $b\in \mathbf{T}$,
\begin{equation}\label{e:simoncrit2}
\liminf_{\eta \downarrow 0} \int_I \frac{1}{\left(\Im R_{\lambda+\ii\eta}^+(o_b)\,|S_{\lambda+\ii\eta}(L_b)|^2\right)^p}\,\dd\lambda <\infty \,.
\end{equation}
Then $H_{\mathbf{T}}$ has purely absolutely continuous spectrum in $I$.

In particular, if conditions \eqref{e:nodir} and \emph{\textbf{(Green\itshape{-p})}} hold, then $\prob$-a.e.\ operator $H_{\mathbf{T}}$ has pure AC spectrum in $I_1$.
\end{thm}
\begin{proof}
Let $\varphi\in L^2(\mathcal{T})$ be continuous, compactly supported in a ball $\Lambda\subset \mathcal{T}$.
\begin{align}\label{e:greenac1}
\int_{E_0}^{E_1} |\Im \langle \varphi, G^{\lambda+\ii\eta}\varphi \rangle|^p\,\dd\lambda & = \int_{E_0}^{E_1} \left|\Im \int_{\Lambda\times\Lambda} \overline{\varphi(x)} G^{\lambda+\ii\eta}(x,y)\varphi(y)\,\dd x\,\dd y\right|^p \dd\lambda \nonumber \\
& \le \|\varphi\|_{\infty}^{2p} \cdot |\Lambda|^{2(p-1)}\int_{E_0}^{E_1}\int_{\Lambda\times\Lambda} |G^{\lambda+\ii\eta}(x,y)|^p\,\dd x\,\dd y\,\dd\lambda\,.
\end{align}
Now $\int_{\Lambda\times \Lambda} |G^{z}(x,y)|^p\,\dd x\,\dd y = \frac{1}{4} \sum_{b,b'\in B(\Lambda)} \int_0^{L_b} \int_0^{L_{b'}} |G^{z}(x_b,y_{b'})|^p\,\dd x_b \,\dd y_{b'}$.

Fix $b,b'\in B(\Lambda)$. Since for any $v\in V$, $y_{b'} \mapsto G^{z}(v,y_{b'})$ is an eigenfunction with eigenvalue $z$, we have\footnote{This holds if the potentials $W$ are symmetric, otherwise one should replace $S_{z}(L_b-x_b)$ by $S_{z}(L_b)C_{z}(x_b)-C_{z}(L_b)S_{z}(x_b)$. This does not affect the argument.}
\[
G^{z}(v,y_{b'}) = \frac{S_{z}(L_{b'}-y_{b'}) G^{z}(v,o_{b'}) + S_{z}(y_{b'}) G^{z}(v,t_{b'})}{S_{z}(L_{b'})} \,,
\]
similarly for the first argument, so we deduce that for $(b,x_b), (b', y_{b'})\in \mathcal{T}$, $b\neq b'$,
\begin{align}\label{e:greenac2}
G^{z}(x_b,y_{b'}) &= \frac{S_{z}(L_{b}-x_{b}) G^{z}(o_b,y_{b'}) + S_{z}(x_{b}) G^{z}(t_b,y_{b'})}{S_{z}(L_{b})} \\ \nonumber
 &= \frac{S_{z}(L_b-x_b)S_{z}(L_{b'}-y_{b'}) G^{z}(o_b,o_{b'}) + S_{z}(L_b-x_b)S_{z}(y_{b'})G^{z}(o_b,t_{b'}) }{S_{z}(L_b)S_{z}(L_{b'})} \\ \nonumber
&\quad+ \frac{S_{z}(x_b)S_{z}(L_{b'}-y_{b'})G^{z}(t_b,o_{b'}) + S_{z}(x_b)S_{z}(y_{b'})G^{z}(t_b,t_{b'})}{S_{z}(L_b)S_{z}(L_{b'})} \,.
\end{align}
Let $b_1,\ldots, b_k$ be a path with $b_1=b$ and $b_k=b'$. We observe that for any $b_{k+1}\in\cN_{b_k}^+$,
\begin{align*}
|\zeta^{z}(b_1)\cdots\zeta^{z}(b_k)| &\le \left[\sum_{(b_2';b_{k+1}')} |\zeta^{z}(b_1)\cdots\zeta^{z}(b_k')|^2\Im R_{z}^+(o_{b_{k+1}'}) \right]^{1/2}\frac{1}{|\Im R_{z}^+(o_{b_{k+1}})|^{1/2}} \\
&\le \left(\frac{|\Im R_{z}^+(o_{b_1})|}{|\Im R_{z}^+(o_{b_{k+1}})|}\right)^{1/2} ,
\end{align*}
where we used \eqref{e:ASW} repeatedly in the last step. On the other hand, by \eqref{e:greener},
\[
|G^{z}(o_{b_1},o_{b_1})| = \frac{1}{|R_{z}^+(o_{b_1}) + R_{z}^-(o_{b_1})|} \le \frac{1}{\Im R_{z}^+(o_{b_1})}\,,
\]
where we used the fact that $ R_{z}^\pm(v) $ is Herglotz. Hence,
\[
|G^{z}(o_{b_1};t_{b_k})| \le \frac{1}{|\Im R_{z}^+(o_{b_1}) \Im R_{z}^+(o_{b_{k+1}})|^{1/2}} \,.
\]
The other $G^z(v_0;v_r)$ appearing in \eqref{e:greenac2} are bounded similarly.

For $b=b'$, the first equality in \eqref{e:greenac2} should be modified as we don't have an eigenfunction at the point $x_b=y_b$. In fact, assuming without loss that $x_b\le y_b$, we have
\begin{multline}\label{e:expansame}
G^z(x_b,y_b) = \frac{1}{S_z(L_b)}\big([S_z(L_b)C_z(x_b)-C_z(L_b)S_z(x_b)]G^z(o_b,y_b) \\
+ S_z(x_b)[G^z(t_b,y_b) + S_z(L_b)C_z(y_b)-C_z(L_b)S_z(y_b)]\big).
\end{multline}
This can be checked by explicit calculation from \eqref{e:othergreen} and \eqref{e:greener}: these tell us that
\[
G^z(o_b,y_b) = G^z(o_b,o_b) \phi^-(o_b)\phi^+(y_b) = G^z(o_b,o_b)\phi^+(y_b),
\]
\begin{align*}
G^z(t_b,y_b) &= G^z(o_b,o_b) \phi^-(y_b)\phi^+(L_b) \\
&= G^z(o_b,o_b)[\phi^+(y_b)\phi^-(L_b) - (R^++R^-)(C_z(L_b)S_z(y_b)-S_z(L_b)C_z(y_b))]\\
& = G^z(o_b,o_b)[C_z(L_b)-R^-S_z(L_b)]\phi^+(y_b) + C_z(L_b)S_z(y_b)-S_z(L_b)C_z(y_b).
\end{align*}
Here we passed from $\phi^-(y_b)\phi^+(L_b)$ to $\phi^+(y_b)\phi^-(L_b)$ as in Lemma~\ref{lem:ASWrep}. Inserting this into the RHS of \eqref{e:expansame} gives $G^z(o_b,o_b)\phi^-(x_b)\phi^+(y_b)=G^z(x_b,y_b)$ as asserted.

In any case, for any $b\in B$, $\int_0^{L_b} |S_{z}(x_b)|^p\,\dd x_b \le c$, uniformly in $\lambda+\ii\eta\in (E_0,E_1)+\ii (0,1)$ (in fact, $\lambda\mapsto S_{\lambda}(x)$ is analytic). Recalling \eqref{e:greenac1}, \eqref{e:greenac2}, we have proved that
\begin{multline*}
\int_{E_0}^{E_1} |\Im \langle \varphi,G^{\lambda+\ii\eta}\varphi\rangle|^p\,\dd\lambda \\
\le \tilde{c}\, \|\varphi\|_{\infty}^{2p} \cdot |\Lambda|^{2(p-1)}\sum_{b,b'\in B(\Lambda)} \int_{E_0}^{E_1} \frac{\dd\lambda}{|\Im R_{z}^+(o_b)\Im R_{z}^+(o_{b''})|^{p/2}|S_{z}(L_b)S_{z}(L_{b'})|^p} \,,
\end{multline*}
where $b''$ is any edge with $o_{b''}=t_{b'}$.
Applying Cauchy-Schwarz and \eqref{e:simoncrit2}, we see that \eqref{e:simoncrit} is satisfied for any continuous $\varphi$ of compact support. Hence, $H_{\mathbf{T}}$ has pure AC spectrum in $(E_0,E_1)$.

In particular, if \textbf{(Green\itshape{-p})} holds for a random tree, then by Fatou's lemma and Fubini, we have
\begin{multline*}
\expect\left(\liminf_{\eta\downarrow 0} \int_{I_1} \sum_{t_b\sim o_b} \frac{1}{\left(\Im R_{\lambda+\ii\eta}^+(o_b)\,|S_{\lambda+\ii\eta}(L_b)|^2\right)^p}\,\dd\lambda\right) \\
\le \liminf_{\eta\downarrow 0} \int_{I_1} \expect\left(\sum_{t_b\sim o_b} \frac{1}{\left(\Im R_{\lambda+\ii\eta}^+(o_b)\,|S_{\lambda+\ii\eta}(L_b)|^2\right)^p} \right)\dd\lambda < \infty\,,
\end{multline*}
where we applied \textbf{(Green\itshape{-p})} in the last step. It follows that the $\liminf$ on the left-hand side is finite for $\prob$-a.e.\ operator $H_{\mathbf{T}}$. Hence, $\prob$-a.e.\ $H_{\mathbf{T}}$ has pure AC spectrum in $I_1$.
\end{proof}

\section{The uniform contraction estimate}\label{app:uni}

The proof of Proposition~\ref{p:ka} goes by analyzing different cases. Following \cite{KLW} we introduce ``visible sets of vertices'' $\Vis_{\gm}(g,\varepsilon)$ and $\Vis_{\Im}^{i}(g,\varepsilon)$. The idea is that we seek a strict contraction for $\kappa_\ast^{(\pi)}$ by estimating some terms in the weighted sum~\eqref{eq:cont_coeff}.  By controlling the visibility we ensure that estimates on $Q_{x,y}^{(\pi)}(g)\cos\alpha_{x,y}^{(\pi)}(g)$ do not become ``invisible'', for example by the weightings $q_y(g)$ or $\gamma_x(g)$ becoming too small (which would jeopardize an implied control over $c_x^{(\pi)}$ and $\kappa_\ast^{(p)}$, respectively (recall \eqref{e:cx1}-\eqref{e:cx2}).

One sees that Lemmas 4 and 5 in \cite{KLW} hold verbatim for quantum graphs.

We then study three cases to prove Proposition~\ref{p:ka}. In the first case, certain $\gamma_x^{(\pi)}(g)$ is very small. This implies the terms in the sum defining $\kappa_{\ast}^{(p)}(g)$ have different magnitudes, and the result follows from (an improved) Jensen inequality. In the second case it is assumed that all $\gamma_x^{(\pi)}(x)$ have essentially the same size, but certain $\Im g_y^{(\pi)}$, $y\in S_{\ast'}$ is very small. In this case one proves that certain $c_x^{(\pi)}(g)$ must be small, so the result follows from \cite[Lemma 5]{KLW}. Finally in the third case, it is proved that if $g\notin B_r(H)$, then there are always $\pi,x,y$ such that $Q_{x,y}^{(\pi)}(g)\cos\alpha_{x,y}^{(\pi)}(g)$ is uniformly smaller than one. Assuming we are neither in the first nor the second case, one deduces that some $c_x^{(\pi)}(g)$ is small, concluding the proof again using \cite[Lemma 5]{KLW}. See \cite[Section 4.7]{KLW} for more details as to how these cases are put together to prove Proposition~\ref{p:ka}.

For quantum graphs, the first and second cases are the same as in \cite{KLW}. Up to this point we are only using calculus and the definition of $\kappa_\ast^{(p)}(g)$ and $c_x(g)$. So more precisely, Propositions 4 and 5 from \cite{KLW} hold without change to quantum graphs. 

For the third case however, we need some effort to carry out the adaptation.

We start with the following lemma.

\begin{lem}\label{lem:zetasin}
Given $z\in \h$, we have $-\cot \left( \sqrt{z}L\right) \in \h$. In particular, assuming $W\equiv 0$, we have $\zeta^z(b)\sin \left(\sqrt{z}L_b\right) \in \h$.
\end{lem}
\begin{proof}
  In Lemma \ref{lem:cot_herglotz} the more general fact that
  $-\frac{S_z'(L)}{\sqrt{z} S_z(L)}$ is Herglotz is proved.
  Specialising to $W\equiv0$ we get that $-\cot\sqrt{z}L$ is Herglotz.

    For the second claim, from \eqref{e:r+-id} it follows that
    \begin{displaymath}
      \frac{-1}{\sqrt{z}S_z(L_b)\zeta^z(b)} = \frac{R_z^+(t_b)}{\sqrt{z}} -
      \frac{S_z'(L_b)}{\sqrt{z}S_z(L_b)}\,,
\end{displaymath}
so by the previous part and the fact that $\frac{R_z^+(t_b)}{\sqrt{z}}\in\h$,
we get that 
$\frac{-1}{\sqrt{z}S_z(L_b)\zeta^z(b)}\in\h$.
Again, specialising to $W\equiv0$ completes the proof.
\end{proof}

In analogy to \eqref{e:zetawt}, saying that $\zeta^z(b) = C_z(L_b) + S_z(L_b)R_z^+(o_b)$, given $g\in \h^{S_{o,o'}}$, we define $\zeta_g^z(\ast')=\zeta_g^z(\ast',\alpha,L)$ by
\begin{equation}\label{e:zetoprime}
\zeta^z_g(\ast') = \cos\sqrt{z}L+g_{\ast'}\sin\sqrt{z}L  \,.
\end{equation}
Then $\zeta_H^z(\ast',\alpha_{\ast'},L_{\ast'}) = \zeta_0^z(\ast,\ast')$ as expected. Moreover, from \eqref{e:goprime}, one easily checks that $\zeta^z_g(\ast') = \frac{1}{-\phi_{\ast'}(g)\sin\sqrt{z}L+\cos\sqrt{z}L}$, so that $\frac{-1}{\sin\left(\sqrt{z}L\right) \zeta_g^z(\ast')}= \phi_{\ast'}(g) - \cot\sqrt{z}L$, which implies that $\sin(\sqrt{z}L)\zeta_g^z(\ast')\in\h$ by Lemma~\ref{lem:zetasin}.

\smallskip

We consider the argument of vectors related to the operator's Green function. The convention is that $\arg z \in (-\pi,\pi]$. 
Moreover, we denote by $d_{\Sp^{1}}(\cdot,\cdot)$ the translation invariant metric in $\Sp^{1}$ which is normalized by $d_{\Sp^{1}}(0,\pi)=\pi$.

Denote $Z_v^z = \zeta^z_0(v_-,v)\sin\sqrt{z}L_{v}^0$, the unperturbed $\zeta^z\sin\sqrt{z}L$ of the edge $(v_-,v)$. We define a quantity, related to the `minimal angle' of this with the real axis, by
\begin{equation}\label{e:theta0}
\theta_0:=\frac{1}{10}\min\left\{d_{\Sp^{1}}(\arg Z_k^z,\be)\mid \be\in\{0,\pi\}, z\in I+\ii[0,1], k\in\mathfrak{A}\right\},
\end{equation}
where we denoted $Z_k^z=Z_v^z$ if $v$ has label $k$.
In view of Lemma~\ref{lem:zetasin}, since $I\subset\Sigma$ is chosen compact, the minimum exists and we have $\theta_0>0$. We also let
\[
\varsigma_0 = \min\left\{\Im H_k^z : z\in I+\ii[0,1], k\in\mathfrak{A}\right\} ,
\]
\begin{equation}\label{e:maxiga}
\varsigma_1 = \max\left\{|\Gamma_k^z| : z\in I+\ii[0,1],k\in \mathfrak{A} \right\},
\end{equation}
where we denoted $H_k^z = H_v^z$ for $v$ of label $k$ and similarly for $\Gamma$. Again $\varsigma_0>0$.

We shall also need the function
\begin{equation}\label{e:veps1}
\varepsilon_1 = \varepsilon_1(r) = \min_{x\in S_{\ast,\ast'}}\inf_{z\in I+\ii[0,1]}\inf_{g_x\in\h\setminus B_r(H_x^z)} |g_x - H_x^z| \,.
\end{equation}

We may now state the main result of this appendix (corresponding to case 3 in the above discussion).

\begin{prp}\label{p:al>0}
Let $I\subset \Sigma$ be compact and satisfy \eqref{e:diravoid}. There is $c=c(\theta_0)<1$, $\eps_\ast(\theta_0,\varsigma_0,I,\varepsilon_D)>0$ and $R:[0,\eps_\ast)\to[0,\infty)$ with $\lim_{\eps\to0}R(\eps)=0$ such that for all $\eps\in[0,\eps_\ast)$, if
\begin{itemize}
\item  $g\in \h^{S_{\ast,\ast'}}\setminus B_{R(\eps)}(H)$ and $g_{\ast'}=g_{\ast'}(z,\alpha,L)$ is defined by \eqref{e:goprime}--\eqref{e:goprime2},
\item $|\alpha-\alpha_{\ast'}^0 | \le \eps$ and $|L-L_{\ast'}^0| \le \eps$,
\item $\Im z\in [0,\eta_\ast]$ for some $\eta_\ast(\theta_0,\varepsilon_D)$,
\end{itemize}
then there is $\pi\in \Pi$ with
\[
Q_{x,y}^{(\pi)}(g)\le c\qquad\text{or}\qquad \cos\alpha_{x,y}^{(\pi)}(g)\le c,
\]
either for  some  $x,y\in S_{\ast'}$ or for $x=\ast'$ and all $y\in S_{\ast} \setminus \{\ast'\}$.
\end{prp}
\begin{proof}
First recall that for $\xi_i\in \C$
\begin{equation}\label{e:ds1iso}
d_{\Sp^1}\left(\arg(\xi_1),\arg(\xi_2)\right) = d_{\Sp^1}\left(\arg(\xi_1\bar{\xi}_2),0\right) = d_{\Sp^1}\left(\arg\big(\frac{\xi_1}{\xi_2}\big),0\right) .
\end{equation}
Also, for $\alpha,\be,\gm\in \Sp^{1}$, we have by the triangle inequality and translation invariance,
\begin{equation}
  \label{e:dtrii}
  d_{\Sp^1}(\al+\be,\gamma) \ge d_{\Sp^1}(\al,0) - d_{\Sp^1}(\be,\gamma).
\end{equation}
Moreover, as $\al_{x,y}(g)=\arg((g_{x}-H_{x}^z)\overline{(g_{y}-H_{y}^z)})$, we have by \eqref{e:ds1iso},
\begin{equation}\label{e:dsp}
d_{\Sp^{1}}(\al_{x,y}(g),0) =d_{\Sp^{1}}(\arg(g_{x}-H_{x}^z),\arg(g_{y}-H_{y}^z)).
\end{equation}

By Lipschitz continuity, we may find $c_I$ such that for all $z\in I+\ii[0,1]$ and $L\in \bigcup_{i\in\mathfrak{A}} [L_i^0-1,L_i^0+1]$, we have
\begin{equation}\label{e:ci}
\left|\sin\sqrt{z}(L-L_i^0)\right| \le c_I\cdot |L-L_i^0| \,.
\end{equation}
We also take $c_I'$ such that
\begin{equation}\label{e:ciprime}
\left|\frac{\alpha - \alpha_i^0}{\sqrt{z}}\right| \le c_I' \cdot|\alpha - \alpha_i^0| \,.
\end{equation}
We then choose
\begin{equation}\label{e:epstoile}
\eps_{\ast} := \min\left(\frac{\theta_0}{c_I'(1+\theta_0)}\varsigma_0, \frac{\theta_0\varepsilon_D}{4c_I\varsigma_1(1+\theta_0)},\frac{\theta_0\varepsilon_D}{2c_I(1+\varsigma_1^2)(1+\theta_0)}\varsigma_0\right),
\end{equation}
\begin{equation}\label{e:mi,d}
M_{I,D} = \max\left(c_I', \frac{2c_I(1+\varsigma_1^2)}{\varepsilon_D}\right)
\end{equation}
and define $R:[0,\eps_\ast)\to[0,\infty)$ by
\[
R(\eps)=\varepsilon_{1}^{-1}(\theta_{\eps}),\quad\text{with } \theta_{\eps}=\frac{(1+\theta_0)}{\theta_0}M_{I,D}\eps\,,
\]
where $\varepsilon_1(r)$ is defined in \eqref{e:veps1}. By \cite[Lemma 10]{KLW} or \cite[Lemma 2.18]{Kel}, $R(\eps) = \frac{\theta_{\eps}^2}{\varsigma_0(\varsigma_0-\theta_{\eps})}$ is well-defined, since $\varsigma_0-\theta_{\eps}>0$ by the choice of $\eps_\ast$.

Take $g\in \h^{S_{\ast,\ast'}}\setminus B_{R(\eps)}$ and $\eps\in[0,\eps_{\ast})$. If there is $x\in S_{\ast,\ast'}$ such that $g_x=H_x^z$, then $Q_{x,y}(g)=0$ by definition for all $y\in S_{\ast,\ast'}$ and we are done. 
So  assume that $g_{x}\neq H_{x}^z$ for all $x\in S_{\ast,\ast'}$. We also assume
\begin{equation}\label{e:asuboun}
d_{\Sp^{1}}({\al_{x,y}^{(\pi)}(g),0})\le \theta_0 \qquad \text{for all } \pi\in\Pi \text{ and } x,y\in S_{\ast'}
\end{equation}
since otherwise we are done. Our aim is to show $d_{\Sp^{1}}({\al_{\ast',y}^{(\pi)}(g),0})>\theta_0$ for some $\pi\in\Pi$ and all $y\in S_{\ast}\setminus\{\ast'\}$.

Recall that $H_{\pi(x)}^z=H_x^z$ for $\pi\in \Pi$ and $x\in S_{\ast,\ast'}$. We set
\[
\tau_{\ast'}^{(\pi)}:=\sum_{y\in S_{\ast'}}\left({g_{y}^{(\pi)}-H_{y}^z}\right) .
\]
Recalling \eqref{e:dsp}, our assumption \eqref{e:asuboun} implies that
\begin{equation}\label{e:sumii}
\big|\tau_{\ast'}^{(\pi)}\big|\ge \big|g_{x}^{(\pi)}-H_x^z\big|
\end{equation}
for all $\pi\in\Pi$ and $x\in S_{\ast'}$, and
\begin{equation}\label{e:sumiii}
d_{\Sp^{1}}\big(\arg(\tau_{\ast'}^{(\pi)}) ,\arg(g_{x}^{(\pi)}-H_{x}^z)\big) \leq 2\theta_0
\end{equation}
for all $\pi\in\Pi$ and $x\in S_{\ast,\ast'}$. See \cite[Lemma 9]{KLW} or \cite[Lemma 4.19]{Kel} for a proof.

We now proceed with two claims.

\smallskip

\emph{Claim 1: There exists  $\pi\in\Pi$ such that $|\tau_{\ast'}^{(\pi)}| \geq \frac{(1+\theta_0)}{\theta_0}M_{I,D}\eps.$}\\
Proof of Claim~1.
We assumed that $g\notin B_{R(\eps)}$, so there is $\hat{x}\in S_{\ast,\ast'}$ such that $\gm_{\hat{x}}(g)> R(\eps)$. By the choice of $\ast'$ w.r.t.\ $\textbf{(C2)}$ there is $\pi$ such that $\pi(\hat{x})\in S_{\ast'}$. 
By \eqref{e:sumii} and the definitions \eqref{e:veps1} of  $\varepsilon_{1}$ and $R$ above, we obtain
\begin{displaymath}
|\tau_{\ast'}^{(\pi^{-1})}| \ge |g_{\pi(\hat{x})}^{(\pi^{-1})}-H_{\pi(\hat{x})}^z| = |g_{\hat{x}}-H_{\hat{x}}^z|\ge \varepsilon_{1}(R(\eps))=\theta_{\eps}=\frac{(1+\theta_0)}{\theta_0}M_{I,D}\eps\,.
\end{displaymath}

\smallskip

\emph{Claim 2: Denote $Z^z(v) = \zeta^z(v_-,v)\sin \sqrt{z}L_{v}$ and $w=\frac{\alpha-\alpha_{\ast'}^0}{\sqrt{z}}$. Then we have}
\begin{multline*}
g_{\ast'}^{(\pi)} - H_{\ast'}^z = Z_{g^{(\pi)}}^z(\ast')Z_0^z(\ast')\bigg[\left(\tau_{\ast'}^{(\pi)}-w\right)\Big\{1+\cot\sqrt{z}L\cot\sqrt{z}L_{\ast'}^0 \\
+ \Gamma_{\ast'}^z\cdot \frac{\sin\sqrt{z}(L-L_{\ast'}^0)}{\sin\sqrt{z}L\sin\sqrt{z}L_{\ast'}^0}\Big\} + ((\Gamma_{\ast'}^z)^2+1) \frac{\sin\sqrt{z}(L-L_{\ast'}^0)}{\sin\sqrt{z}L\sin\sqrt{z}L_{\ast'}^0}\bigg].
\end{multline*}
Proof of Claim~2. We have by (\ref{e:goprime}),
\begin{align*}
g_{\ast'}^{(\pi)}-H_{\ast'}^z & = \frac{\phi_{\ast'}(g^{(\pi)})\cos\sqrt{z}L+\sin\sqrt{z}L}{-\phi_{\ast'}(g^{(\pi)})\sin\sqrt{z}L+\cos\sqrt{z}L} - \frac{\Gamma_{\ast'}^z\cos\sqrt{z}L_{\ast'}^0+\sin\sqrt{z}L_{\ast'}^0}{-\Gamma_{\ast'}^z\sin\sqrt{z}L_{\ast'}^0+\cos\sqrt{z}L_{\ast'}^0} \\
& = \frac{1}{(-\phi_{\ast'}(g^{(\pi)})\sin\sqrt{z}L+\cos\sqrt{z}L)(-\Gamma_{\ast'}^z\sin\sqrt{z}L_{\ast'}^0+\cos\sqrt{z}L_{\ast'}^0)} \\
& \quad \times \Big[\phi_{\ast'}(g^{(\pi)})\{\cos\sqrt{z}L\cos\sqrt{z}L_{\ast'}^0+\sin\sqrt{z}L\sin\sqrt{z}L_{\ast'}^0\\
& \qquad\qquad\qquad + \Gamma_{\ast'}^z(\sin\sqrt{z}L\cos\sqrt{z}L_{\ast'}^0-\cos \sqrt{z}L\sin\sqrt{z}L_{\ast'}^0)\} \\
& \qquad - \Gamma_{\ast'}^z (\cos\sqrt{z}L\cos\sqrt{z}L_{\ast'}^0+\sin\sqrt{z}L\sin\sqrt{z}L_{\ast'}^0) \\
& \qquad +\sin\sqrt{z}L\cos\sqrt{z}L_{\ast'}^0-\cos\sqrt{z}L\sin\sqrt{z}L_{\ast'}^0\Big] \\
& = \frac{(\phi_{\ast'}(g^{(\pi)})-\Gamma_{\ast'}^z)\cos\sqrt{z}(L-L_{\ast'}^0) + (\phi_{\ast'}(g^{(\pi)})\Gamma_{\ast'}^z+1)\sin\sqrt{z}(L-L_{\ast'}^0)}{(-\phi_{\ast'}(g^{(\pi)})\sin\sqrt{z}L+\cos\sqrt{z}L)(-\Gamma_{\ast'}^z\sin\sqrt{z}L_{\ast'}^0+\cos\sqrt{z}L_{\ast'}^0)} \,.
\end{align*}
But by reverse M\"obius transformation, $\phi_{\ast'}(g^{(\pi)}) = \frac{g_{\ast'}^{(\pi)}\cos\sqrt{z}L-\sin\sqrt{z}L}{g_{\ast'}^{(\pi)}\sin\sqrt{z}L+\cos\sqrt{z}L}$. Hence,
\begin{align*}
-\phi_{\ast'}(g^{(\pi)})\sin\sqrt{z}L+\cos\sqrt{z}L &= \frac{\sin^2\sqrt{z}L-g_{\ast'}^{(\pi)}\sin\sqrt{z}L\cos\sqrt{z}L}{g_{\ast'}^{(\pi)}\sin\sqrt{z}L+\cos\sqrt{z}L} + \cos\sqrt{z}L \\
&= \frac{1}{g_{\ast'}^{(\pi)}\sin\sqrt{z}L+\cos\sqrt{z}L} = \frac{1}{\zeta^z_{g^{(\pi)}}(\ast')}
\end{align*}
by definition \eqref{e:zetoprime}. Thus,
\begin{multline*}
g_{\ast'}^{(\pi)}-H_{\ast'}^z  = \zeta_{g^{(\pi)}}^z(\ast')\zeta_0^z(\ast,\ast')\big[(\phi_{\ast'}(g^{(\pi)})-\Gamma_{\ast'}^z)\{\cos\sqrt{z}(L-L_{\ast'}^0)+\Gamma_{\ast'}^z\sin\sqrt{z}(L-L_{\ast'}^0)\}\\
+((\Gamma_{\ast'}^z)^2+1)\sin\sqrt{z}(L-L_{\ast'}^0)\big] .
\end{multline*}

Next, $\phi_{\ast'}(g^{(\pi)}) = \sum_{x\in S_{\ast'}} g_x^{(\pi)} - \frac{\alpha}{\sqrt{z}}$, so $(\phi_{\ast'}(g^{(\pi)})-\Gamma_{\ast'}^z) = \tau_{\ast'}^{(\pi)} -w$. Claim 2 follows.

\smallskip

\emph{Conclusion:} By Claim 2, using $\arg(\xi\zeta)=\arg(\xi)+\arg(\zeta)$ and $\arg(\xi+\zeta)=\arg(\xi)+\arg(1+\zeta/\xi)$, we get
\begin{multline*}
\arg(g_{\ast'}^{(\pi)} - H_{\ast'}^z) = \arg\big(Z_{g^{(\pi)}}^z(\ast')Z_0^z(\ast')\big) + \arg\bigg[\Big(1-\frac{w}{\tau_{\ast'}^{(\pi)}}\Big)\Big\{1+\cot\sqrt{z}L\cot\sqrt{z}L_{\ast'}^0\\
+\Gamma_{\ast'}^z\cdot  \frac{\sin\sqrt{z}(L-L_{\ast'}^0)}{\sin\sqrt{z}L\sin\sqrt{z}L_{\ast'}^0}\Big\} + \frac{(\Gamma_{\ast'}^z)^2+1}{\tau_{\ast'}^{(\pi)}}\cdot  \frac{\sin\sqrt{z}(L-L_{\ast'}^0)}{\sin\sqrt{z}L\sin\sqrt{z}L_{\ast'}^0}\bigg] +\arg(\tau_{\ast'}^{(\pi)})
\end{multline*}
We thus get for the permutation $\pi\in\Pi$ taken from Claim~1 and all $y\in S_{\ast} \setminus \{\ast'\}$
\begin{multline*}
d_{\Sp^{1}}\big({{\al_{\ast',y}^{(\pi)}}(g),0}\big) = d_{\Sp^1}\left(\arg(g_{\ast'}^{(\pi)}-H_{\ast'}^z),\arg(g_y^{(\pi)}-H_y^z)\right)\\
\geq d_{\Sp^{1}}\left(\arg(Z_{g^{(\pi)}}^z(\ast')Z_0^z(\ast')),0\right) - d_{\Sp^1}\bigg(\arg\bigg[1-\frac{w}{\tau_{\ast'}^{(\pi)}}+\\
+\left(1-\frac{w}{\tau_{\ast'}^{(\pi)}}\right)\Gamma_{\ast'}^z\cdot  \frac{\sin\sqrt{z}(L-L_{\ast'}^0)}{\sin\sqrt{z}L\sin\sqrt{z}L_{\ast'}^0} + \frac{(\Gamma_{\ast'}^z)^2+1}{\tau_{\ast'}^{(\pi)}}\cdot  \frac{\sin\sqrt{z}(L-L_{\ast'}^0)}{\sin\sqrt{z}L\sin\sqrt{z}L_{\ast'}^0}\\
+\left(1-\frac{w}{\tau_{\ast'}^{(\pi)}}\right)\cot\sqrt{z}L\cot\sqrt{z}L_{\ast'}^0\bigg],0\bigg)
-d_{\Sp^{1}}\left(\arg(\tau_{\ast'}^{(\pi)}), \arg{(g_{y}^{(\pi)}-H_{y}^z)}\right)
\end{multline*}
where we used \eqref{e:dtrii} twice.  To bound the first distance, note that we have $Z_0^z(\ast')\in \h$ and $d_{\Sp^{1}}(\arg Z_0^z(\ast'),\be)\ge 10\theta_0$ for
$\be\in\{0,\pi\}$, by the definition \eqref{e:theta0} of
$\theta_0$. Moreover, since $Z_{g^{(\pi)}}^z(\ast')\in\h$ by the
argument after \eqref{e:zetoprime}, we have
$d_{\Sp^{1}}(\arg(Z_{g^{(\pi)}}^z(\ast')),\be)>0$ for
$\be\in\{0,\pi\}$. Hence,
$d_{\Sp^{1}}(\arg(Z_0^z(\ast')Z_{g^{(\pi)}}^z(\ast')),0)>10\theta_0$.
The last distance can be estimated by \eqref{e:sumiii}.

For the second distance we begin by noticing $|\frac{w}{\tau_{\ast'}^{(\pi)}}| \le \frac{\theta_0}{1+\theta_0}$. In fact, $|w| \le c_I'\eps$ by \eqref{e:ciprime}, so this follows from Claim~1 and \eqref{e:mi,d}. Also, using \eqref{e:ci},
\begin{displaymath}
\left|\frac{(\Gamma_{\ast'}^z)^2+1}{\tau_{\ast'}^{(\pi)}}\cdot  \frac{\sin\left(\sqrt{z}(L-L_{\ast'}^0)\right)}{\sin\left(\sqrt{z}L\right)\sin\left(\sqrt{z}L_{\ast'}^0\right)}\right| \le \frac{(\varsigma_1^2+1)\theta_0}{M_{I,D}\eps(1+\theta_0)}\frac{c_I\eps}{\varepsilon_D}
\end{displaymath}
and
\begin{displaymath}
\left|\Big(1-\frac{w}{\tau_{\ast'}^{(\pi)}}\Big)\Gamma_{\ast'}^z\cdot  \frac{\sin\left(\sqrt{z}(L-L_{\ast'}^0)\right)}{\sin\left(\sqrt{z}L\right)\sin\left(\sqrt{z}L_{\ast'}^0\right)}\right|\le 2\varsigma_1\frac{c_I\eps}{\varepsilon_D}.
\end{displaymath}
Now as noted in \cite[Lemma 8]{KLW}, for any $\zeta,\xi\in \C$, $|\zeta|<1$, we have
\begin{equation}\label{e:diszx}
 d_{\Sp^{1}}(\arg\ap{1+ \xi+\zeta},0)\le \begin{cases} \frac{|\zeta|}{1-|\zeta|} & \text{if } \xi=0,\\ d_{\Sp^{1}}(\arg \xi,0) +\frac{|\zeta|}{1-|\zeta|} & \text{if } \xi\neq 0. \end{cases}
\end{equation}
We use this with $\xi = \big(1-\frac{w}{\tau_{\ast'}^{(\pi)}}\big)\cot\left(\sqrt{z}L\right)\cot\left(\sqrt{z}L_{\ast'}^0\right)$ 
and
\begin{displaymath}
\zeta = \frac{-w}{\tau_{\ast'}^{(\pi)}}+\Big(1-\frac{w}{\tau_{\ast'}^{(\pi)}}\Big)\Gamma_{\ast'}^z\cdot \frac{\sin\left(\sqrt{z}(L-L_{\ast'}^0)\right)}{\sin\left(\sqrt{z}L\right)\sin\left(\sqrt{z}L_{\ast'}^0\right)} + \frac{(\Gamma_{\ast'}^z)^2+1}{\tau_{\ast'}}\cdot \frac{\sin\left(\sqrt{z}(L-L_{\ast'}^0)\right)}{\sin\left(\sqrt{z}L\right)\sin\left(\sqrt{z}L_{\ast'}^0\right)}.
\end{displaymath}
From the estimates above,
\[
|\zeta| \le \frac{\theta_0}{1+\theta_0} + \frac{c_I}{\varepsilon_D} \frac{(\varsigma_1^2+1)\theta_0}{M_{I,D}(1+\theta_0)} + \frac{2c_I\eps\varsigma_1}{\varepsilon_D} \le \frac{\theta_0}{1+\theta_0}  + \frac{\theta_0}{2(1+\theta_0)}+\frac{\theta_0}{2(1+\theta_0)} \le \frac{3\theta_0}{1+3\theta_0} \,,
\]
using \eqref{e:epstoile}, \eqref{e:mi,d}. Here, we used $\theta_0 \le \frac{1}{3}$ (which holds since $\theta_0\le \frac{\pi}{10}$).

Applying \eqref{e:diszx} thus yields
\begin{multline*}
d_{\Sp^{1}}\big({{\al_{\ast',y}^{(\pi)}}(h),0}\big) > 10\theta_0 - 3\theta_0 -2\theta_0 - d_{\Sp^1}\bigg(\arg\bigg[\left(1-\frac{w}{\tau_{\ast'}^{(\pi)}}\right)\cot\sqrt{z}L\cot\sqrt{z}L_{\ast'}^0\bigg],0\bigg)\\
\ge 5\theta_0 - d_{\Sp^1}\bigg(\arg\left(1-\frac{w}{\tau_{\ast'}^{(\pi)}}\right),0\bigg) - d_{\Sp^1}\bigg(\arg\bigg[\cot\sqrt{z}L\cot\sqrt{z}L_{\ast'}^0\bigg],0\bigg)
\end{multline*}

The first distance is controlled again by \eqref{e:diszx} with $\xi=0$. For the second distance, we observe that $\Re z$ is not a Dirichlet value for $L_{\ast'}^0,L$ by assumption \eqref{e:diravoid}, so the argument tends to $0$ as $\Im z\downarrow 0$. So for $\Im z$ small enough, the last distance is $\le \theta_0$. This yields
\[
d_{\Sp^{1}}\big({{\al_{\ast',y}^{(\pi)}}(g),0}\big)>5\theta_0-\theta_0-\theta_0 > \theta_0 \,.
\]
The assertion follows by letting $c:=\cos\theta_0$.
\end{proof}

\bigskip

\noindent {\bf{Acknowledgments:}} M.I. was funded by the Labex IRMIA during part of the writing of this article.

M.S.\ was supported by a public
grant as part of the \textit{Investissement d'avenir} project,
reference ANR-11-LABX-0056-LMH, LabEx LMH. He thanks the Universit\'e Paris Saclay
for excellent working conditions, where part of this work was done.

We thank \href{https://www.math.u-psud.fr/~pankrashkin/}{Konstantin Pankrashkin}
and
\href{https://www-m7.ma.tum.de/bin/view/Analysis/SimoneWarzel}{Simone Warzel}
for useful discussions.

\providecommand{\bysame}{\leavevmode\hbox to3em{\hrulefill}\thinspace}
\providecommand{\MR}{\relax\ifhmode\unskip\space\fi MR }
\providecommand{\MRhref}[2]{%
  \href{http://www.ams.org/mathscinet-getitem?mr=#1}{#2}
}
\providecommand{\href}[2]{#2}

\end{document}